\documentclass{article}
\usepackage{amsthm}
\usepackage{amsmath}
\usepackage{amssymb}

\theoremstyle{definition}
\newtheorem{definition}{Definition}[section]

\newtheorem{example}[definition]{Example}
\newtheorem{openproblem}[definition]{Open problem}

\newtheorem{remark}[definition]{Remark}

\theoremstyle{plain}
\newtheorem{theorem}[definition]{Theorem}
\newtheorem{lemma}[definition]{Lemma}
\newtheorem{proposition}[definition]{Proposition}
\newtheorem{corollary}[definition]{Corollary}

\newcommand{\Lat}{\mathbb{L}}
\newcommand{\BA}{\mathbb{BA}}
\newcommand{\BI}{\mathbb{BI}}
\newcommand{\KA}{\mathbb{KA}}
\newcommand{\KL}{\mathbb{PKA}}
\newcommand{\POML}{\mathbb{POML}}

\newcommand{\BZ}{\mathbb{BZL}}
\newcommand{\PBZs}{\mathbb{PBZL}^{\ast }}
\newcommand{\OL}{\mathbb{OL}}
\newcommand{\OML}{\mathbb{OML}}
\newcommand{\MOL}{\mathbb{MOL}}
\newcommand{\AOL}{\mathbb{AOL}}
\newcommand{\MOD}{\mathbb{MOD}}

\newcommand{\SAOL}{\mathbb{SAOL}}
\newcommand{\DIST}{\mathbb{DIST}}

\newcommand{\SDM}{\mathbb{SDM}}
\newcommand{\WSDM}{\mathbb{WSDM}}
\newcommand{\PBZ}{PBZ$^{\ast }$}
\newcommand{\modular}{\mbox{\textup{MOD}}}
\newcommand{\dist}{\mbox{\textup{DIST}}}
\newcommand{\sk}{\mbox{\textup{SK}}}
\newcommand{\sdm}{\mbox{\textup{SDM}}}
\newcommand{\wsdm}{\mbox{\textup{WSDM}}}
\newcommand{\jo}{\mbox{\textup{J0}}}
\newcommand{\ji}{\mbox{\textup{J1}}}
\newcommand{\cn}{\mbox{\textup{FxPt}}}

\def\T{{\mathbb T}}
\def\O{\mbox{\textup{D2OML}}}

\def\R{\mbox{\textup{D2OL$\wedge$}}}
\def\RV{\mbox{\textup{D2OL$\vee$}}}
\def\N{{\mathbb N}}

\def\K{{\mathbb K}}
\def\I{{\rm I}}
\def\o{{\rm O}}
\def\s{{\mathbb S}}
\def\G{{\mathbb G}}
\def\U{{\mathbb U}}
\def\V{{\mathbb V}}
\def\W{{\mathbb W}}
\def\X{{\mathbb X}}
\def\Y{{\mathbb Y}}
\def\C{{\mathbb C}}
\def\D{{\mathbb D}}
\def\H{{\mathrm H}}
\def\S{{\mathrm S}}
\def\P{{\mathrm P}}
\def\1{\textcircled{1}}
\def\2{\textcircled{2}}

\numberwithin{equation}{section}

\begin{document}
\title{\vspace*{-70pt}Subreducts and Subvarieties of \PBZ --lattices}

\author{Claudia Mure\c san\\ {\small \bf cmuresan@fmi.unibuc.ro, claudia.muresan@g.unibuc.ro, c.muresan@yahoo.com}\\ {\small \sc University of Bucharest, Faculty of Mathematics and Computer Science}}

\date{\today }

\maketitle

\begin{center}{\em Dedicated to Professor George Georgescu}\end{center}

\begin{abstract} {\em \PBZ --lattices} are bounded lattice--ordered structures endowed with two complements, called {\em Kleene} and {\em Brouwer}; by definition, they are the paraorthomodular Brouwer--Zadeh lattices in which the pairs of elements with their Kleene complements satisfy the Strong De Morgan condition. These algebras arise in the study of Quantum Logics and they form a variety $\PBZs $ which includes orthomodular lattices with an extended signature (with the two complements coinciding), as well as antiortholattices (whose Brouwer complements are trivial). The former turn out to have directly irreducible lattice reducts and, under distributivity, no nontrivial elements with bounded lattice complements. We establish a lattice isomorphism between the lattice of subvarieties of the variety $\SAOL $ generated by the antiortholattices with the Strong De Morgan property and the ordinal sum of the three--element chain with the lattice of subvarieties of the variety $\KL $ of pseudo--Kleene algebras, which also gives us axiomatizations for all subvarieties of $\SAOL $ from those of the subvarieties of $\KL $ and proves that the variety $\KL $ is generated by the class of the bounded involution lattice reducts of the members of $\SAOL $ and thus of those of any subvariety of $\PBZs $ that includes $\SAOL $, hence neither of these classes is a variety. We also obtain an infinity of pairwise disjoint infinite ascending chains of varieties of \PBZ --lattices, out of which one is formed of subvarieties of $\SAOL $ and another one from subvarieties of the variety of distributive \PBZ --lattices.
 

\end{abstract}

\section{Introduction}

{\em \PBZ --lattices} are bounded lattices endowed with two unary operations: an involution $\cdot ^{\prime }$, called {\em Kleene complement}, which satisfies the {\em Kleene condition}: $x\wedge x^{\prime }\leq y\vee y^{\prime }$, as well as paraorthomodularity, and the {\em Brouwer complement}, which reverses order, is smaller than the Kleene complement, and satisfies only one of the De Morgan laws, along with condition $(\ast )$, which is a weakening of the other De Morgan law (called {\em Strong De Morgan}), obtained from it by replacing one of the variables with the Kleene complement of the other.

The study of \PBZ --lattices originates in the foundations of quantum mechanics; they have been introduced in \cite{GLP} as abstractions for the sets of effects of complex separable Hilbert spaces endowed with the spectral order and two kinds of complements \cite{catnis,groote,mose}, and, from the previous such abstractions used in the unsharp approach to Quantum Logics, which include {\em effect algebras} \cite{fouben}, {\em quantum MV--algebras} \cite{giu} and {\em Brouwer--Zadeh posets} \cite{catmar}, they present the advantage of forming a variety, which we denote by $\PBZs $, as paraorthomodularity becomes an equational property under the other axioms of \PBZ --lattices.

$\PBZs $ includes the variety $\OML $ of {\em orthomodular lattices} considered with an extended signature, by endowing each orthomodular lattice with a second complement equalling their Kleene complement, and it also includes the proper universal class $\AOL $ of {\em antiortholattices}, which are, by definition, the \PBZ --lattices with all nonzero elements {\em dense}, i.e. having $0$ as their Brouwer complement. Antiortholattices are exactly the \PBZ --lattices in which $0$ and $1$ are the only elements whose Kleene complements are bounded lattice complements. We prove that, in distributive antiortholattices, moreover, $0$ and $1$ are the only elements that have bounded lattice complements; however, this property is not preserved in the non--distributive case. We also prove that the lattice reducts of antiortholattices are directly irreducible and further investigate direct irreducibility, as well as the lengths of two important subposets of the lattice reducts of \PBZ --lattices. In Theorem \ref{notgendist} we obtain an infinite ascending chain of subvarieties of the variety $\DIST $ of distributive \PBZ --lattices, which is included in the variety generated by antiortholattices.

The \PBZ --lattices with the $0$ meet--irreducible are exactly the antiortholattices that satisfy the Strong De Morgan condition; we prove that the variety $\SAOL $ they generate is also generated by the class of the antiortholattices with the $0$ strictly meet--irreducible, that is the ordinal sums of the two--element chain with pseudo--Kleene algebras and again the two--element chain (Theorem \ref{saol}). Our {\bf main theorem} (Theorem \ref{thegenop}) states that the operator that takes subvarieties $\V $ of the variety $\KL $ of pseudo--Kleene algebras to the subvarieties of $\SAOL $ generated by the ordinal sums of the two--element chain with members of $\V $ and again the two--element chain is a lattice isomorphism between the lattice of subvarieties of $\KL $ and the principal filter of the lattice of subvarieties of $\SAOL $ generated by the smallest non--orthomodular subvariety of $\PBZs $, and thus the lattice of subvarieties of $\SAOL $ is isomorphic to the glued sum of the three--element chain with the lattice of subvarieties of $\KL $. Hence the splitting pair in the lattice of subvarieties of $\KL $ formed of the variety of ortholattices and the variety of Kleene algebras is taken by this lattice isomorphism into a splitting pair in the lattice of subvarieties of $\SAOL $. From the axiomatizations of the subvarieties of $\KL $ we construct axiomatizations for the subvarieties $\V $ of $\SAOL $ in two different ways, depending on whether $\V $ belongs to the principal ideal or the principal filter determined by this splitting pair (Theorem \ref{theax}). We also prove that $\KL $ is generated by the class of the bounded involution lattice reducts of the members of $\SAOL $, hence also by the class of those of any variety of \PBZ --lattices that includes it (Corollary \ref{varredsaol}), therefore neither of these classes is a variety (Corollary \ref{birednotvar}), in contrast with the subvarieties of $\OML $. As a reinterpretation and strengthening of \cite[Corollary~$5.2$]{GLP}, we point out that the lattice reducts of antiortholattices satisfying the Strong de Morgan law generate the whole variety $\Lat $ of lattices, and we have further results for subvarieties of $\Lat $. Not all paraorthomodular pseudo--Kleene algebras are bounded involution lattice reducts of \PBZ --lattices, and not even all bounded lattices that are reducts of paraorthomodular pseudo--Kleene algebras are also bounded lattice reducts of \PBZ --lattices (Remark \ref{ppkasnotredpbzl}). The infinite ascending chain of subvarieties of $\DIST $, the well--known one of varieties of modular ortholattices and the image of this latter chain through the lattice isomorphism in Theorem \ref{thegenop} give us an infinity of pairwise disjoint infinite ascending chains of subvarieties of the variety generated by $\OML \cup \AOL $.

\section{Preliminaries}
\label{preliminaries}

In this paper, we are using the notations from \cite{GLP,PBZ2,rgcmfp,pbzsums}, along with conventions such as denoting, for any algebra ${\bf A}$, by $A$ the set reduct of ${\bf A}$; for brevity, if there no danger of confusion on the algebraic structure we are refering to, occasionally algebras will be denoted the same as their set reducts. The trivial (that is singleton) algebra will be considered subdirectly irreducible.

We denote by $\N $ the set of the natural numbers and by $\N ^*=\N \setminus \{0\}$. For any set $M$, $|M|$ will denote the cardinality of $M$, $\mathcal{P}(M)$ will be the set of the subsets of $M$, and ${\rm Part}(M)$ and ${\rm Eq}(M)$ will be the lattices of the partitions and the equivalences on $M$, respectively, where ${\rm Eq}(M)$ is ordered by the set inclusion, while the order $\leq $ of ${\rm Part}(M)$ is given by: for any $\pi ,\rho \in {\rm Part}(M)$, $\pi \leq \rho $ iff every class of $\rho $ is a union of classes from $\pi $, and $eq:{\rm Part}(M)\rightarrow {\rm Eq}(M)$ will be the canonical lattice isomorphism.

For any (bounded) lattice ${\bf L}$, $\prec $ will denote the  cover relation of ${\bf L}$, ${\bf L}^d$ will be the dual of ${\bf L}$ and, if ${\bf L}$ has a $0$, then the set of the atoms of ${\bf L}$ will be denoted by ${\rm At}({\bf L})$. For any $a,b\in L$, we denote by $[a,b]_L=[a)_L\cap (b]_L$ the interval of ${\bf L}$ bounded by $a$ and $b$, as well as any algebraic structure we consider on it. For all $n\in \N ^*$, we denote by ${\bf D}_n$ the $n$--element chain, regardless of the bounded lattice--ordered structure we consider on it.

Recall that, if ${\bf L}$ is a lattice and $x,y\in L$, then $(x,y)$ is called a {\em splitting pair} in ${\bf L}$ iff $y\nleq x$ and $L=(x]_L\cup [y)_L$.

Let $\V $ be a variety of algebras of a similarity type $\tau $, $\C \subseteq \V $, $\W $ a variety of similar algebras with reducts in $\V $, $\D $ a subclass of $\W $ and ${\bf A}$ and ${\bf B}$ members of $\W $. Then ${\rm Di}(\C )$ and ${\rm Si}(\C )$ will denote the class of the members of $\C $ which are directly irreducible and those that are subdirectly irreducible in $\V $, respectively. We will denote by $\T $ the trivial subvariety of $\V $, consisting solely of the trivial algebras from $\V $, that is its singleton members. We will denote by $\I_{\V }(\D )$, $\H _{\V }(\D )$, $\S _{\V }(\D )$ and $\P _{\V }(\D )$ the class of the isomorphic images, homomorphic images, subalgebras and direct products of the $\tau $--reducts of the members of $\D $, respectively, and $V_{\V }(\D )=\H _{\V }\S _{\V }\P _{\V }(\D )$ will denote the subvariety of $\V $ generated by the $\tau $--reducts of the members of $\D $; for any class operator $\o _{\V }$ and any ${\bf M}\in \D $, $\o _{\V }(\{{\bf M}\})$ will be streamlined to $\o _{\V }({\bf M})$. We will abbreviate by ${\bf A}\cong _{\V }{\bf B}$ the fact that the $\tau $--reducts of ${\bf A}$ and ${\bf B}$ are isomorphic. $({\rm Con}_{\V }({\bf A}),\cap ,\vee ,\Delta _A,\nabla _A)$ will be the bounded lattice of the congruences of the $\tau $--reduct of ${\bf A}$, and, for any $n\in \N ^*$ and any constants $\kappa _1,\ldots ,\kappa _n$ from $\tau $, we denote by ${\rm Con}_{\V \kappa _1,\ldots ,\kappa _n}({\bf A})=\{\theta \in {\rm Con}_{\V }({\bf A}):(\forall \, i\in [1,n])\, (\kappa _i^{\bf A}/\theta =\{\kappa _i\})\}$. If $\V $ is the variety of (bounded) lattices, then the index $\V $ will be eliminated from the notations above. $\Lambda (\V )$ will be the lattice of subvarieties of $\V $.

A straightforward consequence of \cite[Corollary $2$, p.$51$]{gralgu} is the fact that ${\rm Con}_{\W }({\bf A})$ is a bounded complete sublattice of ${\rm Con}_{\V }({\bf A})$, while ${\rm Con}_{\V \kappa _1,\ldots ,\kappa _n}({\bf A})$ is a complete sublattice of ${\rm Con}_{\V }({\bf A})$ and thus a bounded lattice \cite[Lemma 2.(iii)]{pbzsums}. An obvious consequence of the first of these two properties and the well--known fact that the $\tau $--reduct of ${\bf A}$ is subdirectly irreducible iff either ${\bf A}$ is trivial or $\Delta _A$ has a unique cover in ${\rm Con}_{\V }({\bf A})$ \cite{gralgu} is the fact that:

\begin{remark} With the notations above, if the $\tau $--reduct of ${\bf A}$ is subdirectly irreducible (in $\V $), then ${\bf A}$ is subdirectly irreducible (in $\W $). Hence, if we denote, for any $\C \subseteq \W $, by $\C _{\V }$ the class of the $\tau $--reducts of the members of $\C $, then,  for any $\C \subseteq \W $, we have ${\rm Si}(\C _{\V })\subseteq {\rm Si}(\C )_{\V }$.\label{sired}\end{remark}

\begin{remark} With the notations in Remark \ref{sired}, we may notice that, for any $\C \subseteq \W $, $\H _{\W }(\C )_{\V }\subseteq \H _{\V }(\C _{\V })$, $\S _{\W }(\C )_{\V }\subseteq \S _{\V }(\C _{\V })$ and $\P _{\W }(\C )_{\V }=\P _{\V }(\C _{\V })$, hence $V _{\W }(\C )_{\V }\subseteq V_{\V }(\C _{\V })$ and, if $\C $ is closed w.r.t. direct products, in particular if $\C $ is a subvariety of $\W $, then $V_{\V }(\C _{\V })=\H _{\V }\S _{\V }(\P _{\W }(\C )_{\V })=\H _{\V }\S _{\V }(\C _{\V })$.\label{varred}\end{remark}

If $t$ and $u$ are $n$--ary terms over $\tau $ for some $n\in \N ^*$ and $A_1,\ldots ,A_n$ are subsets of $A$, then we denote by ${\bf A}\vDash _{A_1,\ldots ,A_n}t(x_1,\ldots ,x_n)\approx u(x_1,\ldots ,x_n)$ the fact that $t^{\bf A}(a_1,\ldots ,a_n)\approx u^{\bf A}(a_1,\ldots ,a_n)$ for all $a_1\in A_1,\ldots ,a_n\in A_n$, where $x_1,\ldots ,x_n$ are the variables in their order of appearance in the equation $t\approx u$.

We denote by ${\rm Mi}({\bf L})$ the set of the meet--irreducible elements of a lattice ${\bf L}$.

Let ${\bf L}$ be a lattice with top element $1^{\bf L}$ and ${\bf M}$ be a lattice with bottom element $0^{\bf M}$. Recall that the {\em ordinal sum} (also called {\em glued sum}) of ${\bf L}$ with ${\bf M}$ is the lattice ${\bf L}\oplus {\bf M}$ (with underlying set $L\oplus M$) obtained by stacking ${\bf M}$ on top of ${\bf L}$ and glueing the top element of ${\bf L}$ together with the bottom element of ${\bf M}$. Note that, for any $\alpha \in {\rm Con}({\bf L})$ and any $\beta \in {\rm Con}({\bf M})$, if we denote by $\alpha \oplus \beta $ the equivalence on $L\oplus M$ whose restrictions to $L$ and $M$ are $\alpha $ and $\beta $, respectively, then $\alpha \oplus \beta \in {\rm Con}({\bf L}\oplus {\bf M})$. Clearly, the operation $\oplus $ on bounded lattices is associative, and so is the operation $\oplus $ on the congruences of such lattices. Note that the map $(\alpha ,\beta )\mapsto \alpha \oplus \beta $ is a lattice isomorphism from ${\rm Con}({\bf L}\times {\bf M})\cong {\rm Con}({\bf L})\times {\rm Con}({\bf M})$ to ${\rm Con}({\bf L}\oplus {\bf M})$.

Now let ${\bf L}$ and ${\bf M}$ be nontrivial bounded lattices. Recall that the horizontal sum of ${\bf L}$ with ${\bf M}$ is the non--trivial bounded lattice ${\bf L}\boxplus {\bf M}$ (with underlying set $L\boxplus M$) obtained by glueing the bottom elements of ${\bf L}$ and ${\bf M}$ together, glueing their top elements together and letting all other elements of ${\bf L}$ be incomparable to every other element of ${\bf M}$. Note that the horizontal sum of nontrivial bounded lattices is commutative and associative, it has ${\bf D}_2$ as a neutral element and it can be generalized to arbitrary families of nontrivial bounded lattices.

See the rigorous definitions of the ordinal and horizontal sums in \cite{pbzsums,eunoucard,eucardbi,euinfbi}.

Now let us take a look at the algebras and varieties studied in the following sections. See \cite{GLP,PBZ2,rgcmfp,pbzsums} for more details on the following notions.

Recall that a \emph{bounded involution lattice} (in brief, \emph{BI--lattice}) is an algebra $\mathbf{L}=(L,\wedge,\vee,\cdot ^{\prime},0,1)$ of type $(2,2,1,0,0)$ such that $(L,\wedge,\vee,0,1)$ is a bounded lattice and $\cdot ^{\prime}:L\rightarrow L$ is an order--reversing operation that satisfies $a^{\prime\prime}=a$ for all $a\in L$. This makes $\cdot ^{\prime}$ a dual lattice automorphism of $\mathbf{L}$, called {\em involution}.

For any (bounded) lattice--ordered algebra ${\bf A}$, we denote by ${\bf A}_l$ the (bounded) lattice reduct of ${\bf A}$. For any algebra ${\bf A}$ having a BI--lattice reduct, we denote that reduct by ${\bf A}_{bi}$. If $\C $ is a class of (bounded) lattice--ordered algebras and $\D $ is a class of algebras having BI--lattice reducts, then we denote by $\C _L=\{{\bf L}_l:{\bf L}\in \C \}$ and by $\D _{BI}=\{{\bf L}_{bi}:{\bf L}\in \D \}$.

A \emph{pseudo--Kleene algebra} is a BI--lattice $\mathbf{L}$ satisfying $a\wedge a^{\prime}\leq b\vee b^{\prime}$ for
all $a,b\in L$. The involution of a pseudo--Kleene algebra is called {\em Kleene complement}. Distributive pseudo--Kleene algebras are called \emph{Kleene algebras} or \emph{Kleene
lattices}.

Let $\mathbf{L}$ be a BI--lattice. Then we denote by $S(\mathbf{L})=\{x\in L:x\vee x^{\prime }=1\}$ and call the elements of $S(\mathbf{L})$ {\em sharp elements} of $\mathbf{L}$. The BI--lattice $\mathbf{L}$ is called an {\em ortholattice} iff all its elements are sharp, and it is called an {\em orthomodular lattice} iff, for all $a,b\in L$, $a\leq b$ implies $b=(b\wedge a^{\prime})\vee a$. By taking $b=1$ in the previous implication, we obtain that any orthomodular lattice is an ortholattice. Note, also, that any modular ortholattice is an orthomodular lattice, and that Boolean algebras are exactly the distributive ortholattices. Clearly, any ortholattice is a pseudo--Kleene algebra.
If ${\bf M}$ is a bounded lattice, ${\bf K}$ is a BI--lattice and $f$ is a dual lattice isomorphism from ${\bf M}$ to ${\bf M}^d$, then the ordinal sum ${\bf M}\oplus {\bf K}_l\oplus {\bf M}^d$ becomes a BI--lattice denoted ${\bf M}\oplus {\bf K}\oplus {\bf M}^d$ when endowed with the involution $\cdot ^{\prime {\bf M}\oplus {\bf K}\oplus {\bf M}^d}$ defined by: $\cdot ^{\prime {\bf M}\oplus {\bf K}\oplus {\bf M}^d}\mid _M=f$, $\cdot ^{\prime {\bf M}\oplus {\bf K}\oplus {\bf M}^d}\mid _K=\cdot ^{\prime {\bf K}}$ and $\cdot ^{\prime {\bf M}\oplus {\bf K}\oplus {\bf M}^d}\mid _{M^d}=f^{-1}$. The BI--lattice ${\bf M}\oplus {\bf D}_1\oplus {\bf M}^d$ will be denoted ${\bf M}\oplus {\bf M}^d$, just as its bounded lattice reduct. Furthermore, ${\bf K}$ is a pseudo--Kleene algebra iff ${\bf M}\oplus {\bf K}\oplus {\bf M}^d$ is a pseudo--Kleene algebra.

If ${\bf A}$ and ${\bf B}$ are nontrivial BI--lattices, then the horizontal sum ${\bf A}_l\boxplus {\bf B}_l$ becomes a BI--lattice ${\bf A}\boxplus {\bf B}$ when endowed with the involution $\cdot ^{\prime {\bf A}\boxplus {\bf B}}$ defined by: $\cdot ^{\prime{\bf A}\boxplus {\bf B}}\mid _A=\cdot ^{\prime {\bf A}}$ and $\cdot ^{\prime {\bf A}\boxplus {\bf B}}\mid _B=\cdot ^{\prime {\bf B}}$, which makes the BI--lattices ${\bf A}$ and ${\bf B}$ subalgebras of ${\bf A}\boxplus {\bf B}$. Clearly, ${\bf A}\boxplus {\bf B}$ is a pseudo--Kleene algebra iff ${\bf A}$ and ${\bf B}$ are pseudo--Kleene algebras and at least one of them is an ortholattice.

A BI--lattice $\mathbf{L}$ is said to be \emph{paraorthomodular} iff, for all $a,b\in L$, if $a\leq b$ and $a^{\prime}\wedge b=0$, then $a=b$. Algebras with BI--lattice reducts will be said to be orthomodular, respectively paraorthomodular iff their BI--lattice reducts are such. Note that any orthomodular lattice is a paraorthomodular BI--lattice and any paraorthomodular ortholattice is orthomodular. However, there are paraorthomodular pseudo--Kleene algebras that are not orthomodular, for instance the diamond ${\bf M}_3={\bf D}_2^2\boxplus {\bf D}_3$ as a horizontal sum of BI--lattices, specifically of the Boolean algebra ${\bf D}_2^2$ and the Kleene chain ${\bf D}_3$, which is clearly not an ortholattice. Let us also note that, for instance, the horizontal sum of BI--chains ${\bf N}_5={\bf D}_3\boxplus {\bf D}_4$ is not a pseudo--Kleene lattice. The smallest ortholattice which is not orthomodular is the {\em Benzene ring} ${\bf B}_6$, with the Kleene complement defined as in the following Hasse diagram, which makes it non--isomorphic with the horizontal sum of BI--chains ${\bf D}_4\boxplus {\bf D}_4$, while its lattice reduct is isomorphic to the horizontal sum of bounded chains ${\bf D}_4\boxplus {\bf D}_4$. An example of a non--modular orthomodular lattice is ${\bf D}_2^2\boxplus {\bf D}_2^3$. The smallest modular ortholattice which is not a Boolean algebra is ${\bf MO}_2={\bf D}_2^2\boxplus {\bf D}_2^2$. Let us also recall, for an arbitrary nonzero cardinality $\kappa $, the notation ${\bf MO}_{\kappa }=\boxplus _{0\leq \iota <\kappa }{\bf D}_2^2$ for the modular ortholattice of length three with $2\kappa $ atoms.\vspace*{3pt}

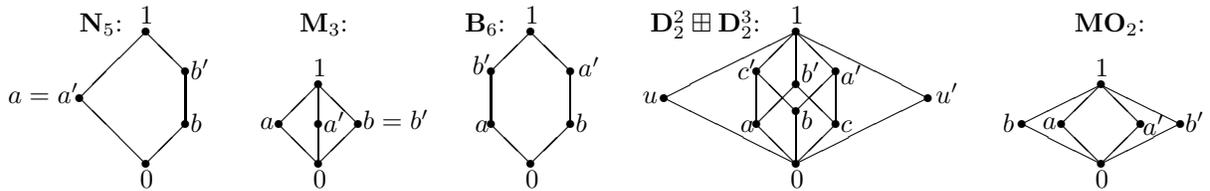
\begin{center}\begin{tabular}{ccccc}
\hspace*{14pt}\begin{picture}(30,53)(0,0)
\put(-10,50){${\bf N}_5$:}
\put(15,0){\circle*{3}}
\put(30,15){\circle*{3}}
\put(15,0){\line(1,1){15}}
\put(15,0){\line(-1,1){25}}
\put(30,15){\line(0,1){20}}
\put(15,50){\circle*{3}}
\put(30,35){\circle*{3}}
\put(-10,25){\circle*{3}}
\put(15,50){\line(1,-1){15}}
\put(15,50){\line(-1,-1){25}}
\put(13,-9){$0$}
\put(13,53){$1$}
\put(32,12){$b$}
\put(32,32){$b^{\prime }$}
\put(-37,23){$a=a^{\prime }$}
\end{picture} & \hspace*{2pt}
\begin{picture}(30,53)(0,0)
\put(10,50){${\bf M}_3$:}
\put(15,7){\circle*{3}}
\put(15,37){\circle*{3}}
\put(0,22){\circle*{3}}
\put(30,22){\circle*{3}}
\put(15,7){\line(1,1){15}}
\put(15,7){\line(-1,1){15}}
\put(15,37){\line(1,-1){15}}
\put(15,37){\line(-1,-1){15}}
\put(15,7){\line(0,1){30}}
\put(15,22){\circle*{3}}
\put(13,-2){$0$}
\put(13,40){$1$}
\put(-7,20){$a$}
\put(17,19){$a^{\prime }$}
\put(32,20){$b=b^{\prime }$}
\end{picture} &\hspace*{18pt}
\begin{picture}(30,53)(0,0)
\put(-10,50){${\bf B}_6$:}
\put(15,0){\circle*{3}}
\put(0,15){\circle*{3}}
\put(30,15){\circle*{3}}
\put(15,0){\line(1,1){15}}
\put(15,0){\line(-1,1){15}}
\put(0,15){\line(0,1){20}}
\put(30,15){\line(0,1){20}}
\put(15,50){\circle*{3}}
\put(0,35){\circle*{3}}
\put(30,35){\circle*{3}}
\put(15,50){\line(1,-1){15}}
\put(15,50){\line(-1,-1){15}}
\put(13,-9){$0$}
\put(13,53){$1$}
\put(-6,12){$a$}
\put(32,12){$b$}
\put(-7,34){$b^{\prime }$}
\put(33,33){$a^{\prime }$}
\end{picture} & \hspace*{35pt}
\begin{picture}(30,53)(0,0)
\put(-40,50){${\bf D}_2^2\boxplus {\bf D}_2^3$:}
\put(15,0){\circle*{3}}
\put(15,30){\circle*{3}}
\put(0,15){\circle*{3}}
\put(30,15){\circle*{3}}
\put(-35,25){\circle*{3}}
\put(65,25){\circle*{3}}
\put(-43,23){$u$}
\put(68,23){$u^{\prime }$}
\put(15,0){\line(1,1){15}}
\put(15,0){\line(-1,1){15}}
\put(15,30){\line(1,-1){15}}
\put(15,30){\line(-1,-1){15}}
\put(15,0){\line(2,1){50}}
\put(15,0){\line(-2,1){50}}
\put(15,50){\line(2,-1){50}}
\put(15,50){\line(-2,-1){50}}
\put(15,0){\line(0,1){20}}
\put(0,15){\line(0,1){20}}
\put(30,15){\line(0,1){20}}
\put(15,30){\line(0,1){20}}
\put(15,20){\circle*{3}}
\put(15,50){\circle*{3}}
\put(0,35){\circle*{3}}
\put(30,35){\circle*{3}}
\put(15,20){\line(1,1){15}}
\put(15,20){\line(-1,1){15}}
\put(15,50){\line(1,-1){15}}
\put(15,50){\line(-1,-1){15}}
\put(13,-9){$0$}
\put(13,53){$1$}
\put(-6,12){$a$}
\put(17,13){$b$}
\put(32,12){$c$}
\put(17,31){$b^{\prime }$}
\put(-7,32){$c^{\prime }$}
\put(32,30){$a^{\prime }$}
\end{picture} &\hspace*{40pt}
\begin{picture}(30,53)(0,0)
\put(5,50){${\bf MO}_2$:}
\put(15,0){\circle*{3}}
\put(15,30){\circle*{3}}
\put(0,15){\circle*{3}}
\put(30,15){\circle*{3}}
\put(-15,15){\circle*{3}}
\put(45,15){\circle*{3}}
\put(15,0){\line(1,1){15}}
\put(15,0){\line(-1,1){15}}
\put(15,30){\line(1,-1){15}}
\put(15,30){\line(-1,-1){15}}
\put(15,0){\line(2,1){30}}
\put(15,0){\line(-2,1){30}}
\put(15,30){\line(2,-1){30}}
\put(15,30){\line(-2,-1){30}}
\put(13,-9){$0$}
\put(13,33){$1$}
\put(-7,13){$a$}
\put(31,11){$a^{\prime }$}
\put(-22,12){$b$}
\put(47,12){$b^{\prime }$}
\end{picture}\end{tabular}\end{center}\vspace*{3pt}

A \emph{Brouwer--Zadeh lattice} (in brief, \emph{BZ--lattice})
is an algebra $\mathbf{L}=(L,\wedge,\vee,\cdot ^{\prime},\cdot ^{\sim },0,1)$ of type $(2,2,1,1,0,0)$ such that $(L,\wedge,\vee,\cdot ^{\prime},0,1)$ is a pseudo--Kleene algebra
and $\cdot ^{\sim }:L\rightarrow L$ is an order--reversing operation, called {\em Brouwer complement}, that satisfies: $a\wedge a^{\sim}=0$ and $a\leq a^{\sim\sim}=a^{\sim\prime}$ for all $a\in L$. Note that, in any BZ--lattice ${\bf L}$, we have, for all $a,b\in L$: $a^{\sim\sim\sim}=a^{\sim}\leq a^{\prime}$, $(a\vee b)^{\sim}=a^{\sim}\wedge b^{\sim}$ and $(a\wedge b)^{\sim}\geq a^{\sim}\vee b^{\sim}$.

We denote in the following way the modularity and the distributivity laws for lattices and also consider the following equations over the type of BZ--lattices, out of which \sdm\ clearly implies $(\ast )$ and \wsdm, while \jo\ implies \ji:\begin{flushleft}\begin{tabular}{rl}
\modular & $x\vee (y\wedge (x\vee z))\approx (x\vee y)\wedge (x\vee z)$\\ 
\dist & $x\vee (y\wedge z)\approx (x\vee y)\wedge (x\vee z)$\\ 
$(\ast )$ & $(x\wedge x^{\prime })^{\sim }\approx x^{\sim }\vee x^{\prime \sim }$\\ 
\sdm\ {\em (Strong De Morgan)} & $(x\wedge y)^{\sim }\approx x^{\sim }\vee y^{\sim }$\\ 
\wsdm\ {\em (weak \sdm )} & $(x\wedge y^{\sim })^{\sim }\approx x^{\sim }\vee y^{\sim \sim }$\\ 
\sk & $x\wedge y^{\sim \sim }\leq x^{\prime \sim }\vee y$\\  
\jo & $(x\wedge y^{\sim })\vee (x\wedge y^{\sim \sim })\approx x$\\ 
\ji & $(x\wedge (x\wedge y)^{\sim })\vee (x\wedge(x\wedge y)^{\sim \sim })\approx x$\end{tabular}\end{flushleft}

A \emph{\PBZ --lattice} is a paraorthomodular BZ--lattice that satisfies condition $(\ast )$. \PBZ --lattices form a variety. In any \PBZ --lattice $\mathbf{L}$, $S(\mathbf{L})=\{a\in L:a^{\prime }=a^{\sim }\}=\{a^{\sim }:a\in L\}$ and $S(\mathbf{L})$ is the universe of the largest orthomodular subalgebra of $\mathbf{L}$, so that $\mathbf{L}$ is orthomodular iff $S(\mathbf{L})=L$ iff $\mathbf{L}\vDash x^{\prime }\approx x^{\sim }$.

We denote by $\Lat $, $\BA $, $\MOL $, $\OML $, $\OL $, $\KA $, $\KL $ and $\BI $ the varieties of lattices, Boolean algebras, modular ortholattices, orthomodular lattices, ortholattices, Kleene algebras, pseudo--Kleene algebras and BI--lattices, respectively, and by $\POML $ the quasivariety of the paraorthomodular pseudo--Kleene algebras. We denote by $\BZ $ and $\PBZs $ the varieties of BZ--lattices and \PBZ --lattices, respectively, and by $\DIST $, $\MOD $, $\SDM $ and $\WSDM $ the varieties of the \PBZ --lattices that satisfy $\dist $, $\modular $, $\sdm $ and $\wsdm $, respectively. 

By the above, $\OML $ can be identified with the subvariety $\{\mathbf{L}\in \PBZs :\mathbf{L}\vDash x^{\prime }\approx x^{\sim }\}$ of $\PBZs $, by endowing each orthomodular lattice, in particular every Boolean algebra, with a Brouwer complement equalling its Kleene complement. In the same way, we can identify $\OL $ with the subvariety $\{\mathbf{L}\in \BZ :\mathbf{L}\vDash x^{\prime }\approx x^{\sim }\}$ of $\BZ $. With this extended signature, $\BA \subset \MOL \subset \OML \subset \OL \subset \SDM \subset \WSDM $.

\begin{remark} By Remark \ref{sired}, for any $\C \subseteq \BZ $, ${\rm Si}(\C _{BI})\subseteq {\rm Si}(\C )_{BI}$.\label{sibired}\end{remark}

If ${\bf A}$ and ${\bf B}$ are nontrivial BZ--lattices, then, exactly when at least one of ${\bf A}$ and ${\bf B}$ is an ortholattice, the horizontal sum ${\bf A}_{bi}\boxplus {\bf B}_{bi}$ becomes a BZ--lattice ${\bf A}\boxplus {\bf B}$ when endowed with the Brouwer complement $\cdot ^{\sim {\bf A}\boxplus {\bf B}}$ defined by: $\cdot ^{\sim {\bf A}\boxplus {\bf B}}\mid _A=\cdot ^{\sim {\bf A}}$ and $\cdot ^{\sim {\bf A}\boxplus {\bf B}}\mid _B=\cdot ^{\sim {\bf B}}$, which makes the BZ--lattices ${\bf A}$ and ${\bf B}$ subalgebras of ${\bf A}\boxplus {\bf B}$. From this, by enforcing paraorthomodularity and condition $(\ast )$, we obtain that ${\bf A}\boxplus {\bf B}$ is a \PBZ --lattice exactly when ${\bf A}$ and ${\bf B}$ are \PBZ --lattices and at least one of them is orthomodular.

An \emph{antiortholattice} is a \PBZ --lattice $\mathbf{L}$ with $S(\mathbf{L})=\{0,1\}$, or, equivalently, a \PBZ --lattice $\mathbf{L}$ whose Brouwer complement is {\em trivial}, that is $a^{\sim }=0$ for all $a\in L\setminus \{0\}$ (and, of course, $0^{\sim }=1$, as in every BZ--lattice). Note that any pseudo--Kleene algebra $\mathbf{L}$ with $S(\mathbf{L})=\{0,1\}$ (which implies paraorthomodularity), in particular any pseudo--Kleene algebra with the $0$ meet--irreducible, in particular any Kleene chain, becomes an antiortholattice when endowed with the trivial Brouwer complement. Moreover, clearly, in any BZ--lattice $\mathbf{L}$ with the $0$ meet--irreducible (which implies $(\ast )$), the Brouwer complement is trivial, so $\mathbf{L}$ is an antiortholattice. Furthermore, if ${\bf M}$ is a nontrivial bounded lattice and ${\bf K}$ is a pseudo--Kleene algebra, then the BI--lattice ${\bf M}\oplus {\bf K}\oplus {\bf M}^d$, endowed with the trivial Brouwer complement, becomes an antiortholattice, that we denote by ${\bf M}\oplus {\bf K}\oplus {\bf M}^d$, as well; the antiortholattice ${\bf M}\oplus {\bf D}_1\oplus {\bf M}^d$ will be denoted ${\bf M}\oplus {\bf M}^d$, as its bounded lattice reduct. We denote by $\AOL $ the proper universal class of antiortholattices (see Section \ref{aolmoddist} below). We have $V_{\BZ }(\AOL )\subset \WSDM $ and we denote by $\SAOL =\SDM \cap V_{\BZ }(\AOL )$.

\begin{remark}{\rm \cite{rgcmfp}} A relative axiomatization w.r.t. $\PBZs $ for $V_{\BZ }(\AOL )$, respectively $\OML \vee V_{\BZ }(\AOL )$, is given by $\jo $, respectively $\{\ji ,\wsdm \}$, hence one for $\SAOL $, respectively $\OML \vee \SAOL $, is given by $\{\jo ,\sdm \}$, respectively $\{\ji ,\sdm \}$.\label{j0vaol}\end{remark}

Let ${\bf M}$ be a bounded lattice and $\C ,\D $ be classes of bounded, BI or BZ--lattices. Then we denote:\linebreak $\begin{cases}
{\bf M}\oplus \C \oplus {\bf M}^d=\{{\bf M}\oplus {\bf K}\oplus {\bf M}^d:{\bf K}\in \C \};\\ 
\C \boxplus \D =\T \cup \{{\bf A}\boxplus {\bf B}:{\bf A}\in \C \setminus \T ,{\bf B}\in \D \setminus \T \}.\end{cases}$By the above, if $\C \subseteq \BI $, then ${\bf M}\oplus \C \oplus {\bf M}^d\subseteq \BI $, and, if ${\bf M}$ is non--trivial and $\C \subseteq \KL $, then ${\bf M}\oplus \C \oplus {\bf M}^d\subset \AOL $. If $\C ,\D \subseteq \BI $, then $\C \boxplus \D \subseteq \BI $; if $\C \subseteq \OL $ and $\D \subseteq \KL $, then $\C \boxplus \D \subseteq \KL $; if $\C \subseteq \OL $ and $\D \subseteq \BZ $, then $\C \boxplus \D \subseteq \BZ $; if $\C \subseteq \OML $ and $\D \subseteq \PBZs $, then $\C \boxplus \D \subseteq \PBZs $. By \cite{pbzsums}, $V_{\BZ }(\OML \boxplus \AOL )$ and $V_{\BZ }(\OML \boxplus V_{\BZ }(\AOL ))$ are distinct proper subvarieties of $\PBZs $ that satisfy $\ji $, fail $\jo$ and $\wsdm $ and intersect $\WSDM $ at $\OML \vee V_{\BZ }(\AOL )$.

\section{The Lattice Structures of Antiortholattices and the Varieties of Modular and of Distributive \PBZ --lattices}
\label{aolmoddist}

Recall from \cite{PBZ2} that antiortholattices are directly irreducible, which immediately follows from the fact that they have no nontrivial sharp elements, and from \cite{rgcmfp} that, moreover, the class of the directly irreducible members of $V_{\BZ }(\AOL )$ is $\AOL $. Now let us see that even the lattice reducts of antiortholattices are directly irreducible, and investigate some related properties.

\begin{remark} $\bullet \ $ Any BI--lattice with nontrivial sharp elements, endowed with the trivial Brouwer complement, fails condition $(\ast )$. Indeed, let $\mathbf{L}$ be a BI--lattice having an $a\in S(\mathbf{L})\setminus \{0,1\}$, so that $a^{\prime }\neq 0$, as well. Then, if $\cdot ^{\sim }:L\rightarrow L$ is the trivial Brouwer complement, we have: $a^{\sim }\vee a^{\prime \sim }=0\vee 0=0\neq 1=0^{\sim }=(a\wedge a^{\prime })^{\sim }$, so $\mathbf{L}$, endowed with $\cdot ^{\sim }$, fails condition $(\ast )$.

$\bullet \ $ Any direct product of two nontrivial BI-lattices has the nontrivial sharp elements $(0,1)$ and $(1,0)$ since $(0,1)^{\prime }=(0^{\prime },1^{\prime })=(1,0)$, hence it can not be the BI--lattice reduct of an antiortholattice, and, moreover, by the above, it fails condition $(\ast)$ when endowed with the trivial Brouwer complement.\label{nontrivsharpntrivsim}\end{remark}

Remark \ref{nontrivsharpntrivsim} ensures us that the BI--lattice reduct of any antiortholattice is directly irreducible. Note that a BI--lattice $\mathbf{L}$ can be directly irreducible while its lattice reduct $\mathbf{L}_l$ is directly reducible; indeed, the BI--lattice ${\bf D}_3\boxplus {\bf D}_3$, in which the incomparable elements equal their involutions, is directly irreducible, but its lattice reduct is isomorphic to ${\bf D}_2^2$. However, it turns out that a pseudo--Kleene algebra with a directly reducible lattice reduct always has nontrivial sharp elements, leading to the stronger statement in the next proposition.

\begin{lemma} If $\mathbf{A}$ and $\mathbf{B}$ are bounded lattices and ${\bf L}$ is a pseudo--Kleene algebra such that $\mathbf{L}_l=\mathbf{A}\times \mathbf{B}$, then $(0^{\bf A},1^{\bf B})^{\prime \mathbf{L}}=(1^{\bf A},0^{\bf B})$.\label{klprod}\end{lemma}

\begin{proof} For brevity, we drop the superscripts. Let $(0,1)^{\prime }=(a,b)\in L=A\times B$ and $(1,0)^{\prime }=(c,d)\in L=A\times B$. Since $\mathbf{L}\in \KL $, we have $(0,b)=(0,1)\wedge (a,b)\leq (1,0)\vee (c,d)=(1,d)$ and $(a,1)=(0,1)\vee (a,b)\geq (1,0)\wedge (c,d)=(c,0)$, so that $b\leq d$ in $\mathbf{B}$ and $a\geq c$ in $\mathbf{A}$. Hence $(a,d)=(a,b)\vee (c,d)=(0,1)^{\prime }\vee (1,0)^{\prime }=((0,1)\wedge (1,0))^{\prime }=(0,0)^{\prime }=(1,1)$ and $(c,b)=(a,b)\wedge (c,d)=(0,1)^{\prime }\wedge (1,0)^{\prime }=((0,1)\vee (1,0))^{\prime }=(1,1)^{\prime }=(0,0)$, thus $c=0$ and $a=1$ in $\mathbf{A}$, while $b=0$ and $d=1$ in $\mathbf{B}$. Therefore $(0,1)^{\prime }=(a,b)=(1,0)$.\end{proof}

\begin{proposition} The lattice reduct of any antiortholattice is directly irreducible.\label{aoldirirred}
\end{proposition}

\begin{proof}  Let $\mathbf{L}\in \AOL $ and assume by absurdum that $\mathbf{L}_{l}=\mathbf{A}\times\mathbf{B}$ for some nontrivial bounded lattices $\mathbf{A}$ and $\mathbf{B}$. Then, by Lemma \ref{klprod}, we have $(0,1)^{\prime }=(1,0)$ in $\mathbf{L}$, hence $(0,1)\in S(\mathbf{L})$, which contradicts the fact that $\mathbf{L}$ is an antiortholattice, since $(0,1)\notin \{(0,0),(1,1)\}$.\end{proof} 

We have used above the fact that, since an antiortholattice has no nontrivial sharp elements, the only elements of an antiortholattice whose Kleene complements are bounded lattice complements are $0$ and $1$. In distributive antiortholattices, moreover, we have no nontrivial complemented elements:

\begin{proposition} The only complemented elements of the lattice reduct of a distributive antiortholattice are $0$ and $1$.\end{proposition}

\begin{proof}  Let $\mathbf{L}$ be a distributive antiortholattice and assume by absurdum that, for some $a,b\in L\setminus \{0,1\}$, $a\vee b=1$ and $a\wedge b=0$, so that $a^{\prime }\wedge b^{\prime }=(a\vee b)^{\prime }=1^{\prime }=0$. Since $\mathbf{L}_{bi}\in \KL $, we have $b\wedge b^{\prime }\leq a\vee a^{\prime }$, hence $b\wedge b^{\prime }=(a\vee a^{\prime })\wedge b\wedge b^{\prime }=(a\wedge b\wedge b^{\prime })\vee (a^{\prime }\wedge b\wedge b^{\prime })=0\vee 0=0$, thus $b\in S(\mathbf{L})$, which contradicts the fact that $\mathbf{L}$ is an antiortholattice.\end{proof} 

\begin{example} Here is a modular non--distributive antiortholattice with other complemented elements beside $0$ and $1$, namely, in the following Hasse diagram, $a$ and $a^{\prime }$ are bounded lattice complements of both $b$ and $b^{\prime }$:

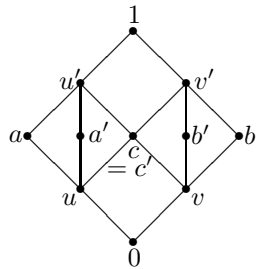
\begin{center}\begin{picture}(40,82)(0,0)
\put(20,0){\circle*{3}}
\put(20,80){\circle*{3}}
\put(20,40){\circle*{3}}
\put(0,20){\circle*{3}}
\put(-20,40){\circle*{3}}
\put(0,60){\circle*{3}}
\put(60,40){\circle*{3}}
\put(0,40){\circle*{3}}
\put(40,40){\circle*{3}}
\put(40,20){\circle*{3}}
\put(40,60){\circle*{3}}
\put(18,-9){$0$}
\put(18,83){$1$}
\put(18,32){$c$}
\put(10,25){$=c^{\prime }$}
\put(-27,38){$a$}
\put(-8,58){$u^{\prime }$}
\put(62,37){$b$}
\put(43,58){$v^{\prime}$}
\put(42,14){$v$}
\put(-7,14){$u$}
\put(3,38){$a^{\prime}$}
\put(42,37){$b^{\prime}$}
\put(20,0){\line(-1,1){40}}
\put(20,0){\line(1,1){40}}
\put(20,80){\line(-1,-1){40}}
\put(20,80){\line(1,-1){40}}
\put(0,20){\line(0,1){40}}
\put(40,20){\line(0,1){40}}
\put(0,20){\line(1,1){40}}
\put(40,20){\line(-1,1){40}}
\end{picture}\end{center}

As we have noticed in \cite{rgcmfp} and recalled in Section \ref{preliminaries}, any pseudo--Kleene algebra with no nontrivial sharp elements is paraorthomodular and satisfies condition $(\ast )$ when endowed with the trivial Brouwer complement, hence it becomes an antiortholattice, since, clearly, any pseudo--Kleene algebra, endowed with the trivial Brouwer complement, becomes a BZ--lattice. Thus the Hasse diagram above represents, indeed, the BI--lattice reduct of an antiortholattice, thus a member of $\AOL \cap \MOD $.\label{aolnontrivblatcompl}\end{example}

\begin{lemma} If $\mathbf{L}$, $\mathbf{A}$ and $\mathbf{B}$ are bounded lattices such that $\mathbf{L}=\mathbf{A}\boxplus\mathbf{B}$, $|A|\geq 2$, $|B|\geq 2$ and $|L|\geq 5$, then $\mathbf{L}$ is directly irreducible.\label{latdirirred}\end{lemma}

\begin{proof}  Assume by absurdum that $\mathbf{L}=\mathbf{K}\times\mathbf{M}$ for some nontrivial bounded lattices $\mathbf{K}$ and $\mathbf{M}$. Since $|L|>4$, we have $|K|>2$ or $|M|>2$. Assume, for instance, that there exists a $u\in K\setminus\{0^{\mathbf{K}},1^{\mathbf{K}}\}$, so that
$(u,1^{\mathbf{M}})\notin \{0^{\mathbf{L}},(0^{\mathbf{K}},1^{\mathbf{M}}),1^{\mathbf{L}}\}$ and $(u,0^{\mathbf{M}})\notin \{0^{\mathbf{L}},(1^{\mathbf{K}},0^{\mathbf{M}}),1^{\mathbf{L}}\}$.

Assume, for instance, that $(u,1^{\mathbf{M}})\in A\setminus\{0^{\mathbf{L}},1^{\mathbf{L}}\}$, so that $(u,1^{\mathbf{M}})\vee b=1^{\mathbf{L}}$ for every $b\in B\setminus\{0^{\mathbf{L}},1^{\mathbf{L}}\}$. Since $(u,1^{\mathbf{M}})\vee(u,0^{\mathbf{M}})=(u,1^{\mathbf{M}})\neq 1^{\mathbf{L}}$, it follows that $(u,0^{\mathbf{M}})\notin B\setminus\{0^{\mathbf{L}},1^{\mathbf{L}}\}$, hence $(u,0^{\mathbf{M}})\in A\setminus\{0^{\mathbf{L}},1^{\mathbf{L}}\}$. Now let $(v,w)\in B\setminus\{0^{\mathbf{L}},1^{\mathbf{L}}\}$. Then $(u\wedge v,w)=(u,1^{\mathbf{M}})\wedge(v,w)=0^{\mathbf{L}}$ and $(u\vee v,w)=(u,0^{\mathbf{M}})\vee(v,w)=1^{\mathbf{L}}$, thus $0^{\mathbf{M}}=w=1^{\mathbf{M}}$, which contradicts the fact that $\mathbf{M}$ is nontrivial. Hence $\mathbf{L}$ is directly irreducible.\end{proof} 

\begin{proposition}\begin{itemize}
\item If $\mathbf{L}\in(\mathbb{OML}\boxplus\mathbb{AOL})\setminus\mathbb{OML}$, then $\mathbf{L}_{l}$ is directly irreducible, thus $\mathbf{L}$ is directly irreducible.

\item If $\mathbf{L}\in(\mathbb{OML}\boxplus V_{\BZ }(\mathbb{AOL}))\setminus(\mathbb{OML}\cup V_{\BZ }(\mathbb{AOL}))$, then $\mathbf{L}_{l}$ is directly irreducible, thus $\mathbf{L}$ is directly irreducible.\end{itemize}\label{gdirirred}\end{proposition}

\begin{proof}  For any $\emptyset \neq \C \subseteq \PBZs $, if $\mathbf{L}\in(\mathbb{OML}\boxplus \C )\setminus(\mathbb{OML}\cup \C )$, then $\mathbf{L}=\mathbf{A}\boxplus \mathbf{B}$ for some $\mathbf{A}\in \OML \setminus \{\mathbf{D}_1,\mathbf{D}_2\}$ and some $\mathbf{B}\in \C \setminus \{\mathbf{D}_1,\mathbf{D}_2\}$, so that $\mathbf{L}_{l}$ is directly irreducible by Lemma \ref{latdirirred}, thus $\mathbf{L}$ is directly irreducible.\end{proof} 

Note from \cite{rgcmfp,pbzsums} that all members of $(\OML \vee V_{\BZ }(\AOL ))\setminus (\OML \cup \AOL )$ are directly reducible and all members of $V_{\BZ }(\OML \boxplus \AOL )\setminus (\OML \boxplus \AOL )$ are directly reducible. Hence:

\begin{corollary}\begin{itemize} 
\item ${\rm Di}(\OML \vee V_{\BZ }(\AOL ))={\rm Di}(\OML )\cup \AOL $.
\item ${\rm Di}(V_{\BZ }(\OML \boxplus \AOL ))={\rm Di}(\OML )\cup ((\OML \boxplus \AOL )\setminus \OML )$.\end{itemize}\end{corollary}

Recall that $\BA =\OML \cap V_{\BZ }(\AOL )$ is the unique atom of the lattice of subvarieties of $\PBZs $ \cite[Theorem 5.4.(2)]{GLP}, so, of course, the trivial variety $\T $ is the only subvariety of $\PBZs $ that has an upper bound for the lengths of the orthomodular subalgebras ${\bf S}({\bf L})$ of its members ${\bf L}$. Now let ${\bf L}$ be an arbitrary \PBZ --lattice and $\V $ be a variety of \PBZ --lattices. In \cite{pbzsums} we have denoted by $D({\bf L})=\{x\in L:x^{\sim }=0\}$ and called the members of $D({\bf L})$ {\em dense elements} of $L$. Note that $1\in D({\bf L})$ and $D({\bf L})$ is the universe of a join subsemilattice of ${\bf L}_l$. Clearly, ${\bf L}$ is an antiortholattice iff $L=D({\bf L})\cup \{0\}$, while, according to \cite[Lemma 10.(i)]{pbzsums}, ${\bf L}$ is orthomodular iff $D({\bf L})=\{1\}$. Consequently, if $\V \subseteq \OML $, then the subposet of dense elements of every member of $\V $ is a singleton, thus its length is $1$, while, if $\V \nsubseteq \OML $, then there is no upper bound for the lengths of the subposets of dense elements of the members of $\V $, since, according to \cite[Theorem 5.5]{GLP}, if $\V \nsubseteq \OML $, then the three--element antiortholattice chain ${\bf D}_3\in \V $, thus, for any non--empty set $I$, ${\bf D}_3^I\in \V $, which has the subposet $D({\bf D}_3^I)$ isomorphic to the Boolean lattice ${\bf D}_2^I\cong \mathcal{P}(I)$, which, of course, has the subchain $(\{(i]_{(I,\leq ^I)}\ |\ i\in I\},\subseteq )$ of length $|I|+1$, where $\leq ^I$ is a good order on the set $I$.

Now let us take a quick look at the varieties $\MOD $ and $\DIST $ of the modular and the distributive \PBZ --lattices, respectively. $\DIST\subsetneq V_{\BZ }(\AOL )$ by Remark \ref{j0vaol}, the immediate fact that any distributive \PBZ --lattice satisfies $\jo $ and the fact that, for instance, the antiortholattice ${\bf D}_2\oplus {\bf M}_3\oplus {\bf D}_2$ is non--distributive. The location of $\MOD $ in the lattice of subvarieties of $\PBZs $ is a bit more difficult to pinpoint, but we may notice that $\MOD $ is incomparable to $\OML $, $\SAOL $, $\WSDM $ and $V_{\BZ }(\OML \boxplus V_{\BZ }(\AOL ))$, thus also to $\SDM $, $V_{\BZ }(\AOL )$, $\OML \vee V_{\BZ }(\AOL )$, $\SDM \vee V_{\BZ }(\AOL )$ and $V_{\BZ }(\OML \boxplus \AOL )$. Indeed, for instance ${\bf D}_2^2\boxplus {\bf D}_2^3\in \OML \setminus \MOD $ and the antiortholattice ${\bf D}_2\oplus ({\bf D}_2^2\boxplus {\bf D}_4)\oplus {\bf D}_2\in \SAOL \setminus \MOD $, while ${\bf M}_3={\bf D}_2^2\boxplus {\bf D}_3\in \MOD \setminus \WSDM $ and the modular \PBZ --lattice ${\bf M}$ represented below fails $\ji $, thus ${\bf M}\in \MOD \setminus V_{\BZ }(\OML \boxplus V_{\BZ }(\AOL ))$.

\begin{center}\begin{tabular}{ccc}
\begin{picture}(60,95)(0,0)
\put(-18,83){${\bf D}_2\oplus ({\bf D}_2^2\boxplus {\bf D}_4)\oplus {\bf D}_2:$}
\put(30,0){\circle*{3}}
\put(30,20){\circle*{3}}
\put(30,30){\circle*{3}}
\put(15,35){\circle*{3}}
\put(45,35){\circle*{3}}
\put(30,40){\circle*{3}}
\put(30,50){\circle*{3}}
\put(30,70){\circle*{3}}
\put(28,-10){$0$}
\put(28,73){$1$}
\put(30,0){\line(0,1){70}}
\put(30,20){\line(-1,1){15}}
\put(30,20){\line(1,1){15}}
\put(30,50){\line(-1,-1){15}}
\put(30,50){\line(1,-1){15}}
\end{picture}
&\hspace*{25pt}
\begin{picture}(80,95)(0,0)
\put(20,75){${\bf M}:$}
\put(20,0){\circle*{3}}
\put(0,20){\circle*{3}}
\put(-8,18){$a$}
\put(20,40){\circle*{3}}
\put(13,38){$b$}
\put(40,20){\circle*{3}}
\put(41,14){$d^{\prime }$}
\put(40,40){\circle*{3}}
\put(42,37){$c$}
\put(60,40){\circle*{3}}
\put(62,34){$b^{\prime }$}
\put(40,60){\circle*{3}}
\put(32,58){$d$}
\put(80,60){\circle*{3}}
\put(82,57){$a^{\prime }$}
\put(60,80){\circle*{3}}
\put(18,-10){$0$}
\put(58,83){$1$}
\put(20,0){\line(-1,1){20}}
\put(20,0){\line(1,1){60}}
\put(0,20){\line(1,1){60}}
\put(40,20){\line(-1,1){20}}
\put(40,20){\line(0,1){40}}
\put(60,40){\line(-1,1){20}}
\put(80,60){\line(-1,1){20}}
\end{picture}
&
\begin{picture}(53,95)(0,0)
\put(0,40){\begin{tabular}{l|l}
$x$ & $x^{\sim {\bf M}}$\\ \hline 
$0$ & $1$\\ 
$a$ & $a^{\prime }$\\ 
$a^{\prime }$ & $a$\\ 
$b$ & $0$\\ 
$b^{\prime }$ & $a$\\ 
$c=c^{\prime }$ & $0$\\ 
$d$ & $0$\\ 
$d^{\prime }$ & $a$\\ 
$1$ & $0$\end{tabular}}
\end{picture}\end{tabular}\end{center}\vspace*{10pt}

We will see more results on the lattice reducts of antiortholattices, as well as of those in $\MOD $ and $\DIST $, in the following section.

\begin{lemma}{\rm \cite{bj}} If a finite algebra ${\bf A}$ belongs to a congruence--distributive variety $\V $, then ${\rm Si}(V_{\V }({\bf A}))\subseteq \H _{\V }\S _{\V }({\bf A})$.
\label{jonssonslemma}\end{lemma}

Note that, for any antiortholattice ${\bf A}$, we have: ${\bf A}\vDash \sdm $ iff $0$ is meet--irreducible in ${\bf A}_l$. Note also that, for any cardinal number $\kappa $, ${\bf D}_2^{\kappa }\oplus {\bf D}_2^{\kappa }$ and ${\bf D}_2^{\kappa }\oplus {\bf D}_2\oplus {\bf D}_2^{\kappa }$ are distributive antiortholattices, and ${\bf D}_2^0\oplus {\bf D}_2^0\cong {\bf D}_1$, ${\bf D}_2^0\oplus {\bf D}_2\oplus {\bf D}_2^0\cong {\bf D}_2$, ${\bf D}_2^1\oplus {\bf D}_2^1\cong {\bf D}_3$, ${\bf D}_2^1\oplus {\bf D}_2\oplus {\bf D}_2^1\cong {\bf D}_4$ and $V_{\BZ }({\bf D}_1)=\T \subsetneq V_{\BZ }({\bf D}_2)=\BA \subsetneq V_{\BZ }({\bf D}_3)\subsetneq V_{\BZ }({\bf D}_4)\subsetneq V_{\BZ }({\bf D}_5)\subsetneq V_{\BZ }({\bf D}_2^2\oplus {\bf D}_2^2)\subseteq \DIST \subsetneq V_{\BZ }(\AOL )$; see Section \ref{geninsaol} for the first strict inclusions here and notice that ${\bf D}_5\in \S _{\BZ }({\bf D}_2^2\oplus {\bf D}_2^2)$ and ${\bf D}_5\vDash \sdm $, while ${\bf D}_2^2\oplus {\bf D}_2^2\nvDash \sdm $.

\begin{center}\begin{tabular}{cc}
\begin{picture}(40,57)(0,0)
\put(1,55){${\bf D}_2^2\oplus {\bf D}_2^2$:}
\put(20,0){\circle*{3}}
\put(10,10){\circle*{3}}
\put(30,10){\circle*{3}}
\put(20,20){\circle*{3}}
\put(10,30){\circle*{3}}
\put(30,30){\circle*{3}}
\put(20,40){\circle*{3}}
\put(20,0){\line(1,1){10}}
\put(20,0){\line(-1,1){10}}
\put(20,40){\line(1,-1){10}}
\put(20,40){\line(-1,-1){10}}
\put(10,10){\line(1,1){20}}
\put(30,10){\line(-1,1){20}}
\put(18,-9){$0$}
\put(18,43){$1$}
\put(3,8){$a$}
\put(32,7){$b$}
\put(24,18){$c=c^{\prime }$}
\put(33,27){$a^{\prime }$}
\put(3,28){$b^{\prime }$}
\end{picture} &\hspace*{85pt}
\begin{picture}(40,57)(0,0)
\put(-65,55){${\bf D}_2^2\oplus {\bf D}_2\oplus {\bf D}_2^2$:}
\put(20,0){\circle*{3}}
\put(10,10){\circle*{3}}
\put(30,10){\circle*{3}}
\put(20,20){\circle*{3}}
\put(20,35){\circle*{3}}
\put(10,45){\circle*{3}}
\put(30,45){\circle*{3}}
\put(20,55){\circle*{3}}
\put(20,0){\line(1,1){10}}
\put(20,0){\line(-1,1){10}}
\put(20,20){\line(0,1){15}}
\put(20,35){\line(1,1){10}}
\put(20,35){\line(-1,1){10}}
\put(20,55){\line(1,-1){10}}
\put(20,55){\line(-1,-1){10}}
\put(10,10){\line(1,1){10}}
\put(30,10){\line(-1,1){10}}
\put(18,-9){$0$}
\put(18,58){$1$}
\put(3,8){$a$}
\put(32,7){$b$}
\put(33,42){$a^{\prime }$}
\put(3,43){$b^{\prime }$}
\put(23,19){$c$}
\put(23,32){$c^{\prime }$}
\end{picture}\end{tabular}\end{center}\vspace*{3pt}

\begin{lemma} Let $I$ and $J$ be sets such that $\emptyset \neq I\subsetneq J$. Then: ${\bf D}_2^I\oplus {\bf D}_2^I\in \H _{\BZ }({\bf D}_2^I\oplus {\bf D}_2\oplus {\bf D}_2^I)$ and ${\bf D}_2^I\oplus {\bf D}_2\oplus {\bf D}_2^I\in \S _{\BZ }({\bf D}_2^J\oplus {\bf D}_2^J)$.\label{distsets}\end{lemma}

\begin{proof} Clearly, ${\bf D}_2^I\oplus {\bf D}_2^I\cong _{\BZ }({\bf D}_2^I\oplus {\bf D}_2\oplus {\bf D}_2^I)/(\Delta _{D_2^I}\oplus \nabla _{D_2}\oplus \Delta _{D_2^I})$.

For all $i,j\in J$, let $\delta _{i,j}=\begin{cases}0, & i\neq j,\\ 1, & i=j,\end{cases}\in D_2$. For every $j\in J$, let $a_j=(\delta _{j,t})_{t\in J}\in D_2^J\subset D_2^J\oplus D_2^J$, so that $\{a_j:j\in J\}={\rm At}({\bf D}_2^J)={\rm At}({\bf D}_2^J\oplus {\bf D}_2^J)$. Then $\displaystyle a_I=\bigvee _{i\in I}a_i<1^{{\bf D}_2^J}$ since $J\setminus I\neq \emptyset $, so that $(a_I]$ is a proper ideal of $D_2^J$, which is clearly a Boolean lattice isomorphic to $D_2^I$. Let $h:D_2^I\rightarrow (a_I]$ be a lattice isomorphism and define $f:D_2^I\oplus D_2\oplus D_2^I\rightarrow D_2^J\oplus D_2^J$, for all $x\in D_2^I$, $f(x)=h(x)$ and $f(x^{\prime })=h(x)^{\prime }$. Clearly, $f$ is a BZ--lattice embedding of ${\bf D}_2^I\oplus {\bf D}_2\oplus {\bf D}_2^I$ into ${\bf D}_2^J\oplus {\bf D}_2^J$.\end{proof} 

\begin{lemma} For any cardinal numbers $\kappa ,\mu $ with $\kappa <\mu $:\begin{itemize}
\item ${\bf D}_2^{\kappa }\oplus {\bf D}_2^{\kappa }\in \H _{\BZ }({\bf D}_2^{\kappa }\oplus {\bf D}_2\oplus {\bf D}_2^{\kappa })$;
\item ${\bf D}_2^{\kappa }\oplus {\bf D}_2\oplus {\bf D}_2^{\kappa }\in \S _{\BZ }({\bf D}_2^{\mu }\oplus {\bf D}_2^{\mu })$;
\item if $\kappa \geq \aleph _0$, then ${\bf D}_2^{\kappa }\oplus {\bf D}_2\oplus {\bf D}_2^{\kappa }\in \S _{\BZ }({\bf D}_2^{\kappa }\oplus {\bf D}_2^{\kappa })$.\end{itemize}\label{subgendist}\end{lemma}

\begin{proof}  By Lemma \ref{distsets}.\end{proof} 

Note from the previous lemma that, for any cardinal numbers $\kappa ,\mu $ with $\kappa <\mu $, we have ${\bf D}_2^{\kappa }\oplus {\bf D}_2^{\kappa }\in \H _{\BZ }(\S _{\BZ }({\bf D}_2^{\mu }$\linebreak $\oplus {\bf D}_2^{\mu }))$ and ${\bf D}_2^{\kappa }\oplus {\bf D}_2\oplus {\bf D}_2^{\kappa }\in \S _{\BZ }(\H _{\BZ }({\bf D}_2^{\mu }\oplus {\bf D}_2\oplus {\bf D}_2^{\mu }))\subseteq \H _{\BZ }(\S _{\BZ }({\bf D}_2^{\mu }\oplus {\bf D}_2\oplus {\bf D}_2^{\mu }))$.

\begin{proposition}\begin{itemize}
\item For any cardinal numbers $\kappa ,\lambda ,\mu $ with $\kappa \leq \lambda <\mu $, we have: $V_{\BZ }({\bf D}_2^{\kappa }\oplus {\bf D}_2^{\kappa })\subseteq V_{\BZ }({\bf D}_2^{\lambda }\oplus {\bf D}_2^{\lambda })\subseteq V_{\BZ }({\bf D}_2^{\lambda }\oplus {\bf D}_2\oplus {\bf D}_2^{\lambda })\subseteq V_{\BZ }({\bf D}_2^{\mu }\oplus {\bf D}_2^{\mu })$.
\item For any cardinal number $\kappa \geq \aleph _0$, $V_{\BZ }({\bf D}_2^{\kappa }\oplus {\bf D}_2^{\kappa })=V_{\BZ }({\bf D}_2^{\kappa }\oplus {\bf D}_2\oplus {\bf D}_2^{\kappa })$.
\item For any cardinal number $\mu $, $V_{\BZ }(\{{\bf D}_2^{\kappa }\oplus {\bf D}_2^{\kappa },{\bf D}_2^{\kappa }\oplus {\bf D}_2\oplus {\bf D}_2^{\kappa }:\kappa \mbox{ a cardinal number}\})=V_{\BZ }(\{{\bf D}_2^{\kappa }\oplus {\bf D}_2^{\kappa }:\kappa \geq \mu \mbox{ a cardinal number}\})=V_{\BZ }(\{{\bf D}_2^{\kappa }\oplus {\bf D}_2\oplus {\bf D}_2^{\kappa }:\kappa \geq \mu \mbox{ a cardinal number}\})$.\end{itemize}\label{ongendist}\end{proposition}

\begin{proof}  By Lemma \ref{subgendist}.\end{proof}  

\begin{lemma} For any nonzero cardinal number $\kappa $:\begin{itemize}
\item the antiortholattice ${\bf D}_2^{\kappa }\oplus {\bf D}_2^{\kappa }$ is simple;
\item the antiortholattice ${\bf D}_2^{\kappa }\oplus {\bf D}_2\oplus {\bf D}_2^{\kappa }$ is subdirectly irreducible, having the congruence lattice isomorphic to the three--element chain, with the single nontrivial congruence $\theta _{\kappa }=\Delta _{D_2^{\kappa }}\oplus \nabla _{D_2}\oplus \Delta _{D_2^{\kappa }}$.\end{itemize}\label{cggendist}\end{lemma}

\begin{proof}  By Lemma \ref{cgordsum} and the fact that ${\bf D}_2^{\kappa }\in \BA \subseteq \OML $, so ${\bf D}_2^{\kappa }$ is congruence--regular, thus ${\rm Con}_{01}({\bf D}_2^{\kappa })=\{\Delta _{{\bf D}_2^{\kappa }}\}\cong {\bf D}_1$.\end{proof} 

\begin{proposition} For any nonzero natural number $n$, we have the strict inclusions: $V_{\BZ }({\bf D}_2^n\oplus {\bf D}_2^n)\subsetneq V_{\BZ }({\bf D}_2^n\oplus {\bf D}_2\oplus {\bf D}_2^n)\subsetneq V_{\BZ }({\bf D}_2^{n+1}\oplus {\bf D}_2^{n+1})$.\label{strictincl}\end{proposition}

\begin{proof} Let $n\in \N ^*$. By Proposition \ref{ongendist}, $V_{\BZ }({\bf D}_2^n\oplus {\bf D}_2^n)\subseteq V_{\BZ }({\bf D}_2^n\oplus {\bf D}_2\oplus {\bf D}_2^n)\subseteq V_{\BZ }({\bf D}_2^{n+1}\oplus {\bf D}_2^{n+1})$. By Lemma \ref{jonssonslemma}, ${\bf D}_2^n\oplus {\bf D}_2\oplus {\bf D}_2^n\in {\rm Si}(V_{\BZ }({\bf D}_2^n\oplus {\bf D}_2\oplus {\bf D}_2^n))\setminus {\rm Si}(V_{\BZ }({\bf D}_2^n\oplus {\bf D}_2^n))$ and ${\bf D}_2^{n+1}\oplus {\bf D}_2^{n+1}\in {\rm Si}(V_{\BZ }({\bf D}_2^{n+1}\oplus {\bf D}_2^{n+1}))\setminus {\rm Si}(V_{\BZ }({\bf D}_2^n\oplus {\bf D}_2\oplus {\bf D}_2^n))$. Hence the strict inclusions in the enunciation.\end{proof} 

\begin{remark} Lemmas \ref{jonssonslemma} and \ref{cggendist}, the strict inclusion mentioned above and Proposition \ref{strictincl} show that $V_{\BZ }({\bf D}_5)\prec V_{\BZ }({\bf D}_2^2\oplus {\bf D}_2^2)\prec V_{\BZ }({\bf D}_2^2\oplus {\bf D}_2\oplus {\bf D}_2^2)$.\end{remark}

\begin{theorem} For any $n\in \N $ and any cardinal number $\kappa >n$:\begin{itemize}
\item $V_{\BZ }({\bf D}_2^n\oplus {\bf D}_2^n)\subsetneq V_{\BZ }({\bf D}_2^n\oplus {\bf D}_2\oplus {\bf D}_2^n)\subsetneq V_{\BZ }({\bf D}_2^{\kappa }\oplus {\bf D}_2^{\kappa })\subseteq V_{\BZ }({\bf D}_2^{\kappa }\oplus {\bf D}_2\oplus {\bf D}_2^{\kappa })\subseteq \DIST $;
\item if $\kappa $ is infinite, then $V_{\BZ }({\bf D}_2^n\oplus {\bf D}_2^n)\subsetneq V_{\BZ }({\bf D}_2^n\oplus {\bf D}_2\oplus {\bf D}_2^n)\subsetneq V_{\BZ }({\bf D}_2^{\kappa }\oplus {\bf D}_2^{\kappa })=V_{\BZ }({\bf D}_2^{\kappa }\oplus {\bf D}_2\oplus {\bf D}_2^{\kappa })\subseteq \DIST $;\end{itemize}

\noindent thus we have the following infinite ascending chain of subvarieties of $\DIST $, with $n\in \N \setminus \{0,1\}$ in what follows: $\T =V_{\BZ }({\bf D}_1)=V_{\BZ }({\bf D}_2^0\oplus {\bf D}_2^0)\subsetneq \BA =V_{\BZ }({\bf D}_2)=V_{\BZ }({\bf D}_2^0\oplus {\bf D}_2\oplus {\bf D}_2^0)\subsetneq V_{\BZ }({\bf D}_3)=V_{\BZ }({\bf D}_2^1\oplus {\bf D}_2^1)\subsetneq V_{\BZ }({\bf D}_4)=V_{\BZ }({\bf D}_2^1\oplus {\bf D}_2\oplus {\bf D}_2^1)\subsetneq \ldots \subsetneq V_{\BZ }({\bf D}_2^n\oplus {\bf D}_2^n)\subsetneq V_{\BZ }({\bf D}_2^n\oplus {\bf D}_2\oplus {\bf D}_2^n)\subsetneq V_{\BZ }({\bf D}_2^{n+1}\oplus {\bf D}_2^{n+1})\subsetneq V_{\BZ }({\bf D}_2^{n+1}\oplus {\bf D}_2\oplus {\bf D}_2^{n+1})\subsetneq \ldots \subsetneq V_{\BZ }({\bf D}_2^{\aleph _0}\oplus {\bf D}_2^{\aleph _0})=V_{\BZ }({\bf D}_2^{\aleph _0}\oplus {\bf D}_2\oplus {\bf D}_2^{\aleph _0})\subseteq \DIST $.\label{notgendist}\end{theorem}

\begin{proof}  By Propositions \ref{strictincl} and \ref{ongendist}.\end{proof} 

\begin{openproblem} Determine whether $V_{\BZ }({\bf D}_2^{\kappa }\oplus {\bf D}_2^{\kappa })=V_{\BZ }({\bf D}_2^{\lambda }\oplus {\bf D}_2^{\lambda })$ for all infinite cardinal numbers $\kappa $ and $\lambda $.

Note that, since $\PBZs $ has finite type, it has at most continuum many subvarieties, so the ascending chain $V_{\BZ }({\bf D}_2^{\aleph _0}\oplus {\bf D}_2^{\aleph _0})=V_{\BZ }({\bf D}_2^{\aleph _0}\oplus {\bf D}_2\oplus {\bf D}_2^{\aleph _0})\subseteq \ldots \subseteq V_{\BZ }({\bf D}_2^{\kappa }\oplus {\bf D}_2^{\kappa })=V_{\BZ }({\bf D}_2^{\kappa }\oplus {\bf D}_2\oplus {\bf D}_2^{\kappa })\subseteq \ldots $ ($\kappa >\aleph _0$ a cardinal number) can not be strictly ascending up to any cardinality.\end{openproblem}

Finding axiomatizations for the varieties above $V_{\BZ }({\bf D}_4)$ in the preceeding ascending chain might be tricky, but we can notice the following, which also constitutes an alternate proof for half of the previous ascending chain being infinite.

For any $n\in \N \setminus \{0,1\}$, let us consider the equation:\vspace*{-5pt}\begin{flushleft}\begin{tabular}{ll}
$\cn (n)$ & $x_1^{\sim }\vee \ldots \vee x_n^{\sim }\vee (\bigwedge _{1\leq i<j\leq n}(x_i\wedge x_j)^{\sim }\wedge (x_1\vee \ldots \vee x_n))\approx $\\ 
& $x_1^{\sim }\vee \ldots \vee x_n^{\sim }\vee (\bigwedge _{1\leq i<j\leq n}(x_i\wedge x_j)^{\sim }\wedge x_1^{\prime }\wedge \ldots \wedge x
_n^{\prime })$\end{tabular}\end{flushleft}

\begin{proposition} For any $n\in \N \setminus \{0,1\}$ and any cardinal numbers $\kappa \geq n$ and $\lambda >n$:\begin{enumerate}
\item\label{eqcnd1} ${\bf D}_2^n\oplus {\bf D}_2^n\vDash \cn (n)$;
\item\label{eqcnd2} ${\bf D}_2^{\lambda }\oplus {\bf D}_2^{\lambda }\nvDash \cn (n)$ and ${\bf D}_2^{\kappa }\oplus {\bf D}_2\oplus {\bf D}_2^{\kappa }\nvDash \cn (n)$.\end{enumerate}\label{eqcnd}\end{proposition}

\begin{proof}  (\ref{eqcnd1}) The following argument is based on the observation that $a_1,\ldots ,a_n$ are the $n$ distinct atoms of the Boolean algebra ${\bf D}_2^n$ iff $a_1,\ldots ,a_n$ are nonzero and satisfy $a_i\wedge a_j=0$ for every $i,j\in [1,n]$ such that $i<j$ iff $a_1,\ldots ,a_n$ are nonzero and $\bigvee _{1\leq i<j\leq n}(a_i\wedge a_j)=0$, hence $a_1,\ldots ,a_n$ are the $n$ distinct atoms of the antiortholattice ${\bf D}_2^n$ iff $a_1^{\sim }=\ldots =a_n^{\sim }=0$ and $(\bigvee _{1\leq i<j\leq n}(a_i\wedge a_j))^{\sim }=1$ iff $a_1^{\sim }\vee \ldots \vee a_n^{\sim }=0$ and $\bigwedge _{1\leq i<j\leq n}(a_i\wedge a_j)^{\sim }=1$.

If, in $\cn (n)$, we take $x_i=0$ for some $i\in [1,n]$, then we obtain $1=1$. If, in $\cn (n)$, we replace $x_1,\ldots ,x_n$ by the $n$ atoms of ${\bf D}_2^n\oplus {\bf D}_2^n$, then both the lhs and the rhs equal the fixpoint of the Kleene complement in ${\bf D}_2^n\oplus {\bf D}_2^n$. Any other values for the variables $x_1,\ldots ,x_n$ in $\cn (n)$ produce the equality $0=0$.

\noindent (\ref{eqcnd2}) Replace $x_1,\ldots ,x_n$ in $\cn (n)$ by $n$ of the at least $n+1$ atoms of ${\bf D}_2^{\lambda }\oplus {\bf D}_2^{\lambda }$, then by $n$ of the at least $n$ atoms of ${\bf D}_2^{\kappa }\oplus {\bf D}_2\oplus {\bf D}_2^{\kappa }$.\end{proof} 

\section{The Lattice of Subvarieties of $\mathbb{SAOL}$}
\label{geninsaol}

By two observations recalled in Sections \ref{preliminaries} and \ref{aolmoddist}, respectively, the \PBZ --lattices with the $0$ meet--irreducible are exactly the antiortholattices satisfying \sdm . The members of ${\bf D}_2\oplus \KL \oplus {\bf D}_2\subset \AOL $ are exactly the \PBZ --lattices with the $0$ strictly meet--irreducible, since, clearly, for any BI--lattice ${\bf K}$, we have: ${\bf D}_2\oplus {\bf K}\oplus {\bf D}_2\in \KL $ iff ${\bf K}\in \KL $. Of course, a similar property holds for modularity or distributivity instead of the Kleene condition.

\begin{lemma} $\AOL \cap \SDM =\S _{\BZ }({\bf D}_2\oplus \KL \oplus {\bf D}_2)$.\label{embedaols}\end{lemma}

\begin{proof} Any nontrivial antiortholattice $\mathbf{L}$ in which $0^{\mathbf{L}}$ is meet--irreducible is a subalgebra of the antiortholattice $\mathbf{A}={\bf D}_2\oplus {\bf L}_{bi}\oplus {\bf D}_2$, because the map $f:L\rightarrow A$ defined by $f(0^{\mathbf{L}})=0^{\mathbf{A}}$, $f(1^{\mathbf{L}})=1^{\mathbf{A}}$ and $f(x)=x$ for all $x\in L\setminus \{0^{\mathbf{L}},1^{\mathbf{L}}\}$ is an embedding of BZ--lattices. So $\AOL \cap \SDM =\{\mathbf{L}\in \AOL :0^{\mathbf{L}}\in {\rm Mi}(\mathbf{L}_l)\}\subseteq \S _{\BZ }({\bf D}_2\oplus \KL \oplus {\bf D}_2)\subseteq \AOL \cap \SDM $ since $\AOL $ and thus $\AOL \cap \SDM $ is closed w.r.t. subalgebras.\end{proof}

Clearly, for any bounded lattices or BI--lattices ${\bf L}$ and ${\bf M}$, ${\bf D}_2\oplus {\bf L}\oplus {\bf D}_2={\bf D}_2\oplus {\bf M}\oplus {\bf D}_2$ iff ${\bf L}={\bf M}$, thus, for any classes $\C $ and $\D $ of bounded lattices or BI--lattices, ${\bf D}_2\oplus \C \oplus {\bf D}_2\subsetneq {\bf D}_2\oplus \D \oplus {\bf D}_2$ iff $\C \subsetneq \D $. Hence the poset in the following Hasse diagram is embedded (as a poset) in the Boolean lattice of the set of subsets of ${\bf D}_2\oplus \KL \oplus {\bf D}_2$. Now let us investigate the relations between the class operators applied to a class $\C $ of BI--lattices and these operators applied to the class ${\bf D}_2\oplus \C \oplus {\bf D}_2$, and, in the process, obtain an independent proof of the result from \cite{PBZ2} stating that $V_{\BZ }({\bf D}_5)=\SAOL \cap \DIST =\SDM \cap \DIST $, hence $V_{\BZ }({\bf D}_5)$ is relatively axiomatized by $\{\sdm ,\dist \}$ w.r.t. $V_{\BZ }(\AOL )$, as well as w.r.t. $\PBZs $.

\begin{center}\begin{picture}(80,42)(0,0)
\put(40,0){\circle*{3}}
\put(60,10){\circle*{3}}
\put(60,20){\circle*{3}}
\put(60,30){\circle*{3}}
\put(40,40){\circle*{3}}
\put(20,20){\circle*{3}}
\put(40,0){\line(2,1){20}}
\put(40,40){\line(2,-1){20}}
\put(60,10){\line(0,1){20}}
\put(40,0){\line(-1,1){20}}
\put(40,40){\line(-1,-1){20}}
\put(12,-10){${\bf D}_2\oplus \BA \oplus {\bf D}_2$}
\put(7,43){${\bf D}_2\oplus \KL \oplus {\bf D}_2$}
\put(63,7){${\bf D}_2\oplus \MOL \oplus {\bf D}_2$}
\put(63,17){${\bf D}_2\oplus \OML \oplus {\bf D}_2$}
\put(63,27){${\bf D}_2\oplus \OL \oplus {\bf D}_2$}
\put(-47,17){${\bf D}_2\oplus \KA \oplus {\bf D}_2$}
\end{picture}\end{center}

\begin{lemma}{\rm \cite[Lemma 3.3.(1)]{PBZ2}} All subdirectly irreducible members of $V_{\BZ }(\AOL )$ are antiortholattices.\label{thesiareaols}\end{lemma}

\begin{remark} By Lemma \ref{thesiareaols}, every subvariety $\V $ of $V_{\BZ }(\AOL )$ satisfies $Si(\V )=Si(\V \cap V_{\BZ }(\AOL ))=Si(\V )\cap Si(V_{\BZ }(\AOL ))=Si(\V )\cap Si(\AOL )=Si(\V \cap \AOL )$ and thus $\V $ is generated by the (subdirectly irreducible) antiortholattices it contains.\label{whensiareaols}\end{remark}

\begin{lemma} For any subvariety $\V $ of $\SAOL $:\begin{itemize}
\item ${\rm Si}(\V )\subseteq \T \cup \S _{\BZ }({\bf D}_2\oplus {\rm Si}(\V )_{BI}\oplus {\bf D}_2)\subseteq \T \cup \S _{\BZ }({\bf D}_2\oplus \V _{BI}\oplus {\bf D}_2)$;
\item $\V \subseteq V_{\BZ }({\bf D}_2\oplus {\rm Si}(\V )_{BI}\oplus {\bf D}_2)\subseteq V_{\BZ }({\bf D}_2\oplus \V _{BI}\oplus {\bf D}_2)$.\end{itemize}\label{varvsbired}\end{lemma}

\begin{proof} By Lemmas \ref{embedaols} and \ref{thesiareaols}, any ${\bf A}\in {\rm Si}(\V )$ satisfies ${\bf A}\in \S _{\BZ }({\bf D}_2\oplus {\bf A}_{bi}\oplus {\bf D}_2)\subseteq \S _{\BZ }({\bf D}_2\oplus {\rm Si}(\V )_{BI}\oplus {\bf D}_2)$, hence ${\rm Si}(\V )\subseteq \S _{\BZ }({\bf D}_2\oplus {\rm Si}(\V )_{BI}\oplus {\bf D}_2)\subseteq \S _{\BZ }({\bf D}_2\oplus \V _{BI}\oplus {\bf D}_2)$, therefore $\V =V_{\BZ }({\rm Si}(\V ))\subseteq V_{\BZ }({\bf D}_2\oplus {\rm Si}(\V )_{BI}\oplus {\bf D}_2)\subseteq V_{\BZ }({\bf D}_2\oplus \V _{BI}\oplus {\bf D}_2)$.\end{proof} 

\begin{remark} In particular, ${\rm Si}(\SAOL )\subseteq \T \cup \S _{\BZ }({\bf D}_2\oplus \SAOL _{BI}\oplus {\bf D}_2)\subseteq \T \cup \S _{\BZ }({\bf D}_2\oplus \KL \oplus {\bf D}_2)\subseteq \SAOL $, thus ${\rm Si}(\SAOL )=\T \cup {\rm Si}(\S _{\BZ }(({\bf D}_2\oplus \KL \oplus {\bf D}_2)\cap \V ))$.\label{sisaol}\end{remark}

By Remark \ref{sibired}, for any class $\D $ of \PBZ --lattices, ${\rm Si}(\D _{BI})\subseteq {\rm Si}(\D )_{BI}$, thus ${\bf D}_2\oplus {\rm Si}(\D _{BI})\oplus {\bf D}_2\subseteq {\bf D}_2\oplus {\rm Si}(\D )_{BI}\oplus {\bf D}_2$, hence $V_{\BZ }({\bf D}_2\oplus {\rm Si}(\D  _{BI})\oplus {\bf D}_2)\subseteq V_{\BZ }({\bf D}_2\oplus {\rm Si}(\D ) _{BI}\oplus {\bf D}_2)$.

\begin{theorem} $\SAOL =V_{\BZ }({\bf D}_2\oplus {\rm Si}(\SAOL )_{BI}\oplus {\bf D}_2)=V_{\BZ }({\bf D}_2\oplus \SAOL _{BI}\oplus {\bf D}_2)=V_{\BZ }({\bf D}_2\oplus \KL \oplus {\bf D}_2)$.\label{saol}\end{theorem}

\begin{proof} By Lemma \ref{varvsbired}, $\SAOL \subseteq V_{\BZ }({\bf D}_2\oplus \SAOL _{BI}\oplus {\bf D}_2)\subseteq V_{\BZ }({\bf D}_2\oplus \KL \oplus {\bf D}_2)\subseteq \SAOL $, hence the equalities in the enunciation.\end{proof} 

For the last equality above, we could also have used directly Remark \ref{whensiareaols}, which shows that $\mathbb{SAOL}$ is ge\-ne\-ra\-ted by the (subdirectly irreducible) antiortholattices with the $0$ meet--irreducible and hence by ${\bf D}_2\oplus \KL \oplus {\bf D}_2$ according to Lemma \ref{embedaols}.

If a variety $\V $ of \PBZ --lattices is such that $\V =V_{\BZ }({\bf D}_2\oplus \V _{BI}\oplus {\bf D}_2)$, then, since ${\bf D}_3={\bf D}_2\oplus {\bf D}_1\oplus {\bf D}_2\in {\bf D}_2\oplus \V _{BI}\oplus {\bf D}_2\subseteq {\bf D}_2\oplus \KL \oplus {\bf D}_2$, it follows that $\V =V_{\BZ }({\bf D}_2\oplus \V _{BI}\oplus {\bf D}_2)\subseteq V_{\BZ }({\bf D}_2\oplus \KL \oplus {\bf D}_2)=\SAOL $ by Theorem \ref{saol}, and ${\bf D}_3\in \V $, thus ${\bf D}_3\in \V _{BI}$, thus ${\bf D}_5={\bf D}_2\oplus {\bf D}_3\oplus {\bf D}_2\in {\bf D}_2\oplus \V _{BI}\oplus {\bf D}_2$, hence ${\bf D}_5\in \V $, therefore $\V _{\BZ }({\bf D}_5)\subseteq \V \subseteq \SAOL $.

\begin{lemma} For any subvariety $\V $ of $\SAOL $, ${\rm Si}(\V )=\T \cup {\rm Si}(\S _{\BZ }(({\bf D}_2\oplus \KL \oplus {\bf D}_2)\cap \V ))$.\label{sivsubsaol}\end{lemma}

\begin{proof} By Remark \ref{sisaol}, ${\rm Si}(\V )={\rm Si}(\SAOL \cap \V )={\rm Si}(\SAOL )\cap \V =(\T \cup {\rm Si}(\S _{\BZ }({\bf D}_2\oplus \KL \oplus {\bf D}_2)))\cap \V =\T \cup ({\rm Si}(\S _{\BZ }({\bf D}_2\oplus \KL \oplus {\bf D}_2))\cap \V )=\T \cup {\rm Si}(\S _{\BZ }({\bf D}_2\oplus \KL \oplus {\bf D}_2)\cap \S _{\BZ }(\V ))=\T \cup {\rm Si}(\S _{\BZ }(({\bf D}_2\oplus \KL \oplus {\bf D}_2)\cap \V ))$.\end{proof}

Notice that, for any $k,n\in \N ^*$ with $k\leq n$, we have ${\bf D}_k\in \S _{\BI }\H _{\BI }({\bf D}_n)\subseteq V_{\BI }({\bf D}_n)$ and ${\bf D}_k\in \S _{\BZ }\H _{\BZ }({\bf D}_n)\subseteq V_{\BZ }({\bf D}_n)$; more precisely ${\bf D}_k$ is a quotient of ${\bf D}_n$ if $k$ is odd and $n$ is even, and ${\bf D}_k$ is a subalgebra of ${\bf D}_n$ in all the other cases.

\begin{lemma}\begin{enumerate}
\item\label{skaols123} $\I _{\BZ }(\{{\bf D}_1,{\bf D}_2,{\bf D}_3\})=\{{\bf L}\in \AOL :{\bf L}\vDash \sk \}=V_{\BZ }({\bf D}_3)\cap \AOL ={\rm Si}(V_{\BZ }({\bf D}_3))$.
\item\label{skaols4} $V_{\BZ }({\bf D}_3)\subsetneq V_{\BZ }({\bf D}_4)$.\end{enumerate}\label{skaols}\end{lemma}

\begin{proof}  (\ref{skaols123}) For any antiortholattice ${\bf L}$ and any $a\in L$, clearly ${\bf L}\vDash _{\{a\},\{0\}}\sk $ and ${\bf L}\vDash _{\{1\},\{a\}}\sk $, hence ${\bf L}\vDash \sk $ iff ${\bf L}\vDash _{L\setminus \{1\},L\setminus \{0\}}\sk $ iff $x\leq y$ for all $x\in L\setminus \{1\}$ and all $y\in L\setminus \{0\}$ iff $|L\setminus \{0,1\}|\leq 1$ iff $|L|\leq 3$ iff ${\bf L}\in \I _{\BZ }(\{{\bf D}_1,{\bf D}_2,{\bf D}_3\})$.

Clearly ${\bf D}_1,{\bf D}_2,{\bf D}_3\in V_{\BZ }({\bf D}_3)$, thus $\I _{\BZ }(\{{\bf D}_1,{\bf D}_2,{\bf D}_3\})\subseteq V_{\BZ }({\bf D}_3)\cap \AOL $. By the above, $V_{\BZ }({\bf D}_3)\vDash \sk $ and thus, if an antiortholattice ${\bf L}$ belongs to $V_{\BZ }({\bf D}_3)$, then ${\bf L}\vDash \sk $, hence ${\bf L}\in \I _{\BZ }(\{{\bf D}_1,{\bf D}_2,{\bf D}_3\})$.

The antiortholattices ${\bf D}_1$, ${\bf D}_2$ and ${\bf D}_3$ are simple, thus subdirectly irreducible, hence $\I _{\BZ }(\{{\bf D}_1,{\bf D}_2,{\bf D}_3\})\subseteq {\rm Si}(V_{\BZ }({\bf D}_3))\subseteq V_{\BZ }({\bf D}_3)\cap \AOL =\I _{\BZ }(\{{\bf D}_1,{\bf D}_2,{\bf D}_3\})$ by Lemma \ref{thesiareaols} and the above, therefore ${\rm Si}(V_{\BZ }({\bf D}_3))=\I _{\BZ }(\{{\bf D}_1,{\bf D}_2,{\bf D}_3\})$.

\noindent (\ref{skaols4}) Of course, $V_{\BZ }({\bf D}_3)\subseteq V_{\BZ }({\bf D}_4)$ since ${\bf D}_3\in \H _{\BZ }({\bf D}_4)$. If we assume that ${\bf D}_4\in V_{\BZ }({\bf D}_3)$, then, since ${\bf D}_4$ is an antiortholattice, (\ref{skaols123}) gives us the contradiction ${\bf D}_4\in V_{\BZ }({\bf D}_3)\cap \AOL =\I _{\BZ }(\{{\bf D}_1,{\bf D}_2,{\bf D}_3\})$. Therefore $V_{\BZ }({\bf D}_3)\subsetneq V_{\BZ }({\bf D}_4)$.\end{proof} 

Note that part of Lemma \ref{skaols} can be obtained by applying Lemma \ref{jonssonslemma}.

\begin{proposition} $V_{\BZ }({\bf D}_3)$ is relatively axiomatized by $\sk $ w.r.t. $V_{\BZ }(\AOL )$, thus by $\{\jo ,\sk \}$ w.r.t. $\PBZs $.\end{proposition}

\begin{proof}  By Lemma \ref{skaols}.(\ref{skaols123}) and Remark \ref{j0vaol}.\end{proof} 

Note that, for any antiortholattice ${\bf L}$, any proper congruence of ${\bf L}$ has the classes of $0$ and $1$ singletons and any lattice congruence of ${\bf L}$ that preserves the Kleene complement and has the classes of $0$ and $1$ singletons also preserves the Brouwer complement of ${\bf L}$, that is: ${\rm Con}_{\BZ }({\bf L})={\rm Con}_{\BZ 01}({\bf L})\cup \{\nabla _L\}={\rm Con}_{\BI 01}({\bf L})\cup \{\nabla _L\}$ \cite{pbzsums}. If we now take a look at the congruences of the ordinal sums constructed in Section \ref{preliminaries}, we may notice that:

\begin{lemma}{\rm \cite{pbzsums}} For any bounded lattice ${\bf M}$ and any BI--lattice ${\bf K}$, if we denote by ${\bf L}={\bf M}\oplus {\bf K}\oplus {\bf M}^d$, then:\begin{itemize}
\item the BI--lattice ${\bf L}$ has ${\rm Con}_{\BI }({\bf L})=\{\alpha \oplus \beta \oplus \alpha ^{\prime }:\alpha \in {\rm Con}({\bf M}),\beta \in {\rm Con}_{\BI }({\bf K})\}\cong {\rm Con}({\bf M})\times {\rm Con}_{\BI }({\bf K})$,
\item if ${\bf M}$ is nontrivial and ${\bf K}$ is a pseudo--Kleene algebra, so that ${\bf L}$ is an antiortholattice, then ${\rm Con}_{\BZ }({\bf L})={\rm Con}_{\BI 01}({\bf L})\cup \{\nabla _L\}=\{\alpha \oplus \beta \oplus \alpha ^{\prime }:\alpha \in {\rm Con}_{01}({\bf M}),\beta \in {\rm Con}_{\BI }({\bf K})\}\cup \{\nabla _L\}\cong ({\rm Con}_{01}({\bf M})\times {\rm Con}_{\BI }({\bf K}))\oplus {\bf D}_2$,\end{itemize}

\noindent where $\alpha ^{\prime }=\{(a^{\prime },b^{\prime }):(a,b)\in \alpha \}\in {\rm Con}({\bf M}^d)={\rm Con}({\bf M})$ for all $\alpha \in {\rm Con}({\bf M})$.\label{cgordsum}\end{lemma}

\begin{lemma}\begin{enumerate}
\item\label{si1} For any ${\bf L}\in \BI $ and any ${\bf K}\in \KL $, ${\rm Con}_{\BI }({\bf D}_2\oplus {\bf L}\oplus {\bf D}_2)=\{\Delta _{{\bf D}_2}\oplus \theta \oplus \Delta _{{\bf D}_2},\nabla _{{\bf D}_2}\oplus \theta \oplus \nabla _{{\bf D}_2}:\theta \in {\rm Con}_{\BI }({\bf L})\}\cong {\bf D}_2\times {\rm Con}_{\BI }({\bf L})$ and ${\rm Con}_{\BZ }({\bf D}_2\oplus {\bf K}\oplus {\bf D}_2)=\{\Delta _{D_2}\oplus \beta \oplus \Delta _{D_2}:\beta \in {\rm Con}_{\BI }({\bf K})\}\cup \{\nabla _{D_2\oplus K\oplus D_2}\}\cong {\rm Con}_{\BI }({\bf K})\oplus \Delta _{D_2}$.
\item\label{si2} For any class $\D \subseteq \KL $, we have, in the variety $\BZ $: ${\rm Si}({\bf D}_2\oplus \D \oplus {\bf D}_2)={\bf D}_2\oplus {\rm Si}(\D )\oplus {\bf D}_2$.\end{enumerate}\label{si}\end{lemma}

\begin{proof} (\ref{si1}) By Lemma \ref{cgordsum}.

\noindent (\ref{si2}) By (\ref{si1}), the antiortholattice ${\bf D}_2\oplus {\bf K}\oplus {\bf D}_2$ is subdirectly irreducible iff the pseudo--Kleene algebra ${\bf K}$ is subdirectly irreducible.\end{proof} 

\begin{lemma}\begin{enumerate}
\item\label{lclsop1} If $I$ is a non--empty set, then, for any families $({\bf L}_i)_{i\in I}\subseteq \BI $ and $({\bf K}_i)_{i\in I}\subseteq \KL $, we have: $\displaystyle {\bf D}_2\oplus (\prod _{i\in I}{\bf L}_i)\oplus {\bf D}_2\in \S _{\BI }(\prod _{i\in I}({\bf D}_2\oplus {\bf L}_i\oplus {\bf D}_2))$ and $\displaystyle {\bf D}_2\oplus (\prod _{i\in I}{\bf K}_i)\oplus {\bf D}_2\in \S _{\BZ }(\prod _{i\in I}({\bf D}_2\oplus {\bf K}_i\oplus {\bf D}_2))$.
\item\label{lclsop2} If ${\bf L}\in \BI $, ${\bf K}\in \KL $, ${\bf M}\in \S _{\BI }({\bf L})$ and ${\bf N}\in \S _{\BI }({\bf K})$, then ${\bf D}_2\oplus {\bf M}\oplus {\bf D}_2\in \S _{\BI }({\bf D}_2\oplus {\bf L}\oplus {\bf D}_2)$ and ${\bf D}_2\oplus {\bf N}\oplus {\bf D}_2\in \S _{\BZ }({\bf D}_2\oplus {\bf K}\oplus {\bf D}_2)$.
\item\label{lclsop2,5} If ${\bf K}\in \OL $, then $\S _{\BI }({\bf D}_2\oplus {\bf K}\oplus {\bf D}_2)=\{{\bf D}_2\}\cup ({\bf D}_2\oplus \S _{\BI }({\bf K})\oplus {\bf D}_2)$ and $\S _{\BZ }({\bf D}_2\oplus {\bf K}\oplus {\bf D}_2)=\{{\bf D}_2\}\cup ({\bf D}_2\oplus \S _{\BI }({\bf K})\oplus {\bf D}_2)$.
\item\label{lclsop3} If ${\bf L}\in \BI $, ${\bf K}\in \KL $, $\theta \in {\rm Con}_{\BI }({\bf L})$ and $\zeta \in {\rm Con}_{\BI }({\bf K})$, then ${\bf D}_2\oplus {\bf L}/\theta \oplus {\bf D}_2\cong _{\BI }({\bf D}_2\oplus {\bf L}\oplus {\bf D}_2)/(\Delta _{D_2}\oplus \theta \oplus \Delta _{D_2})$, ${\bf L}/\theta \cong _{\BI }({\bf D}_2\oplus {\bf L}\oplus {\bf D}_2)/(\nabla _{{\bf D}_2}\oplus \theta \oplus \nabla _{{\bf D}_2})$ and ${\bf D}_2\oplus {\bf K}/\zeta \oplus {\bf D}_2\cong _{\BZ }({\bf D}_2\oplus {\bf K}\oplus {\bf D}_2)/(\Delta _{D_2}\oplus \zeta \oplus \Delta _{D_2})$.\end{enumerate}\label{lclsop}\end{lemma}

\begin{proof} (\ref{lclsop1}) The map from $\displaystyle D_2\oplus (\prod _{i\in I}L_i)\oplus D_2$ to $\displaystyle \prod _{i\in I}(D_2\oplus L_i\oplus D_2)$, respectively $\displaystyle D_2\oplus (\prod _{i\in I}K_i)\oplus D_2$ to $\displaystyle \prod _{i\in I}(D_2\oplus K_i\oplus D_2)$, that preserves the $0$ and $1$ and restricts to the set inclusion on $\displaystyle \prod _{i\in I}L_i$, respectively $\displaystyle \prod _{i\in I}K_i$, is an embedding of BI--lattices, respectively BZ--lattices.

\noindent (\ref{lclsop2}) Clearly, the map from ${\bf D}_2\oplus {\bf M}\oplus {\bf D}_2$ to ${\bf D}_2\oplus {\bf L}\oplus {\bf D}_2$, respectively ${\bf D}_2\oplus {\bf N}\oplus {\bf D}_2$ to ${\bf D}_2\oplus {\bf K}\oplus {\bf D}_2$, that preserves the $0$ and $1$ and restricts to a BI--lattice embedding of ${\bf M}$ into ${\bf L}$, respectively of ${\bf N}$ into ${\bf K}$, is an embedding of BI--lattices, respectively BZ--lattices.

\noindent (\ref{lclsop2,5}) Clearly, ${\bf D}_2\in \S _{\BI }({\bf D}_2\oplus {\bf K}\oplus {\bf D}_2)$ and ${\bf D}_2\in \S _{\BZ }({\bf D}_2\oplus {\bf K}\oplus {\bf D}_2)$ and, by (\ref{lclsop2}), ${\bf D}_2\oplus \S _{\BI }({\bf K})\oplus {\bf D}_2\subseteq \S _{\BI }({\bf D}_2\oplus {\bf K}\oplus {\bf D}_2)$ and ${\bf D}_2\oplus \S _{\BI }({\bf K})\oplus {\bf D}_2\subseteq \S _{\BZ }({\bf D}_2\oplus {\bf K}\oplus {\bf D}_2)$. Now, if ${\bf A}\in \S _{\BI }({\bf D}_2\oplus {\bf K}\oplus {\bf D}_2)\setminus \I _{\BI }({\bf D}_2)$ or ${\bf A}\in \S _{\BZ }({\bf D}_2\oplus {\bf K}\oplus {\bf D}_2)\setminus \I _{\BZ }({\bf D}_2)$, then there exists an $a\in A\setminus \{0,1\}\subseteq K$, and for each such $a$ we have $a^{\prime }\in A\setminus \{0,1\}\subseteq K$ as well, so that $a\vee a^{\prime }=1^{\bf K}$ and $a\wedge a^{\prime }=0^{\bf K}$. Of course, for all $x,y\in A\setminus \{0,1\}\subseteq K$, we have $x\vee y,x\wedge y\in A\setminus \{0,1\}\subseteq K$. Hence $A\setminus \{0,1\}\in \S _{\BI }({\bf K})$, therefore ${\bf A}\in {\bf D}_2\oplus \S _{\BI }({\bf K})\oplus {\bf D}_2$.

\noindent (\ref{lclsop3}) Clear, with Lemma \ref{si}.(\ref{si1}) ensuring that $\Delta _{D_2}\oplus \zeta \oplus \Delta _{D_2}\in {\rm Con}_{\BZ }({\bf D}_2\oplus {\bf K}\oplus {\bf D}_2)$.\end{proof} 

\begin{lemma} Let $\C \subseteq \BI $ and $\D \subseteq \KL $. Then:\begin{enumerate}
\item\label{d2clsop1} ${\bf D}_2\oplus \P _{\BI }(\C )\oplus {\bf D}_2\subseteq \S _{\BI }\P _{\BI }({\bf D}_2\oplus \C \oplus {\bf D}_2)$ and ${\bf D}_2\oplus \P _{\BI }(\D )\oplus {\bf D}_2\subseteq \S _{\BZ }\P _{\BZ }({\bf D}_2\oplus \D \oplus {\bf D}_2)$;
\item\label{d2clsop1,5} if $\D \subseteq \OL $, then ${\bf D}_2\oplus \S _{\BI }(\D )\oplus {\bf D}_2=\S _{\BI }({\bf D}_2\oplus \D \oplus {\bf D}_2)\setminus \{{\bf D}_2\}$ and ${\bf D}_2\oplus \S _{\BI }(\D )\oplus {\bf D}_2=\S _{\BZ }({\bf D}_2\oplus \D \oplus {\bf D}_2)\setminus \{{\bf D}_2\}$;
\item\label{d2clsop2} ${\bf D}_2\oplus \S _{\BI }\P _{\BI }(\C )\oplus {\bf D}_2\subseteq \S _{\BI }\P _{\BI }({\bf D}_2\oplus \C \oplus {\bf D}_2)$ and ${\bf D}_2\oplus \S _{\BI }\P _{\BI }(\D )\oplus {\bf D}_2\subseteq \S _{\BZ }\P _{\BZ }({\bf D}_2\oplus \D \oplus {\bf D}_2)$;
\item\label{d2clsop4} $\H _{\BI }({\bf D}_2\oplus \C \oplus {\bf D}_2)=({\bf D}_2\oplus \H _{\BI }(\C )\oplus {\bf D}_2)\cup \H _{\BI }(\C )\cup \I _{\BZ }({\bf D}_2)$; if $\C \nsubseteq \T $, then $\H _{\BI }({\bf D}_2\oplus \C \oplus {\bf D}_2)=({\bf D}_2\oplus \H _{\BI }(\C )\oplus {\bf D}_2)\cup \H _{\BI }(\C )$; $\H _{\BZ }({\bf D}_2\oplus \D \oplus {\bf D}_2)=({\bf D}_2\oplus \H _{\BI }(\D )\oplus {\bf D}_2)\cup \I _{\BZ }(\{{\bf D}_1,{\bf D}_2\})$;
\item\label{d2clsop3} ${\bf D}_2\oplus \H _{\BI }(\C )\oplus {\bf D}_2\subseteq \H _{\BI }({\bf D}_2\oplus \C \oplus {\bf D}_2)\setminus \I _{\BI }(\{{\bf D}_1,{\bf D}_2\})$ and ${\bf D}_2\oplus \H _{\BI }(\D )\oplus {\bf D}_2=\H _{\BZ }({\bf D}_2\oplus \D \oplus {\bf D}_2)\setminus \I _{\BZ }(\{{\bf D}_1,{\bf D}_2\})$.
\end{enumerate}\label{d2clsop}\end{lemma}

\begin{proof} Using the fact that, for any \PBZ --lattice ${\bf L}$, ${\rm Con}_{\BZ }({\bf L})$ is a sublattice of ${\rm Con}_{\BI }({\bf L})$, we get: (\ref{d2clsop1}) by Lemma \ref{lclsop}.(\ref{lclsop1}), (\ref{d2clsop2}) by (\ref{d2clsop1}) and Lemma \ref{lclsop}.(\ref{lclsop2}) and (\ref{d2clsop4}) by Lemma \ref{lclsop}.(\ref{lclsop3}) and Lemma \ref{si}.(\ref{si1}). (\ref{d2clsop3}) follows from (\ref{d2clsop4}), while (\ref{d2clsop1,5}) follows from Lemma \ref{lclsop}.(\ref{lclsop2,5}).\end{proof} 

\begin{proposition} Let $\C \subseteq \BI $ and $\D \subseteq \KL $. Then:\begin{itemize}
\item ${\bf D}_2\oplus V_{\BI }(\C )\oplus {\bf D}_2\subsetneq V_{\BI }({\bf D}_2\oplus \C \oplus {\bf D}_2)=V_{\BI }({\bf D}_2\oplus V_{\BI }(\C )\oplus {\bf D}_2)$;
\item ${\bf D}_2\oplus V_{\BI }(\D )\oplus {\bf D}_2\subsetneq V_{\BZ }({\bf D}_2\oplus \D \oplus {\bf D}_2)=V_{\BZ }({\bf D}_2\oplus V_{\BI }(\D )\oplus {\bf D}_2)$.\end{itemize}\label{theopvar}\end{proposition}

\begin{proof} The inclusions follow from Lemma \ref{d2clsop} and their strictness from the clear fact that any bounded lattice--ordered algebra with the $0$ strictly meet--irreducible is directly irreducible. These inclusions also prove the nontrivial right--to--left inclusions from the equalities between varieties.\end{proof} 

Trivially, $V_{\BZ }({\bf D}_3)=V_{\BZ }({\bf D}_2\oplus {\bf D}_1\oplus {\bf D}_2)=V_{\BZ }({\bf D}_2\oplus V_{\BI }({\bf D}_1)\oplus {\bf D}_2)=V_{\BZ }({\bf D}_2\oplus \T \oplus {\bf D}_2)$.

\begin{proposition} For any subvariety $\V $ of $\SAOL $, we have: $\V \subseteq V_{\BZ }({\bf D}_2\oplus {\rm Si}(\V _{BI})\oplus {\bf D}_2)=V_{\BZ }({\bf D}_2\oplus {\rm Si}(\V )_{BI}\oplus {\bf D}_2)=V_{\BZ }({\bf D}_2\oplus \V _{BI}\oplus {\bf D}_2)$.\label{sisubvarsaol}\end{proposition}

\begin{proof} By Lemma \ref{si}.(\ref{si2}), Proposition \ref{theopvar} and Lemma \ref{varvsbired}, ${\rm Si}({\bf D}_2\oplus \V _{BI}\oplus {\bf D}_2)={\bf D}_2\oplus {\rm Si}(\V _{BI})\oplus {\bf D}_2$, hence $V_{\BZ }({\bf D}_2\oplus \V _{BI}\oplus {\bf D}_2)\supseteq V_{\BZ }({\rm Si}({\bf D}_2\oplus \V _{BI}\oplus {\bf D}_2))=V_{\BZ }({\bf D}_2\oplus {\rm Si}(\V _{BI})\oplus {\bf D}_2)=V_{\BZ }({\bf D}_2\oplus V_{\BI }({\rm Si}(\V _{BI}))\oplus {\bf D}_2)=V_{\BZ }({\bf D}_2\oplus V_{\BI }(\V _{BI})\oplus {\bf D}_2)\supseteq V_{\BZ }({\bf D}_2\oplus \V _{BI}\oplus {\bf D}_2)\supseteq \V $, hence $\V \subseteq V_{\BZ }({\bf D}_2\oplus {\rm Si}(\V _{BI})\oplus {\bf D}_2)=V_{\BZ }({\bf D}_2\oplus {\rm Si}(\V )_{BI}\oplus {\bf D}_2)=V_{\BZ }({\bf D}_2\oplus \V _{BI}\oplus {\bf D}_2)$.\end{proof}

\begin{corollary} For any subvariety $\V $ of $\SAOL $, we have: $\V \subseteq V_{\BZ }({\bf D}_2\oplus V_{\BI }({\rm Si}(\V _{BI}))\oplus {\bf D}_2)=V_{\BZ }({\bf D}_2\oplus V_{\BI }({\rm Si}(\V )_{BI})\oplus {\bf D}_2)=V_{\BZ }({\bf D}_2\oplus V_{\BI }(\V _{BI})\oplus {\bf D}_2)=V_{\BZ }({\bf D}_2\oplus \V _{BI}\oplus {\bf D}_2)$.\label{varsisubvarsaol}\end{corollary}

\begin{proof} By Propositions \ref{sisubvarsaol} and \ref{theopvar}.\end{proof}

\begin{remark} Note from Remark \ref{varred} that, for any subvariety $\V $ of $\BZ $, $\V _{BI}\subseteq V_{\BI }(\V _{BI})=\H _{\BI }\S _{\BI }(\V _{BI})$.\end{remark}

\begin{corollary} $\SAOL =V_{\BZ }({\bf D}_2\oplus {\rm Si}(\SAOL _{BI})\oplus {\bf D}_2)=V_{\BZ }({\bf D}_2\oplus {\rm Si}(\SAOL )_{BI}\oplus {\bf D}_2)=V_{\BZ }({\bf D}_2\oplus \SAOL _{BI}\oplus {\bf D}_2)=V_{\BZ }({\bf D}_2\oplus V_{\BI }({\rm Si}(\SAOL _{BI}))\oplus {\bf D}_2)=V_{\BZ }({\bf D}_2\oplus V_{\BI }({\rm Si}(\SAOL )_{BI})\oplus {\bf D}_2)=V_{\BZ }({\bf D}_2\oplus V_{\BI }(\SAOL _{BI})\oplus {\bf D}_2)$.\label{varsisaol}\end{corollary}

\begin{proof} By Theorem \ref{saol}, Proposition \ref{sisubvarsaol} and Corollary \ref{varsisubvarsaol}.\end{proof}

\begin{lemma} Let ${\bf L}$ be a BI--lattice. Then:\begin{enumerate}
\item\label{moreinclvar} ${\bf L}\in \H _{\BI }({\bf D}_2\oplus {\bf L}\oplus {\bf D}_2)\subseteq V_{\BI }({\bf D}_2\oplus {\bf L}\oplus {\bf D}_2)$, so $V_{\BI }({\bf L})\subseteq V_{\BI }({\bf D}_2\oplus {\bf L}\oplus {\bf D}_2)$;
\item\label{varsbis1} ${\bf D}_2\oplus {\bf L}\oplus {\bf D}_2\in \S _{\BI }({\bf D}_3\times {\bf L})\subseteq V_{\BI }({\bf D}_3\times {\bf L})$, so $V_{\BI }({\bf D}_2\oplus {\bf L}\oplus {\bf D}_2)\subseteq V_{\BI }({\bf D}_3\times {\bf L})$;
\item\label{varsbis2} if ${\bf D}_3\in V_{\BI }({\bf L})$, then $V_{\BI }({\bf L})=V_{\BI }({\bf D}_2\oplus {\bf L}\oplus {\bf D}_2)$;
\item\label{d3vsol} ${\bf D}_3\notin V_{\BI }({\bf L})$ iff ${\bf L}$ is an ortholattice.\end{enumerate}\label{moreincl}\end{lemma}

\begin{proof} (\ref{moreinclvar}) ${\bf L}\cong _{\BI }({\bf D}_2\oplus {\bf L}\oplus {\bf D}_2)/(\nabla _{D_2}\oplus \Delta _L\oplus \nabla _{D_2})\in \H _{\BI }({\bf D}_2\oplus {\bf L}\oplus {\bf D}_2)$, as already noticed in \cite[Theorem~$3.2$]{GLP}.

\noindent (\ref{varsbis1}) If we denote $D_3=\{0^{{\bf D}_3},c,1^{{\bf D}_3}\}$ and by ${\bf A}={\bf D}_2\oplus {\bf L}\oplus {\bf D}_2\in \BI $, then the map $\varphi :D_2\oplus L\oplus D_2\rightarrow D_3\times L$ defined by: $\varphi (0^{\bf A})=(0^{{\bf D}_3},0^{\bf L})$, $\varphi (1^{\bf A})=(1^{{\bf D}_3},1^{\bf L})$ and $\varphi (x)=(c,x)$ for all $x\in L=A\setminus \{0^{\bf A},1^{\bf A}\}$ is a BI--lattice embedding of ${\bf A}$ into ${\bf D}_3\times {\bf L}$, thus ${\bf A}\in \S _{\BI }({\bf D}_3\times {\bf L})$.

\noindent (\ref{varsbis2}) By (\ref{varsbis1}).

\noindent (\ref{d3vsol}) Since ${\bf D}_3\notin \OL $, we have the right--to--left implication.

If ${\bf D}_3\notin V_{\BI }({\bf L})$ and hence, by the fact that ${\bf D}_3\in \H _{\BI }({\bf D}_4)\subseteq V_{\BI }({\bf D}_4)$, it follows that also ${\bf D}_4\notin V_{\BI }({\bf L})$, then:

$\bullet $ ${\bf D}_3\notin \S _{\BI }({\bf L})$, so there exists no $x\in L$ with $x=x^{\prime }$,

$\bullet $ ${\bf D}_4\notin \S _{\BI }({\bf L})$, so there exists no $x\in L\setminus \{0\}$ with $x<x^{\prime }$,

\noindent hence there exists no $x\in L\setminus \{0\}$ with $x\leq x^{\prime }$. But, for every $u\in L$, we have $u\wedge u^{\prime }\leq u\vee u^{\prime }=(u\wedge u^{\prime })^{\prime }$. Therefore $u\wedge u^{\prime }=0$ for all $u\in L$, which means that ${\bf L}\in \OL $.\end{proof} 

\begin{lemma}{\rm \cite{kal}} $\KA =V_{\BI }({\bf D}_3)$.\label{kalman}\end{lemma}

\begin{proposition} Let $\V $ be a subvariety of $\BI $. Then:

$\bullet $ ${\bf D}_3\in \V $ iff $\KA \subseteq \V $;

$\bullet $ ${\bf D}_3\notin \V $ iff $\V \subseteq \OL $.

In particular, $(\OL ,\KA )$ is a splitting pair in the lattice of subvarieties of $\BI $, thus also in that of $\KL $.\label{olkasplit}\end{proposition}

\begin{proof} Lemma \ref{kalman} proves the first equivalence. In the second equivalence, the left--to--right implication is trivial and the converse follows from Lemma \ref{moreincl}.(\ref{d3vsol}). Clearly, $\KA \nsubseteq \OL $, hence the splitting pair property.\end{proof} 

For any $k,n,p\in \N $ and any pair $(t,u)$, where $t(x_1,\ldots ,x_k,z_1,\ldots ,z_p)$ and $u(y_1,\ldots ,y_n,z_1,\ldots ,z_p)$ are terms in the language of $\BI $ having the arities $k+p$, respectively $n+p$ and $p$ common variables $z_1,\ldots ,z_p$, we consider the $(k+n)$--ary term $m(t,u)$ in the language of $\BZ $, defined as follows:\vspace*{-7pt}$$m(t,u)(x_1,\ldots ,x_k,y_1,\ldots ,y_n,z_1,\ldots ,z_p)=$$$$\bigvee _{i=1}^k(x_i\wedge x_i^{\prime })^{\sim }\vee \bigvee _{j=1}^n(y_j\wedge y_j^{\prime })^{\sim }\vee \bigvee _{h=1}^p(z_h\wedge z_h^{\prime })^{\sim }\vee t(x_1,\ldots ,x_k,z_1,\ldots ,z_p).$$\vspace*{-7pt}Note that:$$m(u,t)(x_1,\ldots ,x_k,y_1,\ldots ,y_n,z_1,\ldots ,z_p)=$$$$\bigvee _{i=1}^k(x_i\wedge x_i^{\prime })^{\sim }\vee \bigvee _{j=1}^n(y_j\wedge y_j^{\prime })^{\sim }\vee \bigvee _{h=1}^p(z_h\wedge z_h^{\prime })^{\sim }\vee u(y_1,\ldots ,y_n,z_1,\ldots ,z_p).$$

\begin{remark} If ${\bf L}\in \BI $ is such that ${\bf D}_3\notin V_{\BI }({\bf L})$, then ${\bf L}\in \OL $ by Lemma \ref{moreincl}.(\ref{d3vsol}), so that ${\bf L}\vDash x\wedge x^{\prime }\approx 0$ and ${\bf L}\vDash x\vee x^{\prime }\approx 1$, therefore, for any terms $t$ and $u$ in the language of $\BI $, there exist terms $r$ and $s$ in the language of $\BI $ having nonzero arities such that: ${\bf L}\vDash t\approx u$ iff ${\bf L}\vDash r\approx s$.\label{nonnullary}\end{remark}
 
\begin{lemma} Let ${\bf L}\in \BI $ and $t$ and $u$ be terms in the language of $\BI $. Then:\begin{itemize}
\item if ${\bf D}_3\in V_{\BI }({\bf L})$, then: ${\bf L}\vDash t\approx u$ iff ${\bf D}_2\oplus {\bf L}\oplus {\bf D}_2\vDash t\approx u$;
\item if $t$ and $u$ have nonzero arities and ${\bf L}\in \KL $, so that ${\bf D}_2\oplus {\bf L}\oplus {\bf D}_2\in \AOL $, then: ${\bf L}\vDash t\approx u$ iff ${\bf D}_2\oplus {\bf L}\oplus {\bf D}_2\vDash m(t,u)\approx m(u,t)$.\end{itemize}\label{eqthrclsop}\end{lemma}

\begin{proof} Lemma \ref{moreincl}.(\ref{varsbis2}) implies the first equivalence.

Now let us assume that ${\bf L}\in \KL $ and denote by ${\bf A}={\bf D}_2\oplus {\bf L}\oplus {\bf D}_2\in \AOL $. Let $k,n,p$ be as in the notation above, and assume that $k+p,n+p\in \N ^*$. Then, for any $a_1,\ldots ,a_k,b_1,\ldots ,b_n,c_1,\ldots ,c_p\in A$, we have, in ${\bf A}$:

\noindent $\bullet \ $ if $a_1,\ldots ,a_k,b_1,\ldots ,b_n,c_1,\ldots ,c_p\in L$, then $m(t,u)^{\bf A}(a_1,\ldots ,a_k,b_1,\ldots ,b_n,c_1,\ldots ,c_p)=t^{\bf L}(a_1,\ldots ,a_k,c_1,\ldots ,c_p)$ and $m(u,t)^{\bf A}(a_1,\ldots ,a_k,b_1,\ldots ,b_n,c_1,\ldots ,c_p)=u^{\bf L}(b_1,\ldots ,b_n,c_1,\ldots ,c_p)$;

\noindent $\bullet \ $ if at least one of the elements $a_1,\ldots ,a_k,b_1,\ldots ,b_n,c_1,\ldots ,c_p$ belongs to $A\setminus L=\{0^{\bf A},1^{\bf A}\}$, then $m(t,u)^{\bf A}(a_1,\ldots ,\linebreak a_k,b_1,\ldots ,b_n,c_1,\ldots ,c_p)=m(u,t)^{\bf A}(a_1,\ldots ,a_k,b_1,\ldots ,b_n,c_1,\ldots ,c_p)=1^{\bf A}$.

Therefore $m(t,u)^{\bf A}(a_1,\ldots ,a_k,b_1,\ldots ,b_n,c_1,\ldots ,c_p)=m(u,t)^{\bf A}(a_1,\ldots ,a_k,b_1,\ldots ,b_n,c_1,\ldots ,c_p)$ for all $a_1,\ldots ,\linebreak
a_k,b_1,\ldots ,b_n,c_1,\ldots ,c_p\in A$ iff $t^{\bf L}(a_1,\ldots ,a_k,c_1,\ldots ,c_p)=u^{\bf L}(b_1,\ldots ,b_n,c_1,\ldots ,c_p)$ for all $a_1,\ldots ,a_k,b_1,\ldots ,b_n,\linebreak c_1,\ldots ,c_p\in L$, that is ${\bf A}\vDash m(t,u)\approx m(u,t)$ iff ${\bf L}\vDash t\approx u$.\end{proof} 

\begin{proposition} For any subclass $\C \subseteq \BI $, we have:\begin{itemize}
\item ${\bf D}_3\in V_{\BI }(\C )$ iff $\KA \subseteq V_{\BI }(\C )$ iff $V_{\BI }(\C )=V _{\BI }({\bf D}_2\oplus \C \oplus {\bf D}_2)$;
\item ${\bf D}_3\notin V_{\BI }(\C )$ iff $V_{\BI }(\C )\subseteq \OL $ iff $V_{\BI }(\C )\subsetneq V _{\BI }({\bf D}_2\oplus \C \oplus {\bf D}_2)$.\end{itemize}\label{moreolka}\end{proposition}

\begin{proof} By Lemma \ref{moreincl}.(\ref{moreinclvar}), $\C \subseteq \H _{\BI }({\bf D}_2\oplus \C \oplus {\bf D}_2)\subseteq V_{\BI }({\bf D}_2\oplus \C \oplus {\bf D}_2)$, hence $V_{\BI }(\C )\subseteq V_{\BI }({\bf D}_2\oplus \C \oplus {\bf D}_2)$. ${\bf D}_3={\bf D}_2\oplus {\bf D}_1\oplus {\bf D}_2\in {\bf D}_2\oplus \T \oplus {\bf D}_2\subseteq {\bf D}_2\oplus \V \oplus {\bf D}_2\subset V_{\BI }({\bf D}_2\oplus \C \oplus {\bf D}_2)$, thus, if $V_{\BI }(\C )=V _{\BI }({\bf D}_2\oplus \C \oplus {\bf D}_2)$, then ${\bf D}_3\in V_{\BI }(\C )$. By Lemma \ref{eqthrclsop}, if ${\bf D}_3\in V_{\BI }(\C )$, then the relative axiomatization w.r.t. $\BI $ of $V_{\BI }(\C )$ coincides to that of $V _{\BI }({\bf D}_2\oplus \C \oplus {\bf D}_2)$, hence $V_{\BI }(\C )=V _{\BI }({\bf D}_2\oplus \C \oplus {\bf D}_2)$. Now apply  Proposition \ref{olkasplit}.\end{proof} 

\begin{corollary} For any subvariety $\V $ of $\BI $, we have:\begin{enumerate}
\item\label{andmoreolka1} ${\bf D}_3\in \V $ iff $\V =\H _{\BI }({\bf D}_2\oplus \V \oplus {\bf D}_2)$ iff $\H _{\BI }({\bf D}_2\oplus \V \oplus {\bf D}_2)$ is a variety;
\item\label{andmoreolka2} ${\bf D}_3\notin \V $ iff $\V \subsetneq \H _{\BI }({\bf D}_2\oplus \V \oplus {\bf D}_2)$ iff $\H _{\BI }({\bf D}_2\oplus \V \oplus {\bf D}_2)$ is not closed w.r.t. direct products iff $\H _{\BI }({\bf D}_2\oplus \V \oplus {\bf D}_2)\subsetneq V_{\BI }({\bf D}_2\oplus \V \oplus {\bf D}_2)$.\end{enumerate}\label{andmoreolka}\end{corollary}

\begin{proof} Proposition \ref{moreolka} and Lemma \ref{moreincl}.(\ref{moreinclvar}), which shows that $\V \subseteq \H _{\BI }({\bf D}_2\oplus \V \oplus {\bf D}_2)\subseteq V_{\BI }({\bf D}_2\oplus \V \oplus {\bf D}_2)$, prove the first equivalence  in each of (\ref{andmoreolka1}) and (\ref{andmoreolka2}), along with the right-to-left implication in the second equivalence from (\ref{andmoreolka1}). By Lemma \ref{d2clsop}.(\ref{d2clsop4}), $\H _{\BI }({\bf D}_2\oplus \V \oplus {\bf D}_2)=({\bf D}_2\oplus \V \oplus {\bf D}_2)\cup \V $. We have ${\bf D}_3\in {\bf D}_2\oplus \V \oplus {\bf D}_2\subseteq \H _{\BI }({\bf D}_2\oplus \V \oplus {\bf D}_2)$ and, for any ${\bf L}\in \V \setminus \I _{\BI }(\{{\bf D}_1,{\bf D}_2\})\subseteq \H _{\BI }({\bf D}_2\oplus \V \oplus {\bf D}_2)$, ${\bf D}_3\times {\bf L}\notin \V $ since ${\bf D}_3\notin \V $, and ${\bf D}_3\times {\bf L}\notin {\bf D}_2\oplus \V \oplus {\bf D}_2$ since ${\bf D}_3\times {\bf L}$ is directly reducible, thus ${\bf D}_3\times {\bf L}\notin \H _{\BI }({\bf D}_2\oplus \V \oplus {\bf D}_2)$. Hence the rest of the implications.\end{proof} 

\begin{proposition} For any subvariety $\V $ of $\SAOL $ such that ${\bf D}_3\in \V $, if $\D =\{{\bf K}\in \KL :{\bf D}_2\oplus {\bf K}\oplus {\bf D}_2\in \V \}$, then:\begin{enumerate}
\item\label{surjop1} ${\rm Si}(\V )=\I _{\BZ }(\{{\bf D}_1,{\bf D}_2\})\cup {\rm Si}(\S _{\BZ }({\bf D}_2\oplus \D \oplus {\bf D}_2))$;
\item\label{surjop2} if $\D \subseteq \OL $, then ${\rm Si}(\V )=\I _{\BZ }(\{{\bf D}_1,{\bf D}_2\})\cup ({\bf D}_2\oplus {\rm Si}(\S _{\BI }(\D ))\oplus {\bf D}_2)$;
\item\label{surjop3} $\D $ is a subvariety of $\KL $ and $\V =V_{\BZ }({\bf D}_2\oplus \D \oplus {\bf D}_2)=V_{\BZ }(({\bf D}_2\oplus \KL \oplus {\bf D}_2)\cap \V )$.\end{enumerate}\label{surjop}\end{proposition}

\begin{proof} (\ref{surjop1}) If $\D $ is as in the enunciation, then $\D \supseteq \T $ and ${\bf D}_2\oplus \D \oplus {\bf D}_2=({\bf D}_2\oplus \KL \oplus {\bf D}_2)\cap \V $, hence the equality in the enunciation by Lemma \ref{sivsubsaol}.

\noindent (\ref{surjop2}) By (\ref{surjop1}), using Lemma \ref{d2clsop}.(\ref{d2clsop1,5}) and Lemma \ref{si}.(\ref{si2}). 

\noindent (\ref{surjop3}) From (\ref{surjop1}), keeping in mind that ${\bf D}_2\oplus \D \oplus {\bf D}_2\subseteq \V $, we obtain that $\V =V_{\BZ }({\rm Si}(\V ))=V_{\BZ }({\rm Si}(\S _{\BZ }({\bf D}_2\oplus \D \oplus {\bf D}_2)))=V_{\BZ }(\S _{\BZ }({\bf D}_2\oplus \D \oplus {\bf D}_2))=V_{\BZ }({\bf D}_2\oplus \D \oplus {\bf D}_2)=V_{\BZ }(({\bf D}_2\oplus \KL \oplus {\bf D}_2)\cap \V )$. By Proposition \ref{theopvar}, ${\bf D}_2\oplus V_{\BI }(\D )\oplus {\bf D}_2\subset V_{\BZ }({\bf D}_2\oplus \D \oplus {\bf D}_2)=\V $, thus ${\bf D}_2\oplus V_{\BI }(\D )\oplus {\bf D}_2\subseteq ({\bf D}_2\oplus \KL \oplus {\bf D}_2)\cap \V ={\bf D}_2\oplus \D \oplus {\bf D}_2$, hence $V_{\BI }(\D )\subseteq \D $, so $V_{\BI }(\D )=\D $.\end{proof}

\begin{theorem} The operator $\V \mapsto V_{\BZ }({\bf D}_2\oplus \V \oplus {\bf D}_2)$ from the lattice of subvarieties of $\KL $ to the principal filter generated by $V_{\BZ }({\bf D}_3)$ in the lattice of subvarieties of $\SAOL $ is a lattice isomorphism whose inverse is defined by $\W \mapsto \{{\bf K}\in \KL :{\bf D}_2\oplus {\bf K}\oplus {\bf D}_2\in \W \}$.\label{thegenop}\end{theorem}

\begin{proof} Let $\nu :\Lambda (\KL )\rightarrow [V_{\BZ }({\bf D}_3),\SAOL ]_{\Lambda (\PBZs )}$ be the map in the enunciation: $\nu (\V )=V_{\BZ }({\bf D}_2\oplus \V \oplus {\bf D}_2)$ for all $\V \in \Lambda (\KL )$. Recall that $\nu (\T )=V_{\BZ }({\bf D}_3)$ and, by Theorem \ref{saol}, $\nu (\KL )=\SAOL $. 

Clearly, for any subclasses $\C $ and $\D $ of $\BI $, we have: $\C \subseteq \D $ iff ${\bf D}_2\oplus \C \oplus {\bf D}_2\subseteq {\bf D}_2\oplus \D \oplus {\bf D}_2$, which implies $V_{\BZ }({\bf D}_2\oplus \C \oplus {\bf D}_2)\subseteq V_{\BZ }({\bf D}_2\oplus \D \oplus {\bf D}_2)$, so $\nu $ is order--preserving.

Now let $\V $ and $\W $ be subvarieties of $\KL $ such that $\V \neq \W $. Then $\V \nsubseteq \W $ or $\W \nsubseteq \V $. W.l.g. we may assume that $\W \nsubseteq \V $, so that, for some terms $t$, $u$ in the language of BI--lattices, $\V \vDash t\approx u$, but $\W \nvDash t\approx u$, hence ${\bf L}\nvDash t\approx u$ for some ${\bf L}\in \W $. By Proposition \ref{olkasplit}, Remark \ref{nonnullary} and Lemma \ref{eqthrclsop}, if $\V \subseteq \OL $, then $t$ and $u$ can be chosen to have nonzero arities and hence $\nu (\V )=V_{\BZ }({\bf D}_2\oplus \V \oplus {\bf D}_2)\vDash m(t,u)\approx m(u,t)$ and $\nu (\W )=V_{\BZ }({\bf D}_2\oplus \W \oplus {\bf D}_2)\nvDash m(t,u)\approx m(u,t)$, therefore $\nu (\W )\nsubseteq \nu (\V )$. By Proposition \ref{olkasplit}, Lemma \ref{eqthrclsop} and Lemma \ref{moreincl}.(\ref{varsbis2}), if $\KA \subseteq \V $, then $\nu (\V )\vDash t\approx u$ and, since ${\bf L}\in \H _{\BI }({\bf D}_2\oplus {\bf L}\oplus {\bf D}_2)$, we have ${\bf D}_2\oplus {\bf L}\oplus {\bf D}_2\nvDash t\approx u$, thus ${\bf D}_2\oplus {\bf L}\oplus {\bf D}_2\in \nu (\W )\setminus \nu (\V )$. By Proposition \ref{olkasplit}, it follows that, whenever $\W \nsubseteq \V $, we have $\nu (\W )\nsubseteq \nu (\V )$, in particular $\nu (\W )\neq \nu (\V )$, hence $\nu $ is injective.

So $\nu $ is a poset embedding and it preserves non--inclusion, hence it also reflects order. By Proposition \ref{surjop}.(\ref{surjop3}), $\nu $ is also surjective, thus it is a lattice isomorphism, and its inverse maps each $\W \in [V_{\BZ }({\bf D}_3),\SAOL ]_{\Lambda (\PBZs )}$ to the variety $\nu ^{-1}(\W )=\{{\bf K}\in \KL :{\bf D}_2\oplus {\bf K}\oplus {\bf D}_2\in \W \}\in \Lambda (\KL )$.\end{proof}

\begin{corollary} $\Lambda (\SAOL )\cong {\bf D}_3\oplus \Lambda (\KL )$.\end{corollary}

\begin{proof} By Theorem \ref{thegenop} and the fact that the subvarieties of $\SAOL $ are $\T $, $\BA $ and those including $V_{\BZ }({\bf D}_3)$.\end{proof}

\begin{corollary} Let $(\V _j)_{j\in J}$ be a nonempty family of subvarieties of $\KL $ and $(\D _k)_{k\in K}$ be a nonempty family of subclasses of $\KL $. Then:\begin{enumerate}
\item\label{oponarbmj1} $\displaystyle V_{\BZ }({\bf D}_2\oplus (\bigcap _{j\in J}\V _j)\oplus {\bf D}_2)=\bigcap _{j\in J}V_{\BZ }({\bf D}_2\oplus \V _j\oplus {\bf D}_2)$;
\item\label{oponarbmj2} $\displaystyle V_{\BZ }({\bf D}_2\oplus (\bigvee _{j\in J}\V _j)\oplus {\bf D}_2)=\bigvee _{j\in J}V_{\BZ }({\bf D}_2\oplus \V _j\oplus {\bf D}_2)$;
\item\label{oponarbmj3} $\displaystyle V_{\BZ }({\bf D}_2\oplus (\bigcup _{k\in K}\D _k)\oplus {\bf D}_2)=\bigvee _{k\in K}V_{\BZ }({\bf D}_2\oplus \D _k\oplus {\bf D}_2)$.\end{enumerate}\label{oponarbmj}\end{corollary}

\begin{proof} (\ref{oponarbmj1}),(\ref{oponarbmj2}) By Theorem \ref{thegenop}. (\ref{oponarbmj3}) By (\ref{oponarbmj2}) and Proposition \ref{theopvar}.\end{proof}

\begin{corollary} $V_{\BI }({\rm Si}(\SAOL _{BI}))=V_{\BI }({\rm Si}(\SAOL )_{BI})=\V_{\BI }(\SAOL _{BI})=\KL $, thus any subvariety $\V $ of $\PBZs $ that includes $\SAOL $ satisfies $V_{\BI }({\rm Si}(\V _{BI}))=V_{\BI }({\rm Si}(\V )_{BI})=\V_{\BI }(\V _{BI})=\KL $, in particular\linebreak $V_{\BI }({\rm Si}(\PBZs _{BI}))=V_{\BI }({\rm Si}(\PBZs )_{BI})=\V_{\BI }(\PBZs _{BI})=\V_{\BI }(\POML )=\KL $.
\label{varredsaol}\end{corollary}

\begin{proof} By Corollary \ref{varsisaol} and Theorem \ref{thegenop}.\end{proof}

\begin{corollary} For any subvariety $\V $ of $\PBZs $ that includes $\SAOL $, $\V _{BI}$ is not a subvariety of $\KL $. In particular, $\SAOL _{BI}$ and $\PBZs _{BI}$ are not varieties.\label{birednotvar}\end{corollary}

Note that the Brouwer complement of any \PBZ --lattice is unique \cite{pbz6}, in other words any pa\-ra\-or\-tho\-mo\-du\-lar pseudo--Kleene algebra $(L,\wedge ,\vee ,\cdot ^{\prime },0,1)$ can be endowed with at most one unary operation $\cdot ^{\sim }$ such that $(L,\wedge ,\vee ,\cdot ^{\prime },\cdot ^{\sim },0,1)$ is a \PBZ --lattice, hence:

\begin{remark} For any \PBZ --lattices ${\bf A}$, ${\bf B}$ and any subclasses $\C $, $\D $ of $\PBZs $, we have: ${\bf A}={\bf B}$ iff ${\bf A}_{bi}={\bf B}_{bi}$, thus ${\bf A}\in \C $ iff ${\bf A}_{bi}\in \C _{BI}$, thus $\C \subseteq \D $ iff $\C _{BI}\subseteq \D _{BI}$ and $\C =\D $ iff $\C _{BI}=\D _{BI}$.

Thus, for instance, $\SAOL _{BI}\subsetneq \PBZs _{BI}$.\end{remark}

While any pseudo--Kleene algebra endowed with the trivial Brouwer complement becomes a BZ--lattice and, in particular, any paraorthomodular pseudo--Kleene algebra endowed with the trivial Brouwer complement becomes a paraorthomodular BZ--lattice, so $\BZ _{BI}=\KL $ and the BI--lattice reducts of the paraorthomodular BZ--lattices form the entire quasivariety $\POML $, we have:

\begin{remark} $\PBZs _{BI}\subsetneq \POML $.

Indeed, for instance, the lattice ${\bf D}_3^2$ can be endowed with two involutions, both of which are Kleene complements; out of these Kleene algebras (which are distributive, thus modular, thus paraorthomodular), the direct product of the three--element Kleene chain with itself is the BI--lattice reduct of the direct product of the three--element antiortholattice chain with itself, while the other is the paraorthomodular Kleene algebra is the following leftmost diagram, which can not be endowed with any Brouwer complement $\cdot ^{\sim }$ that would make it a \PBZ --lattice, since then the element $a$ would be sharp, thus $b\leq a$ would imply $a^{\prime }=a^{\sim }\leq b^{\sim }\leq b^{\prime }$, thus $b^{\sim }\in [a^{\prime })\cap (b^{\prime }]=\emptyset $, a contradiction, hence this BI--lattice belongs to $\KA \setminus \DIST _{BI}$, which also shows that $\DIST _{BI}$ is strictly included in the class of paraorthomodular Kleene algebras.

\begin{center}\begin{tabular}{cc}\begin{picture}(30,63)(0,0)
\put(15,0){\circle*{3}}
\put(15,60){\circle*{3}}
\put(15,30){\circle*{3}}
\put(0,15){\circle*{3}}
\put(0,45){\circle*{3}}
\put(30,15){\circle*{3}}
\put(30,45){\circle*{3}}
\put(-15,30){\circle*{3}}
\put(45,30){\circle*{3}}
\put(13,-9){$0$}
\put(13,63){$1$}
\put(18,27){$c\!=\!c^{\prime }$}
\put(-22,28){$a$}
\put(-6,44){$b^{\prime }$}
\put(32,12){$d$}
\put(-6,10){$b$}
\put(47,27){$a^{\prime}$}
\put(32,44){$d^{\prime}$}
\put(15,0){\line(-1,1){30}}
\put(15,0){\line(1,1){30}}
\put(15,60){\line(-1,-1){30}}
\put(15,60){\line(1,-1){30}}
\put(0,15){\line(1,1){30}}
\put(30,15){\line(-1,1){30}}
\end{picture}
&\hspace*{50pt}
\begin{picture}(40,63)(0,0)
\put(20,0){\circle*{3}}
\put(20,40){\circle*{3}}
\put(0,20){\circle*{3}}
\put(60,40){\circle*{3}}
\put(20,20){\circle*{3}}
\put(40,40){\circle*{3}}
\put(40,20){\circle*{3}}
\put(40,60){\circle*{3}}
\put(18,-9){$0$}
\put(38,63){$1$}
\put(63,38){$u^{\prime}$}
\put(16,42){$v^{\prime}$}
\put(42,14){$v$}
\put(-8,17){$u$}
\put(11,17){$w$}
\put(43,38){$w^{\prime}$}
\put(20,0){\line(-1,1){20}}
\put(20,0){\line(1,1){40}}
\put(40,60){\line(1,-1){20}}
\put(20,0){\line(0,1){40}}
\put(40,20){\line(0,1){40}}
\put(0,20){\line(1,1){40}}
\put(40,20){\line(-1,1){20}}
\end{picture}\end{tabular}\end{center}

We can even find lattices that can be organized as paraorthomodular pseudo--Kleene algebras, but not as \PBZ --lattices. Indeed, the modular lattice in the rightmost diagram above has two involutions, and the unique involution above up to a BI--lattice isomorphism. This involution makes it a modular, thus paraorthomodular pseudo--Kleene algebra whose set of sharp elements is $\{0,u,u^{\prime },w,w^{\prime },1\}$, which is not even a sublattice of this lattice, hence we have no Brouwer complement w.r.t. which this would be the underlying set of a subalgebra of the resulting BZ--lattice. Therefore the lattice in this rightmost diagram is not the bounded lattice reduct of any \PBZ --lattice, so it does not belong to $\MOD _L$ or $\PBZs _L$, nor does this modular pseudo--Kleene algebra belong to $\MOD _{BI}$ or $\PBZs _{BI}$, which also shows that $\MOD _{BI}$ is strictly included in the variety of modular pseudo--Kleene algebras.\label{ppkasnotredpbzl}\end{remark}

\begin{theorem}\begin{itemize}
\item $V_{\BZ }({\bf D}_4)=V_{\BZ }({\bf D}_2\oplus \BA \oplus {\bf D}_2)$.
\item $V_{\BZ }({\bf D}_5)=V_{\BZ }({\bf D}_2\oplus \KA \oplus {\bf D}_2)=\SDM \cap \DIST =\SAOL \cap \DIST $, so $V_{\BZ }({\bf D}_5)$ contains all antiortholattice chains.\end{itemize}\label{vd4d5}\end{theorem}

\begin{proof}  By Proposition \ref{theopvar}, $V_{\BZ }({\bf D}_4)=V_{\BZ }({\bf D}_2\oplus {\bf D}_2\oplus {\bf D}_2)=V_{\BZ }({\bf D}_2\oplus V_{\BI }({\bf D}_2)\oplus {\bf D}_2)=V_{\BZ }({\bf D}_2\oplus \BA \oplus {\bf D}_2)$, while $V_{\BZ }({\bf D}_5)=V_{\BZ }({\bf D}_2\oplus {\bf D}_3\oplus {\bf D}_2)=V_{\BZ }({\bf D}_2\oplus V_{\BI }({\bf D}_3)\oplus {\bf D}_2)=V_{\BZ }({\bf D}_2\oplus \KA \oplus {\bf D}_2)$ according to Lemma \ref{kalman}.

By Lemma \ref{embedaols} and the observation above it, the subdirectly irreducible members of $\SAOL \cap \DIST $ are distributive antiortholattices that satisfy $\sdm $, that is distributive antiortholattices with the $0$ meet--irreducible, and any such antiortholattice ${\bf L}$ is a subalgebra of the distributive antiortholattice ${\bf D}_2\oplus {\bf L}\oplus {\bf D}_2\in {\bf D}_2\oplus \KA \oplus {\bf D}_2$, hence $\SAOL \cap \DIST =V_{\BZ }({\bf D}_2\oplus \KA \oplus {\bf D}_2)=V_{\BZ }({\bf D}_5)$ by the above. Since all antiortholattice chains are distributive and satisfy $\sdm $, the latter statement follows.\end{proof} 

\begin{corollary}\begin{itemize}
\item $V_{\BZ }({\bf D}_1)=\T \prec V_{\BZ }({\bf D}_2)=\BA \prec V_{\BZ }({\bf D}_3)\prec V_{\BZ }({\bf D}_4)=V_{\BZ }({\bf D}_2\oplus \BA \oplus {\bf D}_2)\prec V_{\BZ }({\bf D}_5)=V_{\BZ }({\bf D}_2\oplus \KA \oplus {\bf D}_2)\subsetneq V_{\BZ }({\bf D}_2\oplus \KL \oplus {\bf D}_2)=\SAOL $.
\item $V_{\BZ }({\bf D}_4)\subsetneq V_{\BZ }({\bf D}_2\oplus \MOL \oplus {\bf D}_2)\subsetneq V_{\BZ }({\bf D}_2\oplus \OML \oplus {\bf D}_2)\subsetneq V_{\BZ }({\bf D}_2\oplus \OL \oplus {\bf D}_2)\subsetneq V_{\BZ }({\bf D}_2\oplus \KL \oplus {\bf D}_2)=\SAOL $.
\item Each of the varieties $V_{\BZ }({\bf D}_2\oplus \MOL \oplus {\bf D}_2)$, $V_{\BZ }({\bf D}_2\oplus \OML \oplus {\bf D}_2)$ and $V_{\BZ }({\bf D}_2\oplus \OL \oplus {\bf D}_2)$ is incomparable to $V_{\BZ }({\bf D}_5)$.\end{itemize}\label{thecovers}\end{corollary}

\begin{proof} By Theorems \ref{saol}, \ref{thegenop} and \ref{vd4d5}, with the cover relations following by Lemma \ref{jonssonslemma}.\end{proof} 

\begin{corollary} $(V_{\BZ }({\bf D}_2\oplus \OL \oplus {\bf D}_2),V_{\BZ }({\bf D}_5)=V_{\BZ }({\bf D}_2\oplus \KA \oplus {\bf D}_2))$ is a splitting pair in the lattice of subvarieties of $\SAOL $.\label{splitsubvarsaol}\end{corollary}

\begin{proof} By Proposition \ref{olkasplit} and Theorem \ref{thegenop}, this pair is splitting in the interval $[V_{\BZ }({\bf D}_3),\SAOL ]$ of the lattice of subvarieties of $\PBZs $. But, by \cite[Theorem~$5.5$]{GLP}, the lattice of subvarieties of $\SAOL $ is the ordinal sum of the three--element chain $\{\T ,\BA ,V_{\BZ }({\bf D}_3)\}$ with this interval, hence the statement in the enunciation.\end{proof}

Analogously, since $(\OML ,V_{\BI }({\bf B}_6))$ is a splitting pair in $\Lambda (\OL )$, where ${\bf B}_6$ is the benzene ring, $(V_{\BZ }({\bf D}_2\oplus \OML \oplus {\bf D}_2),V_{\BZ }({\bf D}_2\oplus  {\bf B}_6\oplus {\bf D}_2))$ is a splitting pair in $\Lambda (V_{\BZ }({\bf D}_2\oplus \OL \oplus {\bf D}_2))$, but not in $\Lambda (\SAOL )$, since the subvariety $V_{\BZ }({\bf D}_5)$ of $\SAOL $ is incomparable to each of the varieties $(V_{\BZ }({\bf D}_2\oplus \OML \oplus {\bf D}_2)$ and $V_{\BZ }({\bf D}_2\oplus  {\bf B}_6\oplus {\bf D}_2)$; note that, in \cite{pbz6}, there is a typo in this application of Theorem \ref{thegenop} from the current paper; also note that condition \textcircled{Q} featured in \cite[Subsection~$3.1$]{pbz6} is equational in $\OL $.

Let us note that any lattice ${\bf L}=(L,\leq ^L)$ is a sublattice of the bounded lattice $B({\bf L})=(L\cup \{0,1\},\leq ^L\cup \{(0,x),(x,1):x\in L\})$, where $0,1\notin L$ and $0\neq 1$, which, in turn, as pointed out in the proof of \cite[Lemma~$5.3$]{GLP}, is a sublattice of the lattice reduct of the antiortholattice $B({\bf L})\oplus B({\bf L})^d$, so that ${\bf L}\in \S _{\Lat }(B({\bf L})\oplus B({\bf L})^d)$. Note, also, that the antiortholattice $B({\bf L})\oplus B({\bf L})^d$ has the $0$ meet--irreducible, thus it belongs to $\SAOL $, and that it is modular iff ${\bf L}$ is modular, and distributive iff ${\bf L}$ is distributive. Hence, in addition to \cite[Corollary~$5.3$]{GLP}, we get that $V_{\Lat }(\PBZs _L)=V_{\Lat }(V_{\BZ }(\AOL )_L)=V_{\Lat }(\SAOL _L)=V_{\Lat }(\AOL _L)=V_{\Lat }((\AOL \cap \SDM )_L)=\Lat $. Also, $V_{\Lat }(\MOD _L)=V_{\Lat }((\MOD \cap V_{\BZ }(\AOL ))_L)=V_{\Lat }((\MOD \cap \SAOL )_L)=V_{\Lat }((\MOD \cap \AOL )_L)=V_{\Lat }((\MOD \cap \AOL \cap \SDM )_L)$ is the variety of modular lattices. From these facts and the second example in Remark \ref{ppkasnotredpbzl}, neither of the classes of lattices $\PBZs _L$, $V_{\BZ }(\AOL )_L$, $\SAOL _L$, $\AOL _L$, $(\AOL \cap \SDM )_L$, $\MOD _L$, $(\MOD \cap V_{\BZ }(\AOL ))_L$, $(\MOD \cap \SAOL )_L$, $(\MOD \cap \AOL )_L$, $(\MOD \cap \AOL \cap \SDM )_L$ are varieties. As expected, the variety of distributive lattices is $V_{\Lat }(\DIST )=V_{\Lat }((\DIST \cap \SAOL )_L)=V_{\Lat }((\DIST \cap \SAOL )_L)=V_{\Lat }(V_{\BZ }({\bf D}_5)_L)$, which is the variety generated by the lattice reducts of antiortholattice chains, that is those of bounded involution chains; and, of course, we have similar results for other equations in the language of lattices besides modularity and distributivity, specifically those which are satisfied by an arbitrary lattice ${\bf L}$ iff they are satisfied by the lattice $B({\bf L})\oplus B({\bf L})^d$. 

It is well known, with the cover relations being easy consequences of Lemma \ref{jonssonslemma}, that:

\begin{lemma} $\T =V_{\BI }({\bf D}_1)\prec \BA =V_{\BI }({\bf D}_2)=V_{\BI }({\bf D}_2^2)=V_{\BI }({\bf MO}_1)\prec V_{\BI }({\bf MO}_2)\prec V_{\BI }({\bf MO}_3)\prec \ldots \prec V_{\BI }({\bf MO}_n)\prec \ldots \subsetneq V_{\BI }({\bf MO}_{\aleph _0})\subsetneq \MOL $ ($n\in \N $, $n\geq 3$) is an infinite ascending chain of subvarieties of $\MOL $.\label{subvarMOL}\end{lemma}

\begin{theorem} The following is an infinite ascending chain of subvarieties of $\MOD \cap \SAOL $: $V_{\BZ }({\bf D}_3)\prec V_{\BZ }({\bf D}_4)=V_{\BZ }({\bf D}_2\oplus {\bf D}_2^2\oplus {\bf D}_2)=V_{\BZ }({\bf D}_2\oplus {\bf MO}_1\oplus {\bf D}_2)\prec V_{\BZ }({\bf D}_2\oplus {\bf MO}_2\oplus {\bf D}_2)\prec V_{\BZ }({\bf D}_2\oplus {\bf MO}_3\oplus {\bf D}_2)\prec \ldots \prec V_{\BZ }({\bf D}_2\oplus {\bf MO}_n\oplus {\bf D}_2)\prec \ldots \subsetneq V_{\BZ }({\bf D}_2\oplus {\bf MO}_{\aleph _0}\oplus {\bf D}_2)\subsetneq V_{\BZ }({\bf D}_2\oplus \MOL \oplus {\bf D}_2)$ ($n\in \N $, $n\geq 3$).\label{infmodsaol}\end{theorem}

\begin{proof} By the first equality in Theorem \ref{vd4d5}, Proposition \ref{theopvar}, Theorem \ref{thegenop} and Lemma \ref{subvarMOL}.\end{proof} 

\begin{remark} Note that $\OL =\{{\bf L}\in \BI :{\bf L}\vDash x\vee x^{\prime }\approx y\vee y^{\prime }\}=\{{\bf L}\in \BI :{\bf L}\vDash x\wedge x^{\prime }\approx y\wedge y^{\prime }\}$, and recall that $\OML =\{{\bf L}\in \BI :{\bf L}\vDash x\vee (x^{\prime }\wedge (x\vee y))\approx x\vee y\}$.\end{remark}

Let us consider the following equations in the language of BZ--lattices:\vspace*{-10pt}

\begin{flushleft}\begin{tabular}{ll}
\R & $(x\wedge x^{\prime })^{\sim }\vee (y\wedge y^{\prime })^{\sim }\vee (x\wedge x^{\prime })\approx (x\wedge x^{\prime })^{\sim }\vee (y\wedge y^{\prime })^{\sim }\vee (y\wedge y^{\prime })$\\ 
\RV & $(x\wedge x^{\prime })^{\sim }\vee (y\wedge y^{\prime })^{\sim }\vee x\vee x^{\prime }\approx (x\wedge x^{\prime })^{\sim }\vee (y\wedge y^{\prime })^{\sim }\vee y\vee y^{\prime }$\\ 
\O & $(x\wedge x^{\prime })^{\sim }\vee (y\wedge y^{\prime })^{\sim }\vee x\vee (x^{\prime }\wedge (x\vee y))\approx (x\wedge x^{\prime })^{\sim }\vee (y\wedge y^{\prime })^{\sim }\vee x\vee y$\end{tabular}\end{flushleft}

Note that:

\begin{flushleft}\begin{tabular}{lcl}
\R & coincides to & $m(x\wedge x^{\prime },y\wedge y^{\prime })\approx m(y\wedge y^{\prime },x\wedge x^{\prime })$\\ 
\RV & coincides to & $m(x\vee x^{\prime },y\vee y^{\prime })\approx m(y\vee y^{\prime },x\vee x^{\prime })$\\ 
\O & coincides to & $m(x\vee (x^{\prime }\wedge (x\vee y)),x\vee y)\approx m(x\vee y,x\vee (x^{\prime }\wedge (x\vee y)))$\end{tabular}\end{flushleft}

\begin{proposition}\begin{enumerate}
\item\label{theeqr} $\{{\bf A}\in \AOL :{\bf A}\vDash \R \}=\{{\bf A}\in \AOL :{\bf A}\vDash \RV \}=({\bf D}_2\oplus \OL \oplus {\bf D}_2)\cup \I _{\BZ }(\{{\bf D}_1,{\bf D}_2\})$.
\item\label{thesi} ${\rm Si}(V_{\BZ }({\bf D}_2\oplus \OL \oplus {\bf D}_2))={\rm Si}(({\bf D}_2\oplus \OL \oplus {\bf D}_2)\cup \I _{\BZ }(\{{\bf D}_1,{\bf D}_2\}))=({\bf D}_2\oplus {\rm Si}(\OL )\oplus {\bf D}_2)\cup \I _{\BZ }(\{{\bf D}_1,{\bf D}_2\})$.
\item\label{axordsumol} $V_{\BZ }({\bf D}_2\oplus \OL \oplus {\bf D}_2)$ is relatively axiomatized by \RV\ or, equivalently, by \R\ w.r.t. $V_{\BZ }(\AOL )$.
\end{enumerate}\label{theaxr}\end{proposition}

\begin{proof} (\ref{theeqr}) The antiortholattices ${\bf D}_1$ and ${\bf D}_2$ trivially satisfy $\R $ and, for every ${\bf L}\in \OL $, the antiortholattice ${\bf D}_2\oplus {\bf L}\oplus {\bf D}_2$ fulfills $\R $, according to Lemma \ref{eqthrclsop}. Also, ${\bf D}_4={\bf D}_2\oplus {\bf D}_2\oplus {\bf D}_2\in {\bf D}_2\oplus \OL \oplus {\bf D}_2$.  

Now let ${\bf A}\in \AOL \setminus \I _{\BZ }(\{{\bf D}_1,{\bf D}_2\})$ such that ${\bf A}\vDash \R $, $a\in A\setminus \{0,1\}=A\setminus S({\bf A})$ and $c=a\wedge a^{\prime }$. Then $c\leq a\vee a^{\prime }=c^{\prime }$ and, for all $x\in A\setminus \{0,1\}=A\setminus S({\bf A})$, $x\wedge x^{\prime }=c$ and $x\vee x^{\prime }=(x\wedge x^{\prime })^{\prime }=c^{\prime }$, in particular $c\leq x\leq c^{\prime }$, therefore the interval $[c,c^{\prime }]$ of ${\bf A}$ is an involution sublattice of ${\bf A}_{bi}$, thus a BI--lattice since it is bounded, and fulfills: $[c,c^{\prime }]=A\setminus \{0,1\}$ and, as a BI--lattice, $[c,c^{\prime }]\in \OL $. Therefore ${\bf A}={\bf D}_2\oplus [c,c^{\prime }]\oplus {\bf D}_2\in {\bf D}_2\oplus \OL \oplus {\bf D}_2$.

Similarly for $\RV $.

\noindent (\ref{thesi}) By (\ref{theeqr}), Lemma \ref{si}.(\ref{si2}) and Lemma \ref{thesiareaols} or directly from Proposition \ref{surjop}.(\ref{surjop2}).

\noindent (\ref{axordsumol}) Let $\W =V_{\BZ }({\bf D}_2\oplus \OL \oplus {\bf D}_2)$ and $\U =\{{\bf L}\in V_{\BZ }(\AOL ):{\bf L}\vDash \R \}$. By Lemma \ref{eqthrclsop}, $\W \subseteq \U $. By (\ref{thesi}), any ${\bf A}\in {\rm Si}(\U )$ belongs to $({\bf D}_2\oplus \OL \oplus {\bf D}_2)\cup \I _{\BZ }(\{{\bf D}_1,{\bf D}_2\})$, hence $\U \subseteq V_{\BZ }(({\bf D}_2\oplus \OL \oplus {\bf D}_2)\cup \I _{\BZ }(\{{\bf D}_1,{\bf D}_2\}))=V_{\BZ }(({\bf D}_2\oplus \OL \oplus {\bf D}_2)\cup \{{\bf D}_1,{\bf D}_2\})=\W $. Therefore $\W =\U $.

Similarly for $\RV $.\end{proof} 

As shown by Remark \ref{nonnullary}, the following theorem provides us with a way to relatively axiomatize any subvariety of $\SAOL $ w.r.t. $\SAOL $, thus also w.r.t. $\PBZs $.

\begin{theorem} Let $\V $ be a subvariety of $\KL $, $I$ a (not necessarily nonempty) set and, for all $i\in I$, $t_i$ and $u_i$ terms in the language of $\BI $.\begin{enumerate}
\item\label{theax1} If ${\bf D}_3\in \V $, then: $\V $ is relatively axiomatized by $\{t_i\approx u_i:i\in I\}$ w.r.t. $\KL $ iff $V_{\BZ }({\bf D}_2\oplus \V \oplus {\bf D}_2)$ is relatively axiomatized by $\{t_i\approx u_i:i\in I\}$ w.r.t. $\SAOL $.
\item\label{theax2} If ${\bf D}_3\notin \V $ and, for all $i\in I$, $t_i$ and $u_i$ have nonzero arities, then: $\V $ is relatively axiomatized by $\{t_i\approx u_i:i\in I\}$ w.r.t. $\OL $ iff $V_{\BZ }({\bf D}_2\oplus \V \oplus {\bf D}_2)$ is relatively axiomatized by $\{m(t_i,u_i)\approx m(u_i,t_i):i\in I\}$ w.r.t. $V_{\BZ }({\bf D}_2\oplus \OL \oplus {\bf D}_2)$ iff $V_{\BZ }({\bf D}_2\oplus \V \oplus {\bf D}_2)$ is relatively axiomatized by $\{\R \}\cup \{m(t_i,u_i)\approx m(u_i,t_i):i\in I\}$ or, equivalently, by $\{\RV \}\cup \{m(t_i,u_i)\approx m(u_i,t_i):i\in I\}$, w.r.t. $V_{\BZ }(\AOL )$ (in particular w.r.t. $\SAOL $).\end{enumerate}\label{theax}\end{theorem}

\begin{proof}  Let us denote by $\W =V_{\BZ }({\bf D}_2\oplus \V \oplus {\bf D}_2)\subseteq \SAOL $.

\noindent (\ref{theax1}) Assume that ${\bf D}_3\in \V $. Theorem \ref{saol} gives us the statement for $I=\emptyset $.

Now assume that $I$ is nonempty, and let us denote by $\K =\{{\bf K}\in \KL :{\bf K}\vDash \{t_i\approx u_i:i\in I\}\}$ and by $\U =\{{\bf L}\in \SAOL :{\bf L}\vDash \{t_i\approx u_i:i\in I\}\}$, so that $\U _{BI}\subseteq \K $, thus ${\bf D}_2\oplus \U _{BI}\oplus {\bf D}_2\subseteq {\bf D}_2\oplus \K \oplus {\bf D}_2$, and, by Lemma \ref{eqthrclsop}, if ${\bf D}_3\in \K $, then ${\bf D}_2\oplus \U _{BI}\oplus {\bf D}_2\subseteq {\bf D}_2\oplus \K \oplus {\bf D}_2\subseteq \U $, so that $\U =V_{\BZ }({\bf D}_2\oplus \U _{BI}\oplus {\bf D}_2)=V_{\BZ }({\bf D}_2\oplus \K \oplus {\bf D}_2)$ by Lemma \ref{varvsbired}.

Assume that $\V =\K $ and let us prove that $\W =\U $. Since ${\bf D}_3\in \V =\K $, by the above it follows that ${\bf D}_2\oplus \V \oplus {\bf D}_2={\bf D}_2\oplus \K \oplus {\bf D}_2\subseteq \U $, hence $\W =V_{\BZ }({\bf D}_2\oplus \V \oplus {\bf D}_2)\subseteq \U $. On the other hand, by Lemma \ref{embedaols}, for any ${\bf A}\in {\rm Si}(\U )$, we have ${\bf A}\in \S _{\BZ }({\bf D}_2\oplus {\bf A}_{bi}\oplus {\bf D}_2)\subseteq \W $ since ${\bf D}_2\oplus {\bf A}_{bi}\oplus {\bf D}_2\in {\bf D}_2\oplus \U _{BI}\oplus {\bf D}_2\subseteq {\bf D}_2\oplus \K \oplus {\bf D}_2={\bf D}_2\oplus \V \oplus {\bf D}_2\subseteq \W $, therefore $\U \subseteq \W $, as well.

Now assume that $\W =\U $ and let us prove that $\V =\K $. Since $V_{\BZ }({\bf D}_2\oplus \V \oplus {\bf D}_2)=\W =\U \subseteq V_{\BZ }({\bf D}_2\oplus \U _{BI}\oplus {\bf D}_2)\subseteq V_{\BZ }({\bf D}_2\oplus \K \oplus {\bf D}_2)$ by Lemma \ref{varvsbired} and the above, it follows that $\V \subseteq \K $ by Theorem \ref{thegenop}, so that ${\bf D}_3\in \V \subseteq \K $, thus ${\bf D}_3\in \K $, so, by the above, $V_{\BZ }({\bf D}_2\oplus \K \oplus {\bf D}_2)=\U =\W =V_{\BZ }({\bf D}_2\oplus \V \oplus {\bf D}_2)$, hence $\V =\K $, again by Theorem \ref{thegenop}.

\noindent (\ref{theax2}) Assume that ${\bf D}_3\notin \V $, so that $\V \subseteq \OL $ by Proposition \ref{olkasplit} and thus $\W \subseteq V_{\BZ }({\bf D}_2\oplus \OL \oplus {\bf D}_2)$. Proposition \ref{theaxr}.(\ref{axordsumol}) gives us the statement for $I=\emptyset $, as well as the last equivalence.

Now assume that $I$ is nonempty and that $t_i$ and $u_i$ are non--nullary for each $i\in I$, and let us denote by $\K =\{{\bf K}\in \OL :{\bf K}\vDash \{t_i\approx u_i:i\in I\}\}$ and by $\U =\{{\bf L}\in V_{\BZ }({\bf D}_2\oplus \OL \oplus {\bf D}_2):{\bf L}\vDash \{m(t_i,u_i)\approx m(u_i,t_i):i\in I\}\}$, so that $V_{\BZ }({\bf D}_2\oplus \K \oplus {\bf D}_2)\subseteq \U $ by Lemma \ref{eqthrclsop}.

Assume that $\V =\K $ and let us prove that $\W =\U $. By  Lemma \ref{eqthrclsop} and the fact that $\W \subseteq V_{\BZ }({\bf D}_2\oplus \OL \oplus {\bf D}_2)$, it follows that $\W \subseteq \U $. Now let ${\bf A}\in {\rm Si}(\U )\subseteq {\rm Si}(V_{\BZ }({\bf D}_2\oplus \OL \oplus {\bf D}_2))$, so that, by Proposition \ref{theaxr}.(\ref{thesi}), either ${\bf A}\in \I _{\BZ }(\{{\bf D}_1,{\bf D}_2\})\subseteq \W $ or ${\bf A}={\bf D}_2\oplus {\bf K}\oplus {\bf D}_2$ for some ${\bf K}\in \OL $ and, since ${\bf A}\in \U $, by Lemma \ref{eqthrclsop} it follows that ${\bf K}\in \K =\V $, thus ${\bf A}\in {\bf D}_2\oplus \V \oplus {\bf D}_2\subseteq \W $, hence $\U \subseteq \W $, as well.

Now assume that $\W =\U $ and let us prove that $\V =\K $. We have $V_{\BZ }({\bf D}_2\oplus \V \oplus {\bf D}_2)=\W =\U \supseteq V_{\BZ }({\bf D}_2\oplus \K \oplus {\bf D}_2)$, thus $\V \supseteq \K $ by Theorem \ref{thegenop}. But ${\bf D}_2\oplus \V \oplus {\bf D}_2\subseteq V_{\BZ }({\bf D}_2\oplus \V \oplus {\bf D}_2)=\W =\U \subseteq V_{\BZ }({\bf D}_2\oplus \OL \oplus {\bf D}_2)$, hence, by Lemma \ref{eqthrclsop} and Theorem \ref{thegenop}, $\V \subseteq \K $, as well.\end{proof} 

Let us also note, from the above, that:

\begin{proposition} For any subvariety $\V $ of $\OL $:\begin{itemize}
\item if $\V $ is axiomatized w.r.t. $\OL $ by a (not necessarily nonempty) set of equations $\{t_i\approx u_i:i\in I\}$, where $t_i$ and $u_i$ are nonnullary terms over $\BI $ for all $i\in I$, then: $\AOL \cap V_{\BZ }({\bf D}_2\oplus \V \oplus {\bf D}_2)=\{{\bf A}\in \AOL :{\bf A}\vDash \RV ,{\bf A}\vDash \{m(t_i,u_i)\approx m(u_i,t_i):i\in I\}\}=\{{\bf A}\in \AOL :{\bf A}\vDash \R ,{\bf A}\vDash \{m(t_i,u_i)\approx m(u_i,t_i):i\in I\}\}=({\bf D}_2\oplus \V \oplus {\bf D}_2)\cup \I _{\BZ }(\{{\bf D}_1,{\bf D}_2\})$.
\item ${\rm Si}(V_{\BZ }({\bf D}_2\oplus \V \oplus {\bf D}_2))={\rm Si}(({\bf D}_2\oplus \V \oplus {\bf D}_2)\cup \I _{\BZ }(\{{\bf D}_1,{\bf D}_2\}))=({\bf D}_2\oplus {\rm Si}(\V )\oplus {\bf D}_2)\cup \I _{\BZ }(\{{\bf D}_1,{\bf D}_2\})$.\end{itemize}\label{aolsitheax}\end{proposition}

And, more generally, replacing Lemma \ref{eqthrclsop} with the hypothesis in the following enunciation, we get:

\begin{theorem} Let $\V $ be a subvariety of $\OL $, $I$ a (not necessarily nonempty) set and, for all $i\in I$, $t_i$, $u_i$, $v_i$ and $w_i$ be terms in the language of $\BI $ such that, for any ${\bf L}\in \OL $, ${\bf L}\vDash \{t_i\approx u_i:i\in I\}$ iff ${\bf D}_2\oplus {\bf L}\oplus {\bf D}_2\vDash \{v_i\approx w_i:i\in I\}$. Then: $\bullet $ $\V $ is axiomatized by $\{t_i\approx u_i:i\in I\}$ w.r.t. $\OL $ iff $V_{\BZ }({\bf D}_2\oplus \V \oplus {\bf D}_2)$ is relatively axiomatized by $\{v_i\approx w_i:i\in I\}$ w.r.t. $V_{\BZ }({\bf D}_2\oplus \OL \oplus {\bf D}_2)$, thus by $\{\RV \}\cup \{v_i\approx w_i:i\in I\}$ or, equivalently, by $\{\R \}\cup \{v_i\approx w_i:i\in I\}$ w.r.t. $V_{\BZ }(\AOL )$, and, in this case:

$\bullet $ $\AOL \cap V_{\BZ }({\bf D}_2\oplus \V \oplus {\bf D}_2)=\{{\bf A}\in \AOL :{\bf A}\vDash \RV ,{\bf A}\vDash \{v_i\approx w_i:i\in I\}\}=\{{\bf A}\in \AOL :{\bf A}\vDash \R ,{\bf A}\vDash \{v_i\approx w_i:i\in I\}\}=({\bf D}_2\oplus \V \oplus {\bf D}_2)\cup \I _{\BZ }(\{{\bf D}_1,{\bf D}_2\})$.\label{moreax}\end{theorem}

\begin{corollary}\begin{enumerate}
\item\label{theotheqnsiOML} $\AOL \cap V_{\BZ }({\bf D}_2\oplus \OML \oplus {\bf D}_2)=\{{\bf A}\in \AOL :{\bf A}\vDash \{\R ,\O \}\}=\{{\bf A}\in \AOL :{\bf A}\vDash \{\RV ,\O \}\}=({\bf D}_2\oplus \OML \oplus {\bf D}_2)\cup \I _{\BZ }(\{{\bf D}_1,{\bf D}_2\})$.

\noindent ${\rm Si}(V_{\BZ }({\bf D}_2\oplus \OML \oplus {\bf D}_2))={\rm Si}(({\bf D}_2\oplus \OML \oplus {\bf D}_2)\cup \I _{\BZ }(\{{\bf D}_1,{\bf D}_2\}))=({\bf D}_2\oplus {\rm Si}(\OML )\oplus {\bf D}_2)\cup \I _{\BZ }(\{{\bf D}_1,{\bf D}_2\})$.

\noindent $V_{\BZ }({\bf D}_2\oplus \OML \oplus {\bf D}_2)$ is relatively axiomatized by \O\ w.r.t. $V_{\BZ }({\bf D}_2\oplus \OL \oplus {\bf D}_2)$, thus by $\{\R ,\linebreak \O \}$ or, equivalently, by $\{\RV ,\O \}$ w.r.t. $V_{\BZ }(\AOL )$.

\item\label{theotheqnsiMOL} $\AOL \cap V_{\BZ }({\bf D}_2\oplus \MOL \oplus {\bf D}_2)=\{{\bf A}\in \AOL :{\bf A}\vDash \{\R ,\modular \}\}=\{{\bf A}\in \AOL :{\bf A}\vDash \{\RV ,\modular \}\}\linebreak =({\bf D}_2\oplus \MOL \oplus {\bf D}_2)\cup \I _{\BZ }(\{{\bf D}_1,{\bf D}_2\})$.

\noindent ${\rm Si}(V_{\BZ }({\bf D}_2\oplus \MOL \oplus {\bf D}_2))={\rm Si}(({\bf D}_2\oplus \MOL \oplus {\bf D}_2)\cup \I _{\BZ }(\{{\bf D}_1,{\bf D}_2\}))=({\bf D}_2\oplus {\rm Si}(\MOL )\oplus {\bf D}_2)\cup \I _{\BZ }(\{{\bf D}_1,{\bf D}_2\})$.

\noindent $V_{\BZ }({\bf D}_2\oplus \MOL \oplus {\bf D}_2)=\MOD \cap V_{\BZ }({\bf D}_2\oplus \OL \oplus {\bf D}_2)$.

\item\label{theotheqnsiBA} $\AOL \cap V_{\BZ }({\bf D}_4)=\{{\bf A}\in \AOL :{\bf A}\vDash \{\R ,\dist \}\}=\{{\bf A}\in \AOL :{\bf A}\vDash \{\RV ,\dist \}\}=({\bf D}_2\oplus \BA \oplus {\bf D}_2)\cup \I _{\BZ }(\{{\bf D}_1,{\bf D}_2\})$.

\noindent ${\rm Si}(V_{\BZ }({\bf D}_4))={\rm Si}(V_{\BZ }({\bf D}_2\oplus \BA \oplus {\bf D}_2))={\rm Si}(({\bf D}_2\oplus \BA \oplus {\bf D}_2)\cup \I _{\BZ }(\{{\bf D}_1,{\bf D}_2\}))=({\bf D}_2\oplus {\rm Si}(\BA )\oplus {\bf D}_2)\cup \I _{\BZ }(\{{\bf D}_1,{\bf D}_2\})=\I _{\BZ }(\{{\bf D}_1,{\bf D}_2,{\bf D}_3,{\bf D}_4\})$.

\noindent $V_{\BZ }({\bf D}_4)=V_{\BZ }({\bf D}_2\oplus \BA \oplus {\bf D}_2)=\DIST \cap V_{\BZ }({\bf D}_2\oplus \OL \oplus {\bf D}_2)$.

\item\label{theotheqnsiKA} $V_{\BZ }({\bf D}_5)=V_{\BZ }({\bf D}_2\oplus \KA \oplus {\bf D}_2)=\DIST \cap \SAOL $.

\item\label{theotheqnsiModPKA} $V_{\BZ }({\bf D}_2\oplus \{{\bf K}\in \KL :{\bf K}\vDash \modular \}\oplus {\bf D}_2)=\MOD \cap \SAOL $.\end{enumerate}\label{theotheqnsi}\end{corollary}

\begin{proof} (\ref{theotheqnsiOML}) By Theorem \ref{theax}.(\ref{theax2}) and Proposition \ref{aolsitheax}.

\noindent (\ref{theotheqnsiMOL}), respectively (\ref{theotheqnsiBA}) By Theorem \ref{moreax} with $I$ a singleton and the equation $t_i\approx u_i$ coinciding with $v_i\approx w_i$ and coinciding with $\modular $, respectively $\dist $.

\noindent (\ref{theotheqnsiKA}) and (\ref{theotheqnsiModPKA}) By Theorem \ref{theax}.(\ref{theax1}) and Proposition \ref{aolsitheax}.\end{proof}

Now let us get back to the BI--lattice reducts of the members of varieties of \PBZ --lattices.

\begin{corollary} Let $\W $ be a subvariety of $\SAOL $.\begin{enumerate}
\item\label{genbired1} If $\W $ contains ${\bf D}_5$, then: $V_{\BI }(\W _{BI})=\{{\bf K}\in \KL :{\bf D}_2\oplus {\bf K}\oplus {\bf D}_2\in \W \}$ and $\W =V_{\BZ }({\bf D}_2\oplus \W _{BI}\oplus {\bf D}_2)=V_{\BZ }({\bf D}_2\oplus V_{\BI }(\W _{BI})\oplus {\bf D}_2)$.
\item\label{genbired2} If $\W $ does not contain ${\bf D}_5$, then $\W \subsetneq V_{\BZ }({\bf D}_2\oplus \W _{BI}\oplus {\bf D}_2)$.\end{enumerate}\label{genbired}\end{corollary}

\begin{proof} (\ref{genbired1}) Let us denote by $\V =\{{\bf K}\in \KL :{\bf D}_2\oplus {\bf K}\oplus {\bf D}_2\in \W \}$. Then ${\bf D}_3\in \V $ and, by Theorem \ref{thegenop}, $\V $ is a subvariety of $\KL $ such that $\W =V_{\BZ }({\bf D}_2\oplus \V \oplus {\bf D}_2)$. Thus, by Theorem \ref{theax}.(\ref{theax1}), if a set $\Gamma $ of equations over $\BI $ axiomatizes $\V $ w.r.t. $\KL $, then $\Gamma $ also axiomatizes $\W $ w.r.t. $\SAOL $, thus $\W _{BI}\vDash \Gamma $, so $\V _{\BI }(\W _{BI})\vDash \Gamma $, hence $\V _{\BI }(\W _{BI})\subseteq \V $. But then, by Corollary \ref{varsisubvarsaol}, it follows that $\W \subseteq V_{\BZ }({\bf D}_2\oplus \W _{BI}\oplus {\bf D}_2)=V_{\BZ }({\bf D}_2\oplus V_{\BI }(\W _{BI})\oplus {\bf D}_2)\subseteq V_{\BZ }({\bf D}_2\oplus \V \oplus {\bf D}_2)=\W $, so $\W =V_{\BZ }({\bf D}_2\oplus \W _{BI}\oplus {\bf D}_2)=V_{\BZ }({\bf D}_2\oplus V_{\BI }(\W _{BI})\oplus {\bf D}_2)=V_{\BZ }({\bf D}_2\oplus \V \oplus {\bf D}_2)$, from which, again by Theorem \ref{thegenop}, it follows that $V_{\BI }(\W _{BI})=\V $.

\noindent (\ref{genbired2}) We have $\T \subsetneq V_{\BZ }({\bf D}_2\oplus \T \oplus {\bf D}_2)=V_{\BZ }({\bf D}_3)$ and $\BA \subsetneq V_{\BZ }({\bf D}_2\oplus \BA \oplus {\bf D}_2)=V_{\BZ }({\bf D}_4)$ by Proposition \ref{theopvar}, Theorem \ref{vd4d5} and Corollary \ref{thecovers}. If $\W $ is a subvariety of $\SAOL $ other than $\T $ and $\BA $, then ${\bf D}_3\in \W $, thus ${\bf D}_5\in {\bf D}_2\oplus \W _{BI}\oplus {\bf D}_2$, so, if ${\bf D}_5\notin \W $, then $\W \neq V_{\BZ }({\bf D}_2\oplus \W _{BI}\oplus {\bf D}_2)$, thus $\W \subsetneq V_{\BZ }({\bf D}_2\oplus \W _{BI}\oplus {\bf D}_2)$ by Proposition \ref{sisubvarsaol}.\end{proof}

Recall from Corollary \ref{splitsubvarsaol} that the subvarieties of $\SAOL $ which do not contain ${\bf D}_5$ are exactly the subvarieties of $V_{\BZ }({\bf D}_2\oplus \OL \oplus {\bf D}_2)$, and from Theorem \ref{vd4d5} and Corollary \ref{thecovers} that the five subvarieties of $V_{\BZ }({\bf D}_5)$ are exactly the distributive subvarieties of $\SAOL $, namely the varieties generated by antiortholattice chains.

\begin{corollary} If $\W $ is a proper subvariety of $\SAOL $ which either includes $V_{\BZ }({\bf D}_5)$ or is included in $V_{\BZ }({\bf D}_5)$, then $V_{\BI }(\W _{BI})\subsetneq \KL $.\end{corollary}

\begin{proof} By Theorem \ref{thegenop}, Corollary \ref{varredsaol} and Corollary \ref{genbired}.(\ref{genbired1}) for the case when ${\bf D}_5\in \W $, in particular for $\W =V_{\BZ }({\bf D}_5)$, hence also for the case when $\W \subseteq V_{\BZ }({\bf D}_5)$.\end{proof}

\begin{corollary} For any $\V \in [\KA )_{\Lambda (\KL )}$, if $\W =V_{\BZ }({\bf D}_2\oplus \V \oplus {\bf D}_2)\in [V_{\BZ }({\bf D}_5),\SAOL ]_{\Lambda (\PBZs )}$, then $V_{\BI }(\W _{\BI })=\V $. In particular, $V_{\BI }(V_{\BZ }({\bf D}_5)_{\BI })=\KA $, so, for any $\U \in [V_{\BZ }({\bf D}_5)=\SAOL \cap \DIST ,\DIST ]_{\Lambda (\PBZs )}$, $V_{\BI }(\U _{\BI })=V_{\BI }(\DIST _{\BI })=\KA $ and $\U _{\BI }$ is not a variety.\end{corollary}

\begin{proof} By Proposition \ref{moreolka}, $\V =V_{\BI }({\bf D}_2\oplus \V \oplus {\bf D}_2)\supseteq V_{\BZ }({\bf D}_2\oplus \V \oplus {\bf D}_2)_{BI}=\W _{BI}\supseteq {\bf D}_2\oplus \V \oplus {\bf D}_2$, thus $V_{\BI }(\W _{\BI })=\V $, hence, for any $\U \in [\SAOL \cap \DIST ,\DIST ]_{\Lambda (\PBZs )}$, $V_{\BI }(\U _{\BI })=\KA $, therefore $\U _{\BI }$ is not a variety, by Remark \ref{ppkasnotredpbzl}.\end{proof}

Now let us take another look at the lattice of subvarieties of $\PBZs $ and in particular that of $\OML \vee V_{\BZ }(\AOL )$.

If $\s $, $\V $ and $\W $ are subvarieties of a variety $\U $ of similar algebras, then, as pointed out in {\rm \cite[Subsection $3.3$]{pbz6}}:\begin{enumerate}
\item\label{interssdirprod1} $\V \vee \W =\V \times _s\W $ iff ${\rm Si}(\V \vee \W )={\rm Si}(\V )\cup {\rm Si}(\W )$ iff ${\rm Si}(\V \vee \W )\subseteq {\rm Si}(\V )\cup {\rm Si}(\W )$; 
\item\label{interssdirprod2} if $\V \vee \W =\V \times _s\W $, then $\s \cap (\V \vee \W )=(\s \cap \V )\vee (\s \cap \W )=(\s \cap \V )\times _s(\s \cap \W )$.\end{enumerate}

An obvious consequence of (\ref{interssdirprod2}) is the fact that, if ${\cal L}$ is a sublattice of the lattice of subvarieties of $\U $ such that $\X \vee \Y =\X \times _s\Y $ for all members $\X $, $\Y $ of ${\cal L}$, then the lattice ${\cal L}$ is distributive.

By (\ref{interssdirprod1}) and a generalization from \cite{bj} of Lemma \ref{jonssonslemma} to an arbitrary system of generators $\G $ for a congruence--distributive variety $\X $, stating that the subdirectly irreducible members of $\X =V_{\X }(\G )$ belong to $\H _{\X }\S _{\X }(\Y )$, where $\Y $ is the class of the ultraproducts of the members of $\G $, we get \cite[Theorem $4.10$]{bj}: if $\U $ is congruence--distributive, then $\V \vee \W =\V \times _s\W $.

Also by (\ref{interssdirprod1}), if $\V \vee \W =\V \times _s\W $, in particular if $\U $ is congruence--distributive, then: $\V \vee \W =\U $ iff ${\rm Si}(\V )\cup {\rm Si}(\W )={\rm Si}(\U )$ iff ${\rm Si}(\U )\subseteq \V \cup \W $, thus: $\V \vee \W \subsetneq \U $ iff ${\rm Si}(\U )={\rm Si}(\V )\cup {\rm Si}(\W )\subsetneq {\rm Si}(\U )$ iff ${\rm Si}(\U )\nsubseteq \V \cup \W $. For instance, out of the following subdirectly irreducible \PBZ --lattices, ${\bf L}$ fails both $\wsdm $ and $\ji $, so that it belongs to ${\rm Si}(\PBZs )\setminus (\WSDM \cup V_{\BZ }(\OML \boxplus V_{\BZ }(\AOL )))$ and thus to ${\bf L}\in \PBZs \setminus (\WSDM \vee V_{\BZ }(\OML \boxplus V_{\BZ }(\AOL )))$, while ${\bf M}$ satisfies $\wsdm $ and fails $\sdm $ and $\jo $, so belongs to ${\rm Si}(\WSDM )\setminus (\SDM \cup V_{\BZ }(\AOL ))$ and thus to $\WSDM \setminus (\SDM \vee V_{\BZ }(\AOL ))$, hence $\WSDM \vee V_{\BZ }(\OML \boxplus V_{\BZ }(\AOL ))\subsetneq \PBZs $ and  $\SDM \vee V_{\BZ }(\AOL ))\subsetneq \WSDM $:

\begin{center}\begin{tabular}{ccc}\begin{picture}(30,73)(0,0)
\put(-20,53){${\bf L}$:}
\put(15,0){\circle*{3}}
\put(15,60){\circle*{3}}
\put(15,30){\circle*{3}}
\put(0,15){\circle*{3}}
\put(30,45){\circle*{3}}
\put(-15,30){\circle*{3}}
\put(45,30){\circle*{3}}
\put(13,-9){$0$}
\put(13,63){$1$}
\put(18,27){$c\!=\!c^{\prime }$}
\put(-22,28){$a$}
\put(-6,10){$b$}
\put(47,27){$a^{\prime}$}
\put(32,44){$b^{\prime}$}
\put(15,0){\line(-1,1){30}}
\put(15,0){\line(1,1){30}}
\put(15,60){\line(1,-1){30}}
\put(15,60){\line(-1,-1){30}}
\put(0,15){\line(1,1){30}}
\end{picture}
&
\begin{picture}(245,73)(0,0)
\put(29,50){\begin{tabular}{c|ccccccc}
$x$ & $0$ & $a$ & $a^{\prime }$ & $b$ & $b^{\prime }$ & $c$ & $1$\\ \hline
$x^{\sim }$ & $1$ & $a^{\prime }$ & $a$ & $a^{\prime }$ & $0$ & $0$ & $0$\end{tabular}}
\put(29,10){\begin{tabular}{c|cccccccccc}
$x$ & $0$ & $d$ & $d^{\prime}$ & $e$ & $f$ & $g$ & $e^{\prime}$ &
$f^{\prime}$ & $g^{\prime}$ & $1$\\\hline
$x^{\sim}$ & $1$ & $d^{\prime}$ & $d$ & $d$ & $d$ & $d$ & $0$ & $0$ & $0$ & $0$\end{tabular}}\end{picture}
&
\begin{picture}(60,73)(0,0)
\put(-5,50){${\bf M}$:} 
\put(28,-9){$0$} 
\put(30,0){\circle*{3}} 
\put(30,20){\circle*{3}} 
\put(22,39){$g^{\prime }$} 
\put(30,40){\circle*{3}} 
\put(40,10){\circle*{3}} 
\put(40,30){\circle*{3}} 
\put(42,26){$g$} 
\put(30,60){\circle*{3}} 
\put(40,50){\circle*{3}} 
\put(10,20){\circle*{3}} 
\put(3,18){$d$} 
\put(62,47){$d^{\prime }$} 
\put(24,19){$e$} 
\put(40,7){$f$} 
\put(41,45){$e^{\prime }$} 
\put(23,59){$f^{\prime }$} 
\put(60,50){\circle*{3}} 
\put(40,70){\circle*{3}} 
\put(38,73){$1$}
\put(40,70){\line(1,-1){20}}
\put(40,70){\line(-1,-1){10}}
\put(40,70){\line(0,-1){20}}
\put(30,40){\line(1,-1){10}}
\put(30,40){\line(0,1){20}}
\put(40,30){\line(1,1){20}}
\put(30,40){\line(-1,-1){20}}
\put(40,30){\line(0,-1){20}}
\put(30,40){\line(1,1){10}}
\put(40,30){\line(-1,-1){10}}
\put(30,0){\line(1,1){10}}
\put(30,0){\line(0,1){20}} 
\put(30,0){\line(-1,1){20}}
\end{picture}\end{tabular}\end{center}

Indeed, ${\rm Con}_{\BZ }({\bf L})=\{\Delta _L,\theta ,\nabla _L\}\cong {\bf D}_3$, where $L/\theta =\{\{0,a^{\prime}\},\{b,c,b^{\prime}\},\{a,1\}\}$, while ${\rm Con}_{\BZ }({\bf M})=\{\Delta _M,\alpha ,\beta ,$\linebreak $\nabla _M\}\cong {\bf D}_4$, where $M/\alpha =\{\{0,f\},\{e,g\},\{d\},\{d^{\prime}\},\{g^{\prime},e^{\prime}\},\{f^{\prime},1\}\}$ and $M/\beta =\{\{0,e,f,g,d^{\prime}\},\{d,e^{\prime},f^{\prime},g^{\prime},1\}\}$, so that $\alpha \subset \beta $, hence ${\bf L}$ and ${\bf M}$ are subdirectly irreducible. In ${\bf L}$ we have: $(c\wedge b^{\sim })^{\sim }=1\neq a=c^{\sim }\vee b^{\sim \sim }$, thus ${\bf L}\nvDash \wsdm $, and $b=b\wedge c$ and $(c\wedge b^{\sim })\vee (c\wedge b^{\sim \sim })=b\neq c$, hence ${\bf L}\nvDash \ji $. Similarly, we can check that ${\bf M}$ satisfies $\wsdm $, but fails $\sdm $ and $\jo $ (consider the pairs $(e,f)$, respectively $(f^{\prime},f)$).

Remembering from \cite{pbzsums} that $\WSDM \cap V_{\BZ }(\OML \boxplus \AOL )=\WSDM \cap V_{\BZ }(\OML \boxplus V_{\BZ }(\AOL ))=\OML \vee V_{\BZ }(\AOL )$ and taking into account the distributivity of $\Lambda (\PBZs )$, we obtain the sublattice $\{\OML \vee V_{\BZ }(\AOL ),\SDM \vee V_{\BZ }(\AOL ),\WSDM ,V_{\BZ }(\OML \boxplus \AOL ),\SDM \vee V_{\BZ }(\OML \boxplus \AOL ),\WSDM \vee V_{\BZ }(\OML \boxplus \AOL ),V_{\BZ }(\OML \boxplus V_{\BZ }(\AOL )),\SDM \vee V_{\BZ }(\OML \boxplus V_{\BZ }(\AOL )),\WSDM \vee V_{\BZ }(\OML \boxplus V_{\BZ }(\AOL )),\PBZs \}\cong {\bf D}_3^2\oplus {\bf D}_2$ of $\Lambda (\PBZs )$.

Now, if $\U $ is again arbitrary and we consider the maps:

\noindent $\varphi :[\V \cap \W ,\V ]_{\Lambda (\U )}\times [\V \cap \W ,\W ]_{\Lambda (\U )}\rightarrow [\V \cap \W ,\V \vee \W ]_{\Lambda (\U )}$, $\varphi (\s _1,\s _2)=\s _1\vee \s _2$ for all $\s _1\in [\V \cap \W ,\V ]_{\Lambda (\U )}$ and $\s _2\in [\V \cap \W ,\W ]_{\Lambda (\U )}$,

\noindent $\psi :[\V \cap \W ,\V \vee \W ]_{\Lambda (\U )}\rightarrow [\V \cap \W ,\V ]_{\Lambda (\U )}\times [\V \cap \W ,\W ]_{\Lambda (\U )}$, $\psi (\s )=(\s \cap \V ,\s \cap \W )$ for all $\s \in [\V \cap \W ,\V \vee \W ]_{\Lambda (\U )}$, then:

$\bullet $ $\varphi $ and $\psi $ are clearly order--preserving; moreover, $\varphi $ preserves the join and $\psi $ preserves the intersection;

$\bullet $ if $\V \vee \W =\V \times _s\W $, then $\varphi \circ \psi $ is the identity map of $[\V \cap \W ,\V \vee \W ]_{\Lambda (\U )}$, because, by (\ref{interssdirprod2}) above, for any $\s \in [\V \cap \W ,\V \vee \W ]_{\Lambda (\U )}$, so that $\s \cap \V \in [\V \cap \W ,\V ]_{\Lambda (\U )}$ and $\s \cap \W \in [\V \cap \W ,\W ]_{\Lambda (\U )}$, we have $\s =\s \cap (\V \vee \W )=(\s \cap \V )\vee (\s \cap \W )=\varphi (\psi (\s ))$;

$\bullet $ if $\mathcal{V}\subseteq [\V \cap \W ,\V ]_{\Lambda (\U )}$ and $\mathcal{W}\subseteq [\V \cap \W ,\W ]_{\Lambda (\U )}$ are such that $\s _1\vee \s _2=\s _1\times _s\s _2$ for all $\s _1\in \mathcal{V}$ and all $\s _2\in \mathcal{W}$, then $\psi \circ \varphi $ is the identity map on $\mathcal{V}\times \mathcal{W}$, so, if we also have $\V \vee \W =\V \times _s\W $, then $\varphi \mid _{\mathcal{V}\times \mathcal{W}}:\mathcal{V}\times \mathcal{W}\rightarrow \varphi (\mathcal{V}\times \mathcal{W})$ and $\psi \mid _{\varphi (\mathcal{V}\times \mathcal{W})}:\varphi (\mathcal{V}\times \mathcal{W})\rightarrow \mathcal{V}\times \mathcal{W}$ are mutually inverse order isomorphisms;

indeed, if $\s _1\in \mathcal{V}$ and $\s _2\in \mathcal{W}$ are such that $\s _1\vee \s _2=\s _1\times _s\s _2$, then $\s _1\cap \W \subseteq \V \cap \W \subseteq \s _1\cap \s _2\subseteq \s _1\cap \W $, thus $\s _1\cap \W =\V \cap \W $ and, similarly, $\s _2\cap \V =\V \cap \W $, hence, by (\ref{interssdirprod2}): $\psi (\varphi (\s _1,\s _2))=\psi (\s )=(\s \cap \V ,\s \cap \W )=((\s _1\vee \s _2)\cap \V ,(\s _1\vee \s _2)\cap \W )=((\s _1\cap \V )\vee (\s _2\cap \V ),(\s _1\cap \W )\vee (\s _2\cap \W ))=(\s _1\vee (\V\cap \W ),(\V \cap \W )\vee \s _2)=(\s _1,\s _2)$; and, by the previous statement, if all $\s _1\in \mathcal{V}$ and $\s _2\in \mathcal{W}$ are such that $\s _1\vee \s _2=\s _1\times _s\s _2$ and we also have $\V \vee \W =\V \times _s\W $, then $\varphi \mid _{\mathcal{V}\times \mathcal{W}}:\mathcal{V}\times \mathcal{W}\rightarrow \varphi (\mathcal{V}\times \mathcal{W})$ and $\psi \mid _{\varphi (\mathcal{V}\times \mathcal{W})}:\varphi (\mathcal{V}\times \mathcal{W})\rightarrow \mathcal{V}\times \mathcal{W}$ are mutually inverse bijections, hence they are order isomorphisms by the first statement above;

$\bullet $ if $\s _1\vee \s _2=\s _1\times _s\s _2$ for all $\s _1\in [\V \cap \W ,\V ]_{\Lambda (\U )}$ and all $\s _2\in [\V \cap \W ,\W ]_{\Lambda (\U )}$, in particular if $\U $ is congruence--distributive, then $\varphi $ and $\psi $ are mutually inverse lattice isomorphisms;

indeed, if $\s _1\vee \s _2=\s _1\times _s\s _2$ for all $\s _1\in [\V \cap \W ,\V ]$ and all $\s _2\in [\V \cap \W ,\W ]$, which holds if $\U $ is congruence--distributive according to \cite[Theorem $4.10$]{bj} recalled above, then, by the previous statements, $\varphi $ is surjective, thus $\varphi $ and $\psi $ are mutually inverse order isomorphisms between $[\V \cap \W ,\V ]_{\Lambda (\U )}\times [\V \cap \W ,\W ]_{\Lambda (\U )}$ and $[\V \cap \W ,\V \vee \W ]_{\Lambda (\U )}$, hence they are lattice isomorphisms;

$\bullet $ if $\V \vee \W =\V \times _s\W $, then: $[\V \cap \W ,\V \vee \W ]_{\Lambda (\U )}$ is distributive iff $\varphi $ and $\psi $ are mutually inverse lattice isomorphisms iff $\s _1\vee \s _2=\s _1\times _s\s _2$ for all $\s _1\in [\V \cap \W ,\V ]_{\Lambda (\U )}$ and $\s _2\in [\V \cap \W ,\W ]_{\Lambda (\U )}$;

indeed, the distributivity law proves the first equivalence and we have the right-to-left implication in the second equivalence from the previous statement; now, if $\V \vee \W =\V \times _s\W $ and $[\V \cap \W ,\V \vee \W ]_{\Lambda (\U )}$ is distributive, then, for any $\s _1\in [\V \cap \W ,\V ]_{\Lambda (\U )}$ and $\s _2\in [\V \cap \W ,\W ]_{\Lambda (\U )}$, we have ${\rm Si}(\s _1\vee \s _2)\subseteq {\rm Si}(\V \vee \W )={\rm Si}(\V )\cup {\rm Si}(\W )$, thus ${\rm Si}(\s _1\vee \s _2)={\rm Si}(\s _1\vee \s _2)\cap ({\rm Si}(\V )\cup {\rm Si}(\W ))={\rm Si}((\s _1\vee \s _2)\cap (\V \cup \W ))={\rm Si}(((\s _1\vee \s _2)\cap \V )\cup ((\s _1\vee \s _2)\cap \W ))={\rm Si}((\s _1\vee \s _2)\cap \V )\cup {\rm Si}((\s _1\vee \s _2)\cap \W )={\rm Si}((\s _1\cap \V )\vee (\s _2\cap \V ))\cup {\rm Si}((\s _1\cap \W )\vee (\s _2\cap \W ))={\rm Si}(\s _1\vee (\V \cap \W ))\cup {\rm Si}((\V \cap \W )\vee \s _2)={\rm Si}(\s _1)\cup {\rm Si}(\s _2)$, thus  $\s _1\vee \s _2=\s _1\times _s\s _2$.

Recall that ${\rm At}(\Lambda (\PBZs ))=\{\BA \}$, so, for any $\V \in \Lambda (\PBZs )$, $\Lambda (\V )=\{\T \}\cup [\BA ,\V ]_{\Lambda (\PBZs )}\cong {\bf D}_2\oplus [\BA ,\V ]_{\Lambda (\PBZs )}$. Since \PBZ --lattices are lattice--ordered, $\PBZs $ is congruence--distributive, thus its lattice of subvarieties is distributive, hence, as expected for its intervals, $[\BA ,\OML \vee V_{\BZ }(\AOL )]_{\Lambda (\PBZs )}\cong $\linebreak $[\BA ,\OML ]_{\Lambda (\PBZs )}\times [\BA ,V_{\BZ }(\AOL )]_{\Lambda (\PBZs )}$, with a lattice isomorphism given by $\varphi :[\BA ,\OML ]_{\Lambda (\PBZs )}\times $\linebreak $[\BA ,V_{\BZ }(\AOL )]_{\Lambda (\PBZs )}\rightarrow [\BA ,\OML \vee V_{\BZ }(\AOL )]_{\Lambda (\PBZs )}$ defined by $\varphi (\s _1,\s _2)=\s _1\vee \s _2$ for every $\s _1\in [\BA ,\OML ]$ and $\s _2\in [\BA ,V_{\BZ }(\AOL )]$.

Consequently, since the only covers of $\BA $ in $\Lambda (\PBZs )$ are $V_{\BZ }({\bf MO}_2)$ and $V_{\BZ }({\bf D}_3)$, for any $n\in \N ^*$, the only cover of $V_{\BZ }({\bf MO}_n)$ in $\Lambda (\OML )$ and thus in $[\BA ,\OML ]_{\Lambda (\PBZs )}$ is $V_{\BZ }({\bf MO}_{n+1})$ and, according to \cite[Theorem~$17$]{pbz6}, the only cover of $V_{\BZ }({\bf D}_3)$ in $\Lambda (V_{\BZ }(\AOL ))$ and thus in $[\BA ,V_{\BZ }(\AOL )]_{\Lambda (\PBZs )}$ is $V_{\BZ }({\bf D}_4)$, it follows that the only covers of $V_{\BZ }({\bf D}_3)$ in $\Lambda (\OML \vee V_{\BZ }(\AOL ))$ are $V_{\BZ }({\bf D}_4)$ and $V_{\BZ }({\bf MO}_2)\vee V_{\BZ }({\bf D}_3)=V_{\BZ }(\{{\bf MO}_2,{\bf D}_3\})$, while the only covers of $V_{\BZ }({\bf MO}_n)$ in $\Lambda (\OML \vee V_{\BZ }(\AOL ))$ are $V_{\BZ }({\bf MO}_{n+1})$ and $V_{\BZ }({\bf MO}_n)\vee V_{\BZ }({\bf D}_3)=V_{\BZ }(\{{\bf MO}_n,{\bf D}_3\})$.

If $\mathcal{C}_0$, $\mathcal{C}_1$ and $\mathcal{C}_2$ are the infinite ascending chains of varieties from Lemma \ref{subvarMOL}, Theorem \ref{infmodsaol} and Theorem \ref{notgendist}, respectively, and, for each $i\in [0,2]$ and any variety $\V $ of \PBZ --lattices, we denote by $\mathcal{C}_{i,\V }=\{\V \vee \W :\W \in \mathcal{C}_i\}$, then, for any nontrivial subvarieties $\V $, $\W $ of $V_{\BZ }(\AOL )$ and $\U $, $\Y $ of $\OML $ such that $\V \neq \W $ and $\U \neq \Y $: $\mathcal{C}_{0,\V }$ and $\mathcal{C}_{0,\W }$, respectively $\mathcal{C}_{1,\U }$ and $\mathcal{C}_{1,\Y }$, respectively $\mathcal{C}_{2,\U }$ and $\mathcal{C}_{2,\Y }$ are pairwise disjoint infinite ascending chains of subvarieties of $\MOL \vee V_{\BZ }(\AOL )$, respectively $\OML \vee (\MOD \cap \SAOL )$, respectively $\OML \vee \DIST $; moreover, $\mathcal{C}_{1,\U }$ and $\mathcal{C}_{2,\Y }$ are pairwise disjoint, while $\mathcal{C}_{1,\U }$ and $\mathcal{C}_{2,\U }$ have in common exactly the varieties $\U \vee V_{\BZ }({\bf D}_2)=\U \vee \BA =\U $, $\U \vee V_{\BZ }({\bf D}_3)$ and $\U \vee V_{\BZ }({\bf D}_4)$.

\section*{Acknowledgements}

This work was supported by the research grant {\em Propriet\` a d`Ordine Nella Semantica Algebrica delle Logiche
Non--classiche}, Universit\` a degli Studi di Cagliari,
Regione Autonoma della Sardegna, L. R. 7/2007, n. 7, 2015,\linebreak CUP:F72F16002920002, as well as the research grant number IZSEZO\_186586/1, awarded to the project {\em Re\-ti\-cu\-la\-tions of Concept Algebras} by the Swiss National Science Foundation, within the programme Scientific Exchanges.

I thank Francesco Paoli and Miroslav Plosci\v ca for insightful discussions on some of the issues tackled in this paper.

\end{document}